\documentclass[UTF-8,reqno]{amsart}
\usepackage{enumerate}
\usepackage{mhequ}
\usepackage[margin=1.2in]{geometry}
\usepackage{amssymb,url,color, booktabs,nccmath}
\usepackage{mathrsfs}
\usepackage{enumitem}
\usepackage{graphicx}
\usepackage{tikz}
\usetikzlibrary{shapes,snakes}
\usetikzlibrary{calc}
\usetikzlibrary{decorations.shapes}
\usepackage{color}
\usepackage[colorlinks=true]{hyperref}
\hypersetup{
    linkcolor=blue,          
    citecolor=red,        
    filecolor=blue,      
    urlcolor=cyan
}

\definecolor{darkergreen}{rgb}{0.0, 0.5, 0.0}

\long\def\zhu#1{{\color{red}\footnotesize Xiang:\ #1}}
\long\def\hao#1{{\color{blue}\footnotesize Hao:\ #1}}
\long\def\rong#1{{\color{darkergreen} \footnotesize Rong:\ #1}}

\numberwithin{equation}{section}
\def\theequation{\arabic{section}.\arabic{equation}}
\newcommand{\be}{\begin{eqnarray}}
\newcommand{\ee}{\end{eqnarray}}
\newcommand{\ce}{\begin{eqnarray*}}
\newcommand{\de}{\end{eqnarray*}}
\newtheorem{theorem}{Theorem}[section]
\newtheorem{lemma}[theorem]{Lemma}
\newtheorem{remark}[theorem]{Remark}
\newtheorem{definition}[theorem]{Definition}
\newtheorem{proposition}[theorem]{Proposition}
\newtheorem{Examples}[theorem]{Example}
\newtheorem{corollary}[theorem]{Corollary}

\newenvironment{nouppercase}{%
  \renewcommand{\uppercasenonmath}[1]{}}{}

\newcommand{\rmb}[1]{\textcolor{black}{#1}}

\def\Wick#1{\,\colon\!\! #1 \!\colon}

\def\PPhi{\mathbf{\Phi}}

\def\eps{\varepsilon}

\def\p{\partial}

\def\[{{\Big[}}
\def\]{{\Big]}}
\def\<{{\langle}}
\def\>{{\rangle}}
\def\({{\Big(}}
\def\){{\Big)}}

\def\bx{{\mathbf{x}}}

\def\dif{{\mathord{{\rm d}}}}

\def\no{\nonumber}
\def\={&\!\!=\!\!&}

 \newcommand{\eqdef}{\stackrel{\mbox{\tiny def}}{=}}

\def\bC{{\mathbf C}}

\def\cE{{\mathcal E}}
\def\cF{{\mathcal F}}
\def\cG{{\mathcal G}}
\def\cH{{\mathcal H}}
\def\cI{{\mathcal I}}
\def\cJ{{\mathcal J}}

\def\cS{{\mathcal S}}

\def\mN{{\mathbb N}}

\def\mR{{\mathbb R}}

\def\mT{{\mathbb T}}

\def\mZ{{\mathbb Z}}

\def\1{{\mathbf{1}}}

\def\sL{{\mathscr L}}

\def\E{\mathbf E}

\def\geq{\geqslant}
\def\leq{\leqslant}
\def\ge{\geqslant}
\def\le{\leqslant}

\def\eps{\varepsilon}

\def\p{\partial}

\def\[{{\Big[}}
\def\]{{\Big]}}
\def\<{{\langle}}
\def\>{{\rangle}}
\def\({{\Big(}}
\def\){{\Big)}}

\def\bx{{\mathbf{x}}}

\def\dif{{\mathord{{\rm d}}}}

\def\no{\nonumber}
\def\={&\!\!=\!\!&}
\def\bt{\begin{theorem}}
\def\et{\end{theorem}}
\def\bl{\begin{lemma}}
\def\el{\end{lemma}}
\def\br{\begin{remark}}
\def\er{\end{remark}}
\def\bx{\begin{Examples}}
\def\ex{\end{Examples}}
\def\bd{\begin{definition}}
\def\ed{\end{definition}}
\def\bp{\begin{proposition}}
\def\ep{\end{proposition}}
\def\bc{\begin{corollary}}
\def\ec{\end{corollary}}

\def\geq{\geqslant}
\def\leq{\leqslant}
\def\ge{\geqslant}
\def\le{\leqslant}

 \def\R{\mathbb R}
 \def\R{\mathbb R}

\def\<{\langle} \def\>{\rangle}

\allowdisplaybreaks

\tikzset{
        dot/.style={circle,fill=black,inner sep=0pt, outer sep=0.7pt, minimum size=1mm},
        Phi/.style={white!40!red,thick,snake=coil,segment amplitude=0.6pt, segment length=2pt},
         Z/.style={black!40!green,thick,snake=coil,segment amplitude=0.6pt, segment length=2pt},
        C/.style={thick,black!20!blue},
          Cr/.style={thick,black!20!red},
            Cg/.style={thick,black!20!green},
       }

\begin{document}

\title[An  SPDE approach to perturbation theory of $\Phi^4_2$: asymptoticity and short distance behavior]{\LARGE An  SPDE approach to perturbation theory of $\Phi^4_2$: asymptoticity and short distance behavior}

\author[Hao Shen]{\large Hao Shen}
\address[H. Shen]{Department of Mathematics, University of Wisconsin - Madison, USA}
\email{pkushenhao@gmail.com}
%
%
\author[Rongchan Zhu]{\large Rongchan Zhu}
\address[R. Zhu]{Department of Mathematics, Beijing Institute of Technology, Beijing 100081, China; Fakult\"at f\"ur Mathematik, Universit\"at Bielefeld, D-33501 Bielefeld, Germany}
\email{zhurongchan@126.com}

\author[Xiangchan Zhu]{\large Xiangchan Zhu}
\address[X. Zhu]{ Academy of Mathematics and Systems Science,
Chinese Academy of Sciences, Beijing 100190, China; Fakult\"at f\"ur Mathematik, Universit\"at Bielefeld, D-33501 Bielefeld, Germany}
\email{zhuxiangchan@126.com}

\begin{abstract}
 In this paper we study the perturbation theory of  $\Phi^4_2$ model on the whole plane via stochastic quantization. We use integration by parts formula (i.e. Dyson-Schwinger equations) to generate
the perturbative  expansion for the $k$-point correlation functions, and prove bounds on the remainder of the truncated expansion using PDE estimates; this in particular
proves that the expansion is asymptotic. Furthermore, we derive short distance behaviors of the $2$-point function and the connected $4$-point function, also via suitable Dyson-Schwinger equations combined with PDE arguments.
\end{abstract}


\date{\today}

\begin{nouppercase}
\maketitle
\end{nouppercase}

\setcounter{tocdepth}{1}
\tableofcontents

\section{Introduction}

Physicists' approach to quantum field theory often relies on perturbation theory, which is a formal series expansion
in terms of a certain coupling constant, see for instance the standard book \cite[Chapter~7]{zinn}.
The first goal of this paper is to discuss this  perturbative approach
and provide some estimates on the expansion
using SPDE methods. We will focus on the $\lambda\Phi^4$
model in two space dimensions, formally given by
\begin{align}\label{eq:in}\dif \nu\varpropto\exp(-\frac{\lambda}{4} \int_{\mR^2} \Wick{\Phi(x)^4}\dif x ) \dif\mu,\end{align} where $\dif\mu$ is the Gaussian free field with mean zero and variance $\frac12(m-\Delta)^{-1}$ for $m>0$.
In physics, interesting quantities
such as correlations
are calculated by formally Taylor expanding the exponential density in the coupling constant $\lambda$,
and then applying  Wick theorem for Gaussian measures (which yields Feynman diagrams).
Jaffe \cite{jaffe1965} proved that this formal expansion  is a divergent series for any {\it fixed} $\lambda>0$.

A toy model to illustrate this divergent perturbation theory
is a $\Phi^4$ model in ``zero dimension'':
consider the following function (``partition function'')
$$
\mathcal Z(\lambda) = \int_{\R} e^{-x^2} e^{-\lambda x^4}\dif x \;.
$$
One has $\mathcal Z(0)=\sqrt{\pi}$.  For $\lambda>0$,
one can formally calculate it by
the expansion $e^{-\lambda x^4} = \sum_{n=0}^\infty \frac{(-\lambda)^n}{n!} x^{4n}$, since there is exact formula for Gaussian moments, i.e.
\begin{equ}[e:toy]
\mathcal Z(\lambda) =  \sum_{n=0}^\infty  \int_{\R}
\frac{(-\lambda)^n}{n!} x^{4n} e^{-x^2}\dif x
= \sum_{n=0}^\infty
\frac{(-\lambda)^n}{n!}
\frac{(4n-1)!!\sqrt{\pi}}{2^{2n}}\;.
\end{equ}
This series is divergent for any {\it fixed} $\lambda>0$.
However, despite its divergence,
this  perturbative approach  is widely used
for other quantum field theories, such as quantum electrodynamics (QED),
gauge theory, or the so-called {\it ``standard model''} in general that comprises these models; and these perturbative calculations yield numerical values which agree with experiments at extremely high accuracy, see e.g. \cite[Eq.(4)]{QED}
for calculation of the electron's magnetic moment
 in QED  which agrees with experiment to eleven decimal places.

This means that practically
it is still often useful to calculate physical quantities
by truncating this series.
For instance in the above toy model, $\mathcal Z(0.01)\approx 1.7597$,
and the first two terms on the RHS of \eqref{e:toy} are approximately
  $1.7725-0.0133 = 1.7592$ which is close to the exact value.
 \footnote{Of course since the series is divergent,
for the fixed $\lambda=0.01$ the partial sum of \eqref{e:toy} will not be infinitely close to the exact value of $\mathcal Z(0.01)$; indeed
the $n=70$-th term on the RHS
is approximately equal to $0.9$ and a truncation up to such a high order certainly loses accuracy.}
We also refer to another interesting example along this line with illuminating discussion in \cite[Section~6.1]{folland}.
\footnote{The author explains that convergence and effectiveness of approximation are completely different. For example $\sum_{0}^\infty \frac{(-100)^n}{n!}$ converges to $e^{-100}\approx 10^{-44}$, but unfortunately its  first few partial sums are $1, -99, 4901,-161765\frac23$.}

We remark that the model \eqref{eq:in} we consider here is a type of bosonic models,
whereas the perturbation series for some of the ferminoic models actually converge, see \cite{GK85} (and more recently \cite[Sec.~B,D]{ABDG2020} in the setting of stochastic quantization).

The above discussion indicates that {\it asymptoticity} is a more relevant
notion than convergence in perturbation theory.
Asymptoticity means
that truncating the series after a {\it fixed} finite number of terms provides a good approximation to $\mathcal Z(\lambda)$ as $\lambda$ tends towards zero.
For the $\Phi^4_2$ model Dimock \cite{dimock1974} proved that the perturbative expansion of correlation functions (Schwinger functions) is asymptotic. His proof crucially relies on the constructive field theory results and methods by
Glimm, Jaffe, Spencer \cite{MR363256}.
The asymptoticity of $\Phi^4_3$
was also proved in \cite[Theorem 3]{FO76}, see also \cite[Sec.IV.2]{Magnen76}. We remark that a key input
 to the proof in
 \cite{dimock1974}  is a uniform exponential cluster property
 proven in  \cite{MR363256}, from which Dimock managed to deduce that correlation functions are smooth in $\lambda\in [0,\lambda_0]$ where $\lambda_0>0$ is sufficiently small,
 and asymptoticity essentially follows from Taylor remainder theorem. The smallness of $\lambda$ is mainly due to requirement for the cluster property to hold.
This is also the case for $d=3$. In fact  \cite{FO76} proved cluster property in 3d in detail, and only mentioned that smoothness in $\lambda$ and asymptoticity follow analogously as in 2d  \cite{dimock1974}. \cite[Sec.IV.2]{Magnen76} provided more details but the basic mechanism is also similar with  \cite{dimock1974}. The relation between smoothness of correlations in coupling constants and cluster properties was observed in statistical mechanics, see
\cite{MR323271}.  { Using the results of \cite{dimock1974}, 
Eckmann--Epstein--Fr\"ohlich \cite{MR418716} proved that in weakly coupled $P(\Phi)_2$ theories, perturbation theory in the coupling constant is asymptotic to the S-matrix elements; and the result was extended to $\Phi^4_3$ in \cite{MR519920} where it was shown that 
the time-ordered functions and S-matrix elements admit the standard perturbation series as asymptotic expansions.

Another approach to asymptotic perturbation theory for (one and two-component) $\Phi^4$ model in 2d and 3d was developed by Bovier and Felder in \cite{BF1984}, using an infinite family of skeleton inequalities. The leading orders of these skeleton inequalities were proved by Brydges--Fr\"ohlich--Sokal in \cite{MR648362,MR719815} and applied to construct $\Phi^4_{2,3}$ in \cite{BFS83}.
}

In this paper we give a new  proof of asymptoticity for $\Phi^4_2$ based on SPDE methods.
More precisely, we consider the stochastic quantization of $\Phi^4_2$ model (i.e. the dynamical $\Phi^4_2$ model) on $\mR^+\times \mR^2$:
\begin{align}\label{eq:st}(\p_t-\Delta+m)\Phi=-\frac12\lambda\Wick{\Phi^3}+\xi,\end{align}
with $\xi$ space-time white noise on $\mR^+\times \mR^2$  and  $\Wick{\Phi^3}$ is Wick product defined in Section \ref{sec:A2}. The solution theory of equation \eqref{eq:st}
 is now well developed: see \cite{MR1113223} where martingale solutions are constructed and \cite{DD03} where strong solutions are addressed, as well as the more recent approach to global well-posedness in \cite{MW17}.  It is obtained in \cite{RZZ17} that   the marginal distribution of  the stationary solution to \eqref{eq:st} is given by $\Phi^4_2$ field. Hence, we could use \eqref{eq:st} to study the $\Phi^4_2$ field.

 Another ingredient in our approach is
 the {\it integration by parts (IBP)} formula w.r.t. the $\Phi^4_2$ field, also called Dyson-Schwinger equations.
In our SPDE approach it seems not easy to establish cluster properties
or prove smoothness in $\lambda$ over the entire plane,
  so we do not follow the methodology in \cite{dimock1974}. Instead, we iteratively apply
IBP to generate the perturbation theory {\color{blue} (which is similar with \cite{BFS83} and \cite{BF1984})},
which after a certain order  yields a remainder represented
as integrals of the field $\Phi$ and Gaussian covariances.
Our strategy is then to introduce graphic
representation for such a remainder,
and then apply PDE a priori estimates
to control the graphs.
To this end we will carry out some manipulations on the graphs,
reduce graphs to trees
and design inductive arguments.

 As $\Phi^4_2$ field does not have full support on function space, we consider the following approximation:
 we start with Gibbs measures $(\nu_{M,\eps})_{M,\eps}$ on $\Lambda_{M,\eps}=(\eps(\mZ/ M\mZ))^2$ given by
 \begin{align*}
 \dif \nu_{M,\eps}\varpropto\exp\Big\{-\eps^2\sum_{\Lambda_{M,\eps}}\Big[\frac\lambda4\Phi^4+(-\frac32\lambda a_{M,\eps}+m)\Phi^2+|\nabla_\eps\Phi|^2\Big]\Big\}\prod_{x\in \Lambda_{M,\eps}}\dif \Phi(x),\end{align*}
 where $\nabla_\eps$ denotes the discrete gradient and $a_{M,\eps}$ are renormalization constants. We will use the stochastic quantization equation for $\nu_{M,\eps}$
  to prove that $\{\nu_{M,\eps}\}_{M,\eps}$ form a tight set in Section \ref{sec:A2} (see also \cite{GH18a} for  the proof of $3$d case).  We  still use $(M,\eps)$ to denote the subsequence such that $\nu_{M,\eps}\circ (\cE^{\eps})^{-1}$ converge to $\nu$ weakly with $\cE^{\eps}$ being the extension operator given in \eqref{def:E}. Here $\nu$ could be every accumulation point  and all the theorems hold for every $\nu$.

The $k$-point correlations for $\nu_{M,\eps}$
-- for which we will expand using
 integration by parts formula -- are
given by
 $$
 S_{\lambda,M,\eps}^k (x_1,\dots,x_k) =\int\Big[ \prod_{i=1}^k\Phi(x_i)\Big]\nu_{M,\eps}(\dif \Phi)\;.
 $$
Letting $M\to\infty, \eps\to0$ we obtain the $k$-point function for $\nu$
 \begin{align}\label{def:S1}
 \<S_\lambda^{\nu,k},\varphi\>\eqdef\lim_{\eps\to 0, M\to\infty}\int\int_{\mR^{2k}}\Big(\prod_{i=1}^k\Phi(x_i)\Big)\varphi(x_1,\dots,x_k)\prod_{i=1}^k\dif x_i\dif \nu_{M,\eps}(\Phi),
 \end{align}
 for $\varphi\in \mathcal{S}(\mR^{2k})$ with compact support Fourier transform, where the existence of the limit will be discussed  in Theorem \ref{lem:cm}. We also give the definition of $S_\lambda^{\nu,k}$ using extension operator in Section \ref{sec:3.1}, which coincides with the above definition \eqref{def:S1}.
 Our first main result gives the asymptotic expansion of $S_\lambda^{\nu,k}$.

\begin{theorem}\label{th:main}
It holds that for $\varphi\in \cS(\mR^{2k})$
\begin{equ}[e:main]
	\<S_\lambda^{\nu,k},\varphi\>  = \sum_{n=0}^N  \frac{\lambda^n}{n!} \<F_{n}^k, \varphi\> + \lambda^{N+1} \<R_{N+1}^{\nu,k},\varphi\>,
\end{equ}
with $F_n^k$ only depending on the Green function of Gaussian free field  and
$$
|\<R_{N+1}^{\nu,k},\varphi\>|\lesssim 1+\lambda^{{4k}+12(N+1)}
$$
with the proportional constant independent of $\lambda$ \footnote{Here the proportional constant may depend on $N$. We remark that the terms $\lambda^{{4k}+12(N+1)}$, $\lambda^{27}$ and $\lambda^{54}$ in these theorems provide controls
	for large $\lambda$, but these exponents may not be optimal (see Remark \ref{re}).}.
\end{theorem}
The complete version of Theorem \ref{th:main} is given in Theorem \ref{th:Sk}  and Proposition \ref{pro}.
Note that \cite{dimock1974} only amounts to $\lambda \<R_{N+1}^k,\varphi\> \to 0$ as $\lambda \downarrow 0$, 
{ and the result in \cite{BF1984} also assumes small $\lambda$ regime,}
whereas we give a more precise bound here which holds for arbitrary $\lambda>0$. Moreover, we obtain (see Theorem \ref{th:Sk}  and Proposition \ref{pro}) that \eqref{e:main} holds in $\bC^{-\gamma}(\rho^\ell)$ for $\gamma>0$, some $\ell>0$ and  polynomial weight $\rho$ and the same bound holds for $\|R_{N+1}^k\|_{\bC^{-\gamma}(\rho^\ell)}$. Here $\bC^{-\gamma}(\rho^\ell)$ is the weighted Besov-H\"older space (see Appendix \ref{App} for definition).

Note that the expansion \eqref{e:main} is precisely the Taylor
expansion in the coupling constant $\lambda$ in the aforementioned perturbation theory.
Namely, by taking the $n$-th derivative on the RHS of  \eqref{e:main} and evaluating at $\lambda=0$,
and invoking the above bound on $R_{N+1}^k$, we have
$$
\frac{\dif^n}{\dif\lambda^n}\Big|_{\lambda =0}  \<S_\lambda^{\nu,k} ,\varphi\>
=\<F_n^k,\varphi\>  \qquad \forall n\le N\;.
$$
In physics these derivatives on the LHS
result in expectations of Gaussian products and
are typically computed
by Wick theorems and Feynman diagrams.

\br
By \cite{jaffe1965} each $F_n^k$ has a lower bound depending on $n$ that is not summable, which implies the expansion is a divergent series.
\er

\vspace{2ex}

As the second goal of the paper,
our strategy of combining integration by parts formula and SPDE uniform estimates also allow us to find the leading  short distance behavior  for the $2$-point correlation and the connected $4$-point correlation. Let $C_{M,\eps}$(resp. $C_\eps$) be the Green function of $2(m-\Delta_{\eps})$ on $\Lambda_{M,\eps}$ (resp. on $\Lambda_\eps=\eps \mZ^2$). We also use $C$ to denote the Green function of $2(m-\Delta)$ on $\mR^2$.

\bt\label{thm:2point} Let $\rho(x)=(1+|h x|^2)^{-\delta/2}$  be a polynomial weight with  $h>0$ small enough and $\delta>1/2$.
It holds that for $\gamma>0$ and some $\ell>0$
\begin{equ}[e:2point]
\|S_{\lambda,M,\eps}^2-C_{M,\eps}\|_{\bC^{2-\gamma,\eps}(\Lambda_{\eps},\rho^\ell)}\lesssim \lambda^2+\lambda^{27},
\end{equ}
with the proportional constant independent of $M,\eps$. Furthermore taking tight limit, it holds that
$$\|S_{\lambda}^{\nu,2}-C\|_{\bC^{2-\gamma}(\mR^2,\rho^\ell)}\lesssim \lambda^2+\lambda^{27}.$$
\et

 Here $\bC^{2-\gamma,\eps}(\rho^\ell)$ is the discrete Besov H\"older space and we refer to Appendix \ref{App} for its definition.

\br  By Theorem \ref{thm:2point} we can write $S_\lambda^{\nu,2}$ as a function on $\mR^2$ and the leading short distance behavior of $S_\lambda^{\nu,2}$ is identical to that of $C$ up to a $\bC^{2-\gamma}$ part (although each of them are singular at origin).

\er

We define the following connected $4$-point function
\begin{align*}
U^4_{M,\eps}(x_1,x_2,x_3,x_4)
&\eqdef S^4_{\lambda,M,\eps}(x_1,x_2,x_3,x_4)-S_{\lambda,M,\eps}^2(x_1,x_2) S_{\lambda,M,\eps}^2(x_3,x_4)
\\& \quad -S_{\lambda,M,\eps}^2(x_1,x_3)S_{\lambda,M,\eps}^2(x_2,x_4)-S_{\lambda,M,\eps}^2(x_1,x_4)S_{\lambda,M,\eps}^2(x_2,x_3).
\end{align*}

\bt\label{th:main3}
Let $\rho(x)=(1+|h x|^2)^{-\delta/2}$  be a polynomial weight with  $h>0$ small enough and $\delta>0$. It holds that for $\gamma>0$ and some $\ell>0$
$$\|U^4_{M,\eps}\|_{\bC^{2-\gamma,\eps}(\Lambda_\eps^4,\rho^\ell)}\lesssim\lambda+\lambda^{54},$$
and
$$\Big\|U^{4}_{M,\eps}+6\lambda\int_{\Lambda_{M,\eps}}\Pi_{i=1}^4C_{M,\eps}(x_i-z)\dif z \Big\|_{\bC^{2-\gamma,\eps}(\Lambda_\eps^4,\rho^\ell)}\lesssim\lambda^2+\lambda^{54},$$
with the proportional constant independent of $M,\eps$. Moreover, as in \eqref{def:S1} the tight limit of $U^4_{M,\eps}$ exists which we denote by $U^{\nu,4}_\lambda$ and $$\|U^{\nu,4}_\lambda\|_{\bC^{2-\gamma}(\mR^8,\rho^\ell)}\lesssim\lambda+\lambda^{54},$$
and
$$\Big\|U^{\nu,4}_\lambda+6\lambda\int\Pi_{i=1}^4C(x_i-z)\dif z \Big\|_{\bC^{2-\gamma}(\mR^8,\rho^\ell)}\lesssim\lambda^2+\lambda^{54}.$$
\et

The proof of the above two results are given in Section \ref{sec:short}.
We will see in the proof of Theorem \ref{th:main3} that the expansion of $U^{\nu,4}_\lambda$ cancels some irregular part, and as a result,
the regularity  of  $U^{\nu,4}_\lambda$ is $\bC^{2-\gamma}$ which is better than the regularity of $S^{\nu,4}_\lambda$ which is $\bC^{-\gamma}$, for $\gamma>0$.
Here $\bC^{-\gamma}$  is the optimal regularity of $S^{\nu,4}_\lambda$ if considering in Besov--H\"older space since the Green's function $C$, which is not continuous, appears in the expansion of $S^{\nu,4}_\lambda$.

Similar results were also obtained in \cite[Theorem~1.1]{BFS83}, but in our results above
we obtain bounds under stronger regularity norms;
 moreover  \cite{BFS83} requires $\lambda>0$ to be sufficiently small whereas we do not.

The short distance behavior of $\Phi^4_3$ on the torus has also been studied by Brydges--Dimock--Hurd, see \cite[Theorem 9]{BDH95}, where the $L^p, p<\frac32$-norm of truncated correlations are controlled.
Here we consider the case of $\Phi^4_2$ on the full plane  and we are able to estimate the $\bC^{2-\gamma}$-H\"older norm of $U^{\nu,4}_\lambda$ and $S^{\nu,2}_\lambda-C$.
We refer to the remark at the end of \cite{BDH95}
for a discussion on regularity and
 how singular the truncated correlation functions can be at coinciding points.

\vspace{2ex}

{ We remark that as far as we understand,
our technique in proving the above theorems 
appears to be more robust than skeleton inequalities used in earlier literature (e.g. \cite{BF1984} and \cite{BFS83}), in the sense that
at least in 2d, it is valid for all values of the coupling
	constant $\lambda$ and can be very easily adapted to $N$ component models for any $N>1$ (PDE estimates for such models were discussed in \cite[Section 2 and Section 6.2]{SSZZ2d}).}

Finally we  mention  that it would be interesting
 to see whether  our methodology could be used to  deduce  similar results in 3d.
For 3d stochastic quantization, namely
the dynamical $\Phi^4_3$ model,
we refer to
the  construction of  local solutions \cite{Hairer14,GIP15,MR3846835}
and global solutions \cite{MW18,GH18,AK17, GH18a}.
These works, followed by   \cite{MoinatCPAM,Steele,SSZZ2d,SZZ3d,AK21},
including the present paper,
show that the usage of SPDE methods
is not only successful in constructing Euclidean QFTs
but also proving their axioms and properties.
As remarked in \cite{GH18a}
deducing properties about correlations from the dynamics would be a very interesting and challenging problem,
and our results in the present paper could be viewed as a step toward this direction. It would be also interesting to study connected correlations of higher orders.
We leave these generalizations to further work.

\subsection*{Acknowledgements}
H.S. gratefully acknowledges support by NSF grants DMS-1954091 and CAREER DMS-2044415. {\color{blue} R.Z. and X.Z. are grateful to the financial supports by National Key R\&D Program of China (No. 2022YFA1006300). R.Z. gratefully acknowledges financial support from the NSFC (No.  12271030). X.Z. is grateful to the financial supports in part by National Key R\&D Program of China
(No. 2020YFA0712700) and the NSFC (No. 12090014, 12288201) and the support by key Lab of Random Complex
Structures and Data Science, Youth Innovation Promotion Association (2020003), Chinese Academy of Science.
R.Z. and X.Z. are grateful to
the financial supports of 
 the financial support by the DFG through the CRC 1283 ``Taming uncertainty.}
 We also thank an anonymous referee who brought an important reference \cite{BF1984} to our attention. {\color{blue} Rongchan Zhu is the corresponding author.}

\section{Uniform estimates via SPDEs}\label{sec:A2}

In this section we study $\Phi^4_2$ field via stochastic quantization. We will obtain tightness of the discrete $\Phi^4_2$ field, which gives a construction of $\Phi^4_2$ field in the continuous setting. We also give some uniform estimates via SPDEs, especially how these estimates depend on the parameter $\lambda$.

We consider  $\Phi^4_2$ model in both discrete and continuous settings. In particular, we set $\Lambda_\eps:=\eps \mZ^2$ for $\eps=2^{-N}, N\in \mN\cup\{0\}$ and the periodic lattice $\Lambda_{M,\eps}:=\eps \mZ^2\cap \mT^2_{M}=\eps \mZ^2\cap [-\frac M2,\frac M2)^2$. We use $\cS(\mR^d)$ to denote the class of Schwartz functions on $\mR^d$ and $\cS'(\mR^d)$ denote the dual space of $\cS(\mR^d)$ called Schwartz generalized function space. Given a Banach space $E$ with a norm $\|\cdot\|_E$ and $T>0$, we write $C_TE=C([0,T];E)$ for the space of continuous functions from $[0,T]$ to $E$, equipped with the supremum norm $\|f\|_{C_TE}=\sup_{t\in[0,T]}\|f(t)\|_{E}$ and we write $L_T^pE=L^p([0,T];E)$ for the space of $L^p$ functions from $[0,T]$ to $E$, equipped with the $L^p$ norm $\|f\|_{L^p_TE}=(\int_0^T\|f(t)\|_{E}^p\dif t)^{1/p}$. We also use $o(\eps)$ to denote $a(\eps)\in \mR$ satisfying $\lim_{\eps\to0}\frac{a(\eps)}{\eps}=0$.
 In this paper we use the polynomial weight on $\mR^d$ given by $$\rho(x)=(1+|hx|^2)^{-\delta/2}, \quad \delta,h\geq0.$$  It is easy to see that
 \begin{align}\label{bd:rho}
 |\nabla \rho(x)|\lesssim |h|\rho(x).
 \end{align}
  We use $B^{\alpha,\eps}_{p,q}(\Lambda_\eps^k,\rho), \alpha\in \mR, p,q\in [1,\infty]$ to denote the weighted Besov spaces  on $\Lambda_{\eps}^k, k\in \mN,$ and we recall their definition  in Appendix \ref{App}. We also write $B^{\alpha,\eps}_{p,q}(\rho)=B^{\alpha,\eps}_{p,q}(\Lambda_\eps^k,\rho)$ if there's no confusion. Note that if $\eps=0$, $B^{\alpha,\eps}_{p,q}(\rho)$ is the classical weighted Besov space $B^{\alpha}_{p,q}(\rho)$ on $\mR^{2k}$.  Set $H^{\alpha,\eps}(\rho)=B^{\alpha,\eps}_{2,2}(\rho)$, $\bC^{\alpha,\eps}(\rho)=B^{\alpha,\eps}_{\infty,\infty}(\rho)$ and $H^{\alpha}(\rho)=B^{\alpha}_{2,2}(\rho)$, $\bC^{\alpha}(\rho)=B^{\alpha}_{\infty,\infty}(\rho)$.  We also use $\bC^{\alpha,\eps}_s$ to denote Besov-H\"older space w.r.t. each component on $\Lambda_{\eps}^k$ for $k\in \mathbb{N}$ with definition given in Appendix \ref{App}.
 The duality on $\Lambda_\eps$ is given by
$$\<f,g\>_\eps\eqdef \eps^2\sum_{x\in \Lambda_\eps}f(x)g(x),$$
with $\<\cdot,\cdot\>$ as  the usual duality between $\cS(\mR^d)$ and $\cS'(\mR^d)$.
We also write $L^{p,\eps}$ for the $L^p$ space on $\Lambda_\eps$ endowed with the norm
$$\|f\|_{L^{p,\eps}}=\Big(\eps^2\sum_{x\in \Lambda_\eps}|f(x)|^p\Big)^{1/p},$$
with the usual modification for $p=\infty$. We refer to Appendix \ref{App} for discrete Besov embedding, duality, interpolation, which will be used below.

We start with Gibbs measures $(\nu_{M,\eps})_{M,\eps}$ on $\Lambda_{M,\eps}$ given by
\begin{align}\label{eq:Phiin}
\dif \nu_{M,\eps}\varpropto\exp\Big\{-\eps^2\sum_{\Lambda_{M,\eps}}\Big[\frac\lambda4\Phi^4+(-\frac32\lambda a_{M,\eps}+m)\Phi^2+|\nabla_\eps\Phi|^2\Big]\Big\}\prod_{x\in \Lambda_{M,\eps}}\dif \Phi(x),\end{align}
where $$\nabla_\eps f(x)=\Big(\frac{f(x+\eps e_i)-f(x)}{\eps}\Big)_{i=1,2} $$ denotes the discrete gradient and $a_{M,\eps}$ are renormalization constants defined below. Here $(e_i)_{i=1,2}$ is the canonical basis in $\mR^2$. We write
$$\Delta_\eps f(x)=\eps^{-2}(f(x+\eps e_i)+f(x-\eps e_i)-2f(x)), \quad  x\in \Lambda_\eps$$
as the discrete Laplacian on $\Lambda_{\eps}$ and
$\sL_\eps:=\p_t+m-\Delta_\eps$.
We also use
$\Phi_{M,\eps}$ to denote the stationary solution to the discrete stochastic quantization equation on $\Lambda_{M,\eps}$
\begin{align}\label{eq:Phi}
\sL_\eps\Phi_{M,\eps}+\frac\lambda2 \Phi_{M,\eps}^3-\frac32\lambda a_{M,\eps}\Phi_{M,\eps}=\xi_{M,\eps},\end{align}
where $\xi_{M,\eps}$ is a discrete approximation of a space-time white noise $\xi$ on $\mR^+\times \mR^2$  constructed as follows: Let $\xi_M$ denote the periodization of $\xi$ on $\mR^+\times \mT^2_{M}$ and define the  spatial discretization by
$$\xi_{M,\eps}(t,x):=\eps^{-2}\<\xi_M(t,\cdot),1_{|\cdot-x|_\infty\leq \eps/2}\>,
\qquad (t,x)\in \mR^+\times \Lambda_{M,\eps},$$
where $|x|_\infty=|x_1|\vee |x_2|$ for $x=(x_1,x_2)$.
The law of $\Phi_{M,\eps}$ at every time $t\geq0$ is given by $\nu_{M,\eps}$ defined in \eqref{eq:Phiin}, which is the unique invariant measure to the finite dimensional gradient system \eqref{eq:Phi} (see e.g. \cite{DPZ}). We extend  all the functions and distributions  to $\Lambda_{\eps}$ by periodic extension.

The aim of this section is to show that
$\{\nu_{M,\eps}\}_{M,\eps}$ (properly extended to $\cS'(\mR^2)$) is tight in $\bC^{-\gamma}(\rho)$  for $\gamma>0$ and some polynomial weight $\rho$.
Every accumulation point $\nu$ is an invariant measure of the following SPDE on $\mR^+\times\mR^2$:
\begin{equation}\label{eq:Phi1}
\sL\Phi+\frac\lambda2\Wick{\Phi^3}=\xi,
\end{equation}
with $\sL=\p_t-\Delta+m$.
Here $\Wick{\Phi^3}$ is the Wick power defined by
$$\Wick{\Phi^3}=\lim_{\substack{\eps\to 0 \\ M\to\infty}}\cE^\eps (\Phi^3_{M,\eps}-3a_{M,\eps}\Phi_{M,\eps})$$
where  $\cE^\eps$ is an extension operator from $B^{\alpha,\eps}_{p,q}(\rho)$ to $B^\alpha_{p,q}(\rho)$ for some $\alpha<0$ defined in \eqref{def:E} in Appendix \ref{App}.

Here we understand \eqref{eq:Phi1}  decomposed as the following two equations: $\Phi=Z+Y$
\begin{equation}\label{eq:YZ}
\sL Z=\xi,\quad \sL Y=-\frac\lambda2 (Y^3+3Y^2Z+3Y\Wick{Z^2}+\Wick{Z^3}),
\end{equation}
where
$$\Wick{Z^2}\eqdef\lim_{\eps\to 0, M\to\infty}
\cE^\eps \Wick{Z^2_{M,\eps}}\eqdef\lim_{\eps\to 0, M\to\infty}\cE^\eps (Z^2_{M,\eps}-a_{M,\eps}),$$
$$\Wick{Z^3}\eqdef\lim_{\eps\to 0, M\to\infty}\cE^\eps \Wick{Z^3_{M,\eps}}\eqdef\lim_{\eps\to 0, M\to\infty}\cE^\eps (Z^3_{M,\eps}-3a_{M,\eps}Z_{M,\eps}).$$
Here $Z_{M,\eps}$ are stationary solutions to
\begin{align}\label{eq:ZM}\sL_\eps Z_{M,\eps}=\xi_{M,\eps},\end{align} $a_{M,\eps}=\E Z_{M,\eps}^2(0)$ and the limits are in $C([0,T];\bC^{-\gamma}(\rho^{\kappa}))$ $\mathbf{P}$-a.s. for any $\gamma, \kappa>0$ and $T>0$. By standard renormalization calculation we have that for every $p\geq1, \kappa, \gamma>0$ and $T>0$
\begin{align}\label{bd:Z}
\E\|{Z_{M,\eps}}\|_{C_T\bC^{-\gamma,\eps}(\rho^\kappa)}^p+\E\|\Wick{Z^2_{M,\eps}}\|_{C_T\bC^{-\gamma,\eps}(\rho^\kappa)}^p+\E\|\Wick{Z^3_{M,\eps}}\|_{C_T\bC^{-\gamma,\eps}(\rho^\kappa)}^p\lesssim1,
\end{align}
with the proportional constant independent of $\eps, M$,
and
\begin{align}\label{bd:Z1}
\E\|\cE^\eps{Z_{M,\eps}}-Z\|_{C_T\bC^{-\gamma}(\rho^\kappa)}^p+\E\|\cE^\eps\Wick{Z^2_{M,\eps}}-\Wick{Z^2}\|_{C_T\bC^{-\gamma}(\rho^\kappa)}^p+\E\|\cE^\eps\Wick{Z^3_{M,\eps}}-\Wick{Z^3}\|_{C_T\bC^{-\gamma,\eps}(\rho^\kappa)}^p\to0,
\end{align}
as $M\to \infty, \eps\to0$ (c.f. \cite{MW17, GH18a, ZZ18}).
Global well-posedness of equations \eqref{eq:YZ} and hence equation \eqref{eq:Phi1} have been obtained in \cite{MW17}. In the following we take $(\Phi_{{M,\eps}}, Z_{M,\eps})$ as joint stationary processes satisfying equations \eqref{eq:Phi} and \eqref{eq:ZM}, respectively, which could be constructed by \cite[Lemma 5.7]{SSZZ2d}.  Hence, $Y_{M,\eps}\eqdef \Phi_{{M,\eps}}-Z_{M,\eps}$ satisfy the following equation on $\Lambda_{M,\eps}$:
\begin{align}\label{eq:Y}
\sL_\eps Y_{M,\eps}=-\frac{\lambda}{2}(Y_{M,\eps}^3+3Y_{M,\eps}^2Z_{M,\eps}+3Y_{M,\eps}\Wick{Z_{M,\eps}^2}+\Wick{Z_{M,\eps}^3}).
\end{align}

In the following we  derive uniform estimates in both parameters $M, \eps$ for $Y_{M,\eps}$. To this end, we recall that all the distributions above are extended periodically to the full lattice $\Lambda_{\eps}$.

For general $p\geq2$ we  have the following $L^p$ uniform bounds.

\bl\label{lem:estp} For the polynomial weight $\rho(x)=(1+|h x|^2)^{-\delta/2}$ and $h>0$ small enough, it holds that
for every $p\geq 1, \delta>\frac12$ and $0<\sigma< 1$
\begin{align}\label{eq:zmm1}\E\|Y_{M,\eps}\rho^2\|_{L^{2,\eps}}^{2p}+\E\|Y_{M,\eps}\rho^2\|_{H^{1-\sigma,\eps}}^2\|Y_{M,\eps}\rho^2\|_{L^{2,\eps}}^{2(p-1)}\lesssim \lambda^{p}+\lambda^{3p}.\end{align}
Furthermore, for every $p_0\geq 2$ there is $h>0$ small enough such that
for the  weight $\rho(x)=(1+|hx|^2)^{-\delta/2}$ and every $2\leq p\leq p_0$ with $p$ even, $p\delta>2$
\begin{align}\label{eq:zmm1p}\E\|Y_\eps^p\rho^p\|_{L^{1,\eps}}+\E\||\nabla_\eps Y_\eps|^2|Y_\eps|^{p-2}\rho^p\|_{L^{1,\eps}}\lesssim \lambda+\lambda^{2+p/2}.\end{align}
Here all the proportional constants are independent of $M,\eps$.
\el

\begin{proof}
We omit the subscript $M$ for notation's simplicity, and all the proportional constants in the proof are independent of $\eps$ and $M$. We first note that for the canonical basis $\{e_i\}_{i=1}^2$ in $\mR^2$ and even $p$
	\begin{align}
	\nabla_\eps Y_\eps(x)\cdot \nabla_\eps Y_\eps^{p-1}(x)=&\sum_{i=1}^2(\nabla^i_\eps Y_\eps(x))^2\Big(\sum_{\ell=2}^pY_\eps^{p-\ell}(x+\eps e_i)Y_\eps^{\ell-2}(x)\Big)\no
	\\\geq& \frac12\Big(|\nabla_\eps Y_{\eps}(x)|^2|Y_{\eps}(x)|^{p-2} +\sum_{i=1}^2 (\nabla_\eps^i Y_\eps(x))^2 Y_\eps^{p-2}(x+\eps e_i)\Big) ,\label{eq:p}
	\end{align}
	where $\nabla_\eps^i Y_\eps(x)\eqdef\frac1{\eps} (Y_\eps(x+\eps e_i)-Y_\eps(x))$ and we used
	$Y_\eps^{p-\ell}(x+\eps e_i)Y_\eps^{\ell-2}(x)\geq0$ for $\ell$ even and $$Y_\eps^{p-\ell}(x+\eps e_i)Y_\eps^{\ell-2}(x)\leq \frac12Y_\eps^{p-\ell+1}(x+\eps e_i)Y_\eps^{\ell-3}(x)+\frac12Y_\eps^{p-\ell-1}(x+\eps e_i)Y_\eps^{\ell-1}(x)$$
	for $\ell$ odd
	in the last step.
	Taking the $L^{2,\eps}$ inner product with $Y_{\eps}^{p-1}\rho^p$ on the both sides of \eqref{eq:Y} and using  \eqref{eq:p} we obtain
	\begin{align}
	\frac1p &  \frac{\dif}{\dif t}\|Y_{\eps}^p\rho^p\|_{L^{1,\eps}}+\frac12(D_1+D_2)+m\|Y_{\eps}^p\rho^p\|_{L^{1,\eps}}^2+\frac\lambda2\|Y_{\eps}^{p+2}\rho^p\|_{L^{1,\eps}} 	\notag
	\\
	&\leq |\<[\nabla_\eps,\rho^p]Y_{\eps}^{p-1},\nabla_\eps Y_{\eps}\>_\eps|-\frac{3\lambda}2\<	\rho^pY_{\eps}^{p+1},Z_{\eps}\>_\eps -\frac{3\lambda}2 \<\rho^pY_{\eps}^p,\Wick{Z_{\eps}^2}\>_\eps- \frac\lambda2 \<\rho^pY_{\eps}^{p-1},\Wick{Z_{\eps}^3}\>_\eps\label{eq:bd3}
	\\&\eqdef\sum_{k=1}^4I_k,\no
	\end{align}
	where
	$$ D_1\eqdef \||\nabla_\eps Y_{\eps}|^2|Y_{\eps}|^{p-2} \rho^p\|_{L^{1,\eps}},\quad D_2\eqdef\Big\|\sum_{i=1}^2(\nabla_\eps^i Y_\eps)^2 \cdot Y_\eps^{p-2}(\cdot+\eps e_i)\rho^p\Big\|_{L^{1,\eps}}.$$
	We also set
	$B\eqdef\|Y_{\eps}^{p+2}\rho^p\|_{L^{1,\eps}}$.
	For $I_1$ we have
	\begin{align*}
	I_1\leq C_\rho \|\rho^pY_{\eps}^p\|_{L^{1,\eps}}^{1/2}D_1^{1/2} \leq C_{\rho,\gamma}\|\rho^pY_{\eps}^p\|_{L^{1,\eps}}+\gamma D_1,
	\end{align*}
	where the  constants $C_\rho, C_{\rho,\gamma}$ depend on $\|\rho^{-p}[\nabla_\eps,\rho^p]\|_{L^{\infty,\eps}}$, which for given $p_0$ by \eqref{bd:rho} could be chosen sufficiently small for $p\leq p_0$ by choosing $h$ small enough.
	By Lemma \ref{Le:du} we have for $0<s<\frac13$, $\kappa>0$ 
	\begin{align*}
	I_2\lesssim	\lambda \|Y^{p+1}_{\eps}\|_{B^{s,\eps}_{1,1}(\rho^{p-\kappa})}\|Z_{\eps}\|_{\bC^{-s,\eps}(\rho^\kappa)},
	\end{align*}
	and we use Lemma \ref{Le32}  to have
	\begin{align}\label{eq:bd1}
	\|Y^{p+1}_{\eps}\|_{B^{s,\eps}_{1,1}(\rho^{p-\kappa})}\lesssim \|Y^{p+1}_{\eps}\|_{B^{-s,\eps}_{1,1}(\rho^{p-2\kappa})}^{1/2}\|Y^{p+1}_{\eps}\|_{B^{3s,\eps}_{1,1}(\rho^{p})}^{1/2}.\end{align}
	For the first factor on the RHS of \eqref{eq:bd1} we use H\"older's inequality to have
	$$\|Y^{p+1}_{\eps}\rho^{p-2\kappa}\|_{B^{-s,\eps}_{1,1}}\lesssim\|Y^{p+1}_{\eps}\rho^{p-2\kappa}\|_{L^{1,\eps}}\lesssim\|Y^{p+2}_{\eps}\rho^{p}\|_{L^{1,\eps}}^{\frac{p+1}{p+2}}\|\rho^\theta\|_{L^{1,\eps}}^{\frac1{p+2}}\lesssim B^{\frac{p+1}{p+2}},$$
	where $\theta=(p-2\kappa)(p+2)-p(p+1)=(1-2\kappa)p-4\kappa>2/\delta$ for $\kappa>0$ small enough. For the second factor on the RHS of \eqref{eq:bd1} we use (iii) in Lemma \ref{lem:emb} to obtain
	\begin{align*}
	\|Y^{p+1}_{\eps}\|_{B^{3s,\eps}_{1,1}(\rho^{p})} & \lesssim \|Y^{p+1}_{\eps}\|_{L^{1,\eps}(\rho^{p})}+\|\nabla_\eps Y_\eps^{p+1} \|_{L^{1,\eps}(\rho^{p})}
	\\&\lesssim \|Y^{p+1}_{\eps}\|_{L^{1,\eps}(\rho^{p})}+(D_1^{1/2}+D_2^{1/2})B^{1/2},
	\end{align*}
	where we use \eqref{bd:rho} to have
	\begin{equs}[eq:z]
	\|\nabla_\eps & Y_\eps^{p+1} \|_{L^{1,\eps}(\rho^{p})}\lesssim \sum_{i=1}^2\||\nabla_\eps^i Y_\eps|(Y_\eps^{p}(\cdot+\eps e_i)+Y_\eps^{p}) \rho^p\|_{L^{1,\eps}}\\
	&\lesssim D_2^{1/2}\Big(\sum_{i=1}^2\| |Y_\eps|^{p+2}(\cdot+\eps e_i) \rho^p\|_{L^{1,\eps}}\Big)^{1/2}+D_1^{1/2}B^{1/2}\lesssim(D_1^{1/2}+D_2^{1/2})B^{1/2},
	\end{equs}
	in the last step. Substituting the above calculations into \eqref{eq:bd1} and using Young's inequality
	 we obtain for $\gamma>0$
	\begin{align*}
	I_2 &\lesssim	\lambda (B^{\frac{p+1}{p+2}}+(D_1^{\frac14}+D_2^{\frac14})B^{\frac14+\frac{p+1}{2(p+2)}})\|Z_{\eps}\|_{\bC^{-s,\eps}(\rho^\kappa)}
	\\ & \leq \gamma\lambda B+\lambda C_\gamma\|Z_{\eps}\|_{\bC^{-s,\eps}(\rho^\kappa)}^{p+2}
	+\gamma D_1+\gamma D_2+\lambda^{2+p/2} C_\gamma\|Z_{\eps}\|_{\bC^{-s,\eps}(\rho^\kappa)}^{2(p+2)}.
	\end{align*}
	Similarly, for $I_3$ we use Lemma \ref{Le:du}  to have
	\begin{align*}
	I_3\lesssim	\lambda \|Y^{p}_{\eps}\|_{B^{s,\eps}_{1,1}(\rho^{p-\kappa})}\|\Wick{Z_{\eps}^2}\|_{\bC^{-s,\eps}(\rho^\kappa)},
	\end{align*}
	and we apply Lemma \ref{Le32} to obtain for $0<s<\frac13$
	\begin{align}\label{eq:bd2}
	\|Y^{p}_{\eps}\|_{B^{s,\eps}_{1,1}(\rho^{p-\kappa})}\lesssim \|Y^{p}_{\eps}\|_{B^{-s,\eps}_{1,1}(\rho^{p-2\kappa})}^{1/2}\|Y^{p}_{\eps}\|_{B^{3s,\eps}_{1,1}(\rho^{p})}^{1/2}.\end{align}
	For the first factor on the RHS of \eqref{eq:bd2} we use H\"older's inequality to have
	$$\|Y^{p}_{\eps}\rho^{p-2\kappa}\|_{B^{-s,\eps}_{1,1}}\lesssim\|Y^{p}_{\eps}\rho^{p-2\kappa}\|_{L^{1,\eps}}\lesssim\|Y^{p+2}_{\eps}\rho^{p}\|_{L^{1,\eps}}^{\frac{p}{p+2}}\|\rho^{\theta_1}\|_{L^{1,\eps}}^{\frac2{p+2}}\lesssim B^{\frac{p}{p+2}} ,$$
	where $\theta_1=p(1-\kappa)-2\kappa>2/\delta$ for $\kappa>0$ small enough. For the second factor on the RHS of \eqref{eq:bd2} we use (iii) in Lemma \ref{lem:emb} to obtain
	\begin{align*}
	\|Y^{p}_{\eps}\|_{B^{3s,\eps}_{1,1}(\rho^{p})}&\lesssim \|Y^{p}_{\eps}\|_{L^{1,\eps}(\rho^{p})}+\|\nabla_\eps Y_\eps^p \|_{L^{1,\eps}(\rho^{p})}\no
	\\&\lesssim \|Y^{p}_{\eps}\|_{L^{1,\eps}(\rho^{p})}+(D_1^{1/2}+D_2^{1/2})\|Y^{p}_{\eps}\|_{L^{1,\eps}(\rho^{p})}^{1/2}
	\\&\lesssim B^{\frac{p}{p+2}}+(D_1^{1/2}+D_2^{1/2})B^{\frac{p}{2(p+2)}},
	\end{align*}
	where we use a similar calculation as \eqref{eq:z} in the second step.
	Combining the above calculations and Young's inequality we obtain
	\begin{align*}
	I_3& \lesssim\lambda (B^{\frac{p}{p+2}}+(D_1^{\frac14}+D_2^{\frac14})B^{\frac{3p}{4(p+2)}})\|\Wick{Z_{\eps}^2}\|_{\bC^{-s,\eps}(\rho^\kappa)}
	\\&\leq \gamma\lambda B+\lambda C_\gamma\|\Wick{Z_{\eps}^2}\|_{\bC^{-s,\eps}(\rho^\kappa)}^{\frac{p+2}2}
	+\gamma D_1+\gamma D_2+ C_\gamma\lambda^{\frac{p}6+\frac43}\|\Wick{Z_{\eps}^2}\|_{\bC^{-s,\eps}(\rho^\kappa)}^{\frac{2(p+2)}3}.
	\end{align*}
	Similarly,
	\begin{align*}
	I_4&\lesssim	\lambda (B^{\frac{p-1}{p+2}}+(D_1^{\frac14}+D_2^{\frac14})B^{\frac{3p-4}{4(p+2)}})\|\Wick{Z_{\eps}^3}\|_{\bC^{-s,\eps}(\rho^\kappa)}
	\\&\leq \gamma\lambda B+\lambda C_\gamma\|\Wick{Z_{\eps}^3}\|_{\bC^{-s,\eps}(\rho^\kappa)}^{\frac{p+2}3}
	+\gamma D_1+\gamma D_2+ C_\gamma\lambda^{\frac{p+12}{10}}\|\Wick{Z_{\eps}^3}\|_{\bC^{-s,\eps}(\rho^\kappa)}^{\frac{2(p+2)}5}.
	\end{align*}
	Combining the above estimates and choosing $\gamma$ small enough to absorb $\gamma D_1, \gamma D_2$ and $\gamma \lambda B$ to the LHS of \eqref{eq:bd3}  we obtain
	\begin{equation}\label{eq:Lpp}
	\aligned
	\frac1p \frac{\dif}{\dif t}\|Y_{\eps}^p\rho^p\|_{L^{1,\eps}}+\frac14 D_1 & +\frac14 D_2+\frac{m}2\|Y_{\eps}^p\rho^p\|_{L^{1,\eps}}^2+\frac\lambda4\|Y_{\eps}^{p+2}\rho^p\|_{L^{1,\eps}} 	
	\\
	&\leq (\lambda+\lambda^{2+p/2})C(\mathbb{Z_\eps}),
	\endaligned
	\end{equation}
	where $C(\mathbb{Z_\eps})$ denotes a polynomial in the terms of
		$$\|Z_{M,\eps}\|_{\bC^{-s,\eps}(\rho^\kappa)},\qquad \|\Wick{Z_{M,\eps}^2}\|_{\bC^{-s,\eps}(\rho^\kappa)}, \qquad \|\Wick{Z_{M,\eps}^3}\|_{\bC^{-s,\eps}(\rho^\kappa)}$$ and  $\E C(\mathbb{Z_\eps})\lesssim1$ with the proportional constant independent of $\eps, M$.
	Recall that $(\Phi_{{\eps}}, Z_{\eps})$ are stationary process, so are $Y_{\eps}$.
Integrating over $t$ and taking expectation we obtain
	\begin{equation}\label{eq:Lp}
	\aligned
	\frac1p \E\|Y_{\eps}^p(t)\rho^p\|_{L^{1,\eps}}+\frac14\int_0^t\E(D_1+D_2)\dif s
	&+\frac{m}2\int_0^t\E\|Y_{\eps}^p\rho^p\|_{L^{1,\eps}}^2\dif s+\frac\lambda4\int_0^t\E\|Y_{\eps}^{p+2}\rho^p\|_{L^{1,\eps}}\dif s 	
	\\
	&\leq C_{m,\rho}(\lambda+\lambda^{2+p/2})+\frac1p \E\|Y_{\eps}^p(0)\rho^p\|_{L^{1,\eps}},
	\endaligned
	\end{equation}
	with the constant $C_{m,\rho}\geq0$.
Using stationarity of $Y_\eps$  and dividing $t$ on the both side of \eqref{eq:Lp} we obtain
		\begin{equation*}
	\aligned&\frac14\E\||\nabla_\eps Y_\eps(0)|^2 |Y_\eps(0)|^{p-2}\rho^p\|_{L^{1,\eps}}
	+\frac{m}2\E\|Y_{\eps}^p(0)\rho^p\|_{L^{1,\eps}}^2
	\\&=
\frac1{4t}\int_0^t\E\||\nabla_\eps Y_\eps|^2 |Y_\eps|^{p-2}\rho^p\|_{L^{1,\eps}}\dif s
	+\frac{m}{2t}\int_0^t\E\|Y_{\eps}^p\rho^p\|_{L^{1,\eps}}^2\dif s
	\\
	&\leq C_{m,\rho}(\lambda+\lambda^{2+p/2})+\frac1{pt} \E\|Y_{\eps}^p(0)\rho^p\|_{L^{1,\eps}}.
	\endaligned
	\end{equation*}
	The last term can be absorbed by the LHS for $t$ large enough. Then \eqref{eq:zmm1p} follows.
Moreover, by Lemma \ref{lem:emb}  we have for $\sigma\in (0,1)$ and any polynomial weight $\rho$
	\begin{align}\label{emb}
	\|Y_\eps\|_{B^{1-\sigma,\eps}_{2,2}(\rho)}\lesssim\|Y_\eps\|_{B^{-\sigma,\eps}_{2,2}(\rho)}+\|\nabla_\eps Y_\eps\|_{B^{-\sigma,\eps}_{2,2}(\rho)}\lesssim\|Y_\eps\|_{L^{2,\eps}(\rho)}+\|\nabla_\eps Y_\eps\|_{L^{2,\eps}(\rho)}.
	\end{align}
Choosing $p=2$,  replacing $\rho$ by $\rho^2$ with $\delta>\frac12$ in \eqref{eq:Lpp} and using \eqref{emb} we obtain for some $c_0>0$
	\begin{align*}
	\frac12\frac{\dif}{\dif t}\|Y_\eps(t)\rho^2\|_{L^{2,\eps}}^2
	+c_0\| Y_\eps\rho^2\|_{H^{1-\sigma,\eps}}^2
	& +\frac{m}4\|Y_\eps\rho^2\|_{L^{2,\eps}}^2+\frac{\lambda}{4}\|Y_\eps\rho\|_{L^{4,\eps}}^4
	\\&\leq (\lambda+\lambda^3)C(\mZ_\eps).
	\end{align*}
		Furthermore, we consider $\frac{\dif}{\dif t}\|Y_\eps\rho^2\|_{L^{2,\eps}}^{2p}$  and have for every $p>1$
	\begin{align*}
	\frac1{2p}\frac{\dif}{\dif t}\|Y_\eps\rho^2\|_{L^{2,\eps}}^{2p}&+c_0\| Y_\eps\rho^2\|_{H^{1-\sigma,\eps}}^2\|Y_\eps\rho^2\|_{L^{2,\eps}}^{2(p-1)}+\frac{m}4\|Y_\eps\rho^2\|_{L^{2,\eps}}^{2p}+\frac{\lambda}{4}\|Y_\eps\rho\|_{L^{4,\eps}}^4\|Y_\eps\rho^2\|_{L^{2,\eps}}^{2(p-1)}
	\\&\lesssim (\lambda+\lambda^3) C(\mathbb{Z_\eps})\|Y_\eps\rho^2\|_{L^{2,\eps}}^{2(p-1)}\lesssim (\lambda^{p}+\lambda^{3p})C(\mathbb{Z_\eps})+\frac{m}8\|Y_\eps\rho^2\|_{L^{2,\eps}}^{2p},
	\end{align*}
	which by taking stationary initial condition and similar argument as above gives \eqref{eq:zmm1}.
\end{proof}

By Lemma \ref{lem:estp} we could give  uniform control for the nonlinear terms on the right hand side of \eqref{eq:Y}.
Our aim is to do asymptotic expansion of $k$-point function  up to $\lambda^{N+1}$ for $k, N\in \mN$. In the following we fix a polynomial weight and several parameters
\begin{align}\label{eq:rho}\rho(x)=(1+|h x|^2)^{-\delta/2}, \quad \delta>1/2, \quad p_0>(9k(N+1))\vee (96), \quad p_0 \textrm{ even},\end{align}
with  $h>0$ small enough such that the bounds \eqref{eq:zmm1} holds for every $p\geq 1$ and \eqref{eq:zmm1p}  holds for $1\leq p\leq p_0$ and even $p$. In the subsequent sections we also consider functions on $\Lambda_\eps^k$ and we also abuse the notation $\rho$ to denote the weight in $\mR^{2k}$. In this case we say $\rho$ in \eqref{eq:rho} means that $h>0$ satisfies the same condition as in \eqref{eq:rho} and it is easy to find that $\rho$ could be controlled by  the corresponding weight in $\mR^2$. 

\bc\label{corl:1} For the polynomial weight and $p_0$ in \eqref{eq:rho}, it holds  that
for  $1\leq p\leq p_0/3$, $0<\beta<\frac{1}{{4p}}$, $\kappa>0$
\begin{align}\label{bd:zmm2}
\E\|Y_{M,\eps}\|_{\bC^{\beta,\eps}(\rho^{4+\kappa})}^p\lesssim\lambda^p+\lambda^{4p},
\end{align}
where the proportional constant is independent of $\eps,M$.
\ec
\begin{proof} We consider the case for $p$ even and the general case follows by H\"older's inequality. We also omit the subscript $M$ for notation's simplicity and all the proportional constants in the proof are independent of $\eps$ and $M$. We first  prove that
	for any $0<s<\frac{1}{2p}$ that
	\begin{equs}[bd:zmm1]
		\E\| Y_{M,\eps}^3\|_{L^{p,\eps}(\rho^3)}^p+\E\| Y_{M,\eps}^2Z_{M,\eps}\|_{B^{-s,\eps}_{1+s,\infty}(\rho^{4+\kappa})}^p&+\E\| Y_{M,\eps}\Wick{Z_{M,\eps}^2}\|_{H^{-s,\eps}(\rho^{2+\kappa})}^p
		\\
		&\lesssim \lambda+\lambda^{3p}.
	\end{equs}
	By Lemma \ref{lem:estp} we have for $3p\leq p_0$
	\begin{align*}
	\E\| Y_{\eps}^3\|_{L^{p,\eps}(\rho^3)}^p\lesssim\E\| Y_{\eps}\|_{L^{3p,\eps}(\rho)}^{3p} \lesssim\lambda+\lambda^{2+\frac{3p}2}.
	\end{align*}
	For the second term in \eqref{bd:zmm1} we use Lemma \ref{lem:pow} to have for $s>\gamma>0$
		\begin{align*}
	\|Y_{\eps}^2Z_{\eps}\|_{B^{-s,\eps}_{1+s,\infty}(\rho^{4+\kappa})}
	& \lesssim \|Z_\eps\|_{\bC^{-\gamma,\eps}(\rho^\kappa)}\|Y^2_\eps\|_{B^{s}_{1+s,\infty}(\rho^4)}\lesssim \|Z_\eps\|_{\bC^{-\gamma,\eps}(\rho^\kappa)}\|Y^2_\eps\|_{B^{3s,\eps}_{1,1}(\rho^4)}
	\\
	&\lesssim \|Z_\eps\rho^\kappa\|_{\bC^{-\gamma,\eps}}\|Y_\eps\rho^2\|_{H^{4s,\eps}}\|Y_\eps\rho^2\|_{L^{2,\eps}}
	\lesssim \|Z_\eps\|_{\bC^{-\gamma,\eps}(\rho^\kappa)}\|Y_\eps\rho^2\|_{H^{1-\sigma,\eps}}^{2\theta_1}\|Y_\eps\rho^2\|_{L^{2,\eps}}^{2-2\theta_1}
	\end{align*}
	where we choose $0<\sigma<1$ small enough such that $\theta_1=\frac{2s}{1-\sigma}<1$ and we use Lemma \ref{lem:emb} in the second inequality and Lemma \ref{Le32} in the last inequality.
Thus by \eqref{eq:zmm1} in Lemma \ref{lem:estp} and \eqref{bd:Z}	we obtain for $\sigma$ small satisfying $2sp<1-\sigma$, $$\E\|Y_{\eps}^2Z_{\eps}\|_{B^{-s,\eps}_{1+s,\infty}(\rho^{4+\kappa})}^p\lesssim \E\|Z_\eps\|_{\bC^{-\gamma,\eps}(\rho^\kappa)}^p\|Y_\eps\rho^2\|_{H^{1-\sigma,\eps}}^{2p\theta_1}\|Y_\eps\rho^2\|_{L^{2,\eps}}^{p(2-2\theta_1)}\lesssim \lambda^{p}+\lambda^{3p},$$
where we use H\"older's inequality in the last step.

	\noindent	For the third term in \eqref{bd:zmm1} we use Lemma \ref{lem:pow}, Lemma \ref{Le32}, \eqref{eq:zmm1}, \eqref{bd:Z} and H\"older's inequality to  have for $sp<2(1-\sigma)$
	\begin{align*}
	\E\| Y_{\eps}\Wick{Z_{\eps}^2}\|_{H^{-s,\eps}(\rho^{2+\kappa})}^p
	&\lesssim\E\| Y_{\eps}\|_{H^{s,\eps}(\rho^{2})}^p\|\Wick{Z_{\eps}^2}\|_{\bC^{-\gamma,\eps}(\rho^{\kappa})}^p
	\\
	&\lesssim\E\| Y_{\eps}\|_{H^{1-\sigma, \eps}(\rho^{2})}^{p\theta_1/2}\| Y_{\eps}\|_{L^{2, \eps}(\rho^{2})}^{p(1-\theta_1/2)}\|\Wick{Z_{\eps}^2}\|_{\bC^{-\gamma}(\rho^{\kappa})}^p\lesssim\lambda^{\frac{p}2}+\lambda^{\frac{3p}2}.
	\end{align*}
	Now  \eqref{bd:zmm1} follows.
	
	In the following we prove \eqref{bd:zmm2} by Schauder estimate Lemma \ref{lem:Sch}.
	Set $$F_\eps\eqdef Y_\eps^3+3Y_\eps^2Z_\eps+3Y_\eps\Wick{Z^2_\eps}+\Wick{Z^3_\eps}.$$
  We choose $\alpha,s>0$ such that
	\begin{align}\label{eq:alpha}
	\frac2p<\alpha<\frac{1}{4p}+\frac2p<\frac{s(1-s)}{1+s}+\frac2p<2,\end{align}
	where the last inequality holds by $sp<\frac12$.
		By Schauder's estimate Lemma \ref{lem:Sch}
	\begin{align*}
	\|Y_\eps\|_{L^p_TB^{\alpha,\eps}_{p,p}(\rho)}\lesssim \|Y_\eps(0)\|_{B^{\alpha-\frac2p,\eps}_{p,p}(\rho)}+\lambda\|F_\eps\|_{L_T^pB^{\alpha-2,\eps}_{p,p}(\rho)},
	\end{align*}
	with the proportional constant independent of $T$.
	By Lemma \ref{lem:emb} and the choice of $\alpha$ in \eqref{eq:alpha} we have the following embedding
	$$ B^{-s,\eps}_{1+s,\infty}(\rho^{4+\kappa})\subset B^{\alpha-2,\eps}_{p,p}(\rho^{4+\kappa}), \quad  H^{-s,\eps}(\rho^{2+\kappa})\subset B^{\alpha-2,\eps}_{p,p}(\rho^{4+\kappa}), \quad
	 L^{p,\eps}(\rho^{3})\subset B^{\alpha-2,\eps}_{p,p}(\rho^{4+\kappa}).$$
	Thus,
	\begin{align*}
\E\|Y_\eps(0)\|_{B^{\alpha,\eps}_{p,p}(\rho^{4+\kappa})}^p
&=\frac1T\E\|Y_\eps\|_{L^p_TB^{\alpha,\eps}_{p,p}(\rho^{4+\kappa})}^p\lesssim \frac{1}T\E\|Y_\eps(0)\|_{B^{\alpha-\frac2p,\eps}_{p,p}(\rho^{4+\kappa})}^p+\frac{\lambda^p}T\E\|F_\eps\|_{L_T^pB^{\alpha-2,\eps}_{p,p}(\rho^{4+\kappa})}^p
\\
&\lesssim\frac1T\E\|Y_\eps(0)\|_{B^{\alpha-\frac2p,\eps}_{p,p}(\rho^{4+\kappa})}^p+{\lambda^p}\E\| Y_{M,\eps}^3\|_{L^{p,\eps}(\rho^3)}^p+{\lambda^p}\E\| Y_{M,\eps}^2Z_{M,\eps}\|_{B^{-s,\eps}_{1+s,\infty}(\rho^{4+\kappa})}^p
\\&\qquad\qquad\qquad +{\lambda^p}\E\| Y_{M,\eps}\Wick{Z_{M,\eps}^2}\|_{H^{-s,\eps}(\rho^{2+\kappa})}^p+{\lambda^p}\E\|\Wick{Z_{M,\eps}^3}\|_{\bC^{-s,\eps}(\rho^{\kappa})}^p\\
&\lesssim\frac1T\E\|Y_\eps(0)\|_{B^{\alpha-\frac2p,\eps}_{p,p}(\rho^{4+\kappa})}^p+\lambda^p+\lambda^{4p},
	\end{align*}
	where we use \eqref{bd:zmm1} in the last step.
	Choosing $T$ large enough and using Besov embedding Lemma \ref{lem:emb},
	we obtain \eqref{bd:zmm2}.
\end{proof}

Define
$$\Phi_{{M,\eps}}=Y_{M,\eps}+Z_{M,\eps},\quad \Wick{\Phi^2_{{M,\eps}}}\eqdef \Phi^2_{M,\eps}-a_{M,\eps} =Y_{M,\eps}^2+2Y_{M,\eps}Z_{M,\eps}+\Wick{Z_{M,\eps}^2},$$
\begin{align}\label{def:Wick}\Wick{\Phi_{{M,\eps}}^3}\eqdef\Phi_{{M,\eps}}^3-3a_{M,\eps}\Phi_{{M,\eps}} =Y_{M,\eps}^3+3Z_{M,\eps}Y_{M,\eps}^2+ 3\Wick{Z^2_{{M,\eps}}}Y_{M,\eps}+\Wick{Z_{M,\eps}^3}.\end{align}
Then we can use the right hand sides to control $\E\|\Wick{\Phi^i_{M,\eps}}\|_{\bC^{-\gamma,\eps}(\rho^{9})}^p$ for $i=1,2,3$. More precisely we have the following result which will be useful in the subsequent sections.

\bc\label{co:zmm} For the polynomial weight and $p_0$  as in \eqref{eq:rho}, it holds  for  $0<\gamma<\frac12, \kappa>0, 1\leq p< p_0/9$ that
$$\E\|\Phi_{{M,\eps}}\|_{\bC^{-\gamma,\eps}(\rho^{4+\kappa})}^p+\E\|\Wick{\Phi_{{M,\eps}}^2}\|_{\bC^{-\gamma,\eps}(\rho^{9})}^p+\E\|\Wick{\Phi_{{M,\eps}}^3}\|_{\bC^{-\gamma,\eps}(\rho^{13})}^p\lesssim1+\lambda^{12p},$$
with the proportional constant independent of $M,\eps$.
\ec
\begin{proof} It suffices to consider the case that $p$ even since the general case follows by H\"older's inequality.
	In the following we consider each term separately and we choose $\kappa>0$ small enough.
	By  \eqref{bd:zmm2} we have for $\gamma>0, 0<\beta<\frac1{4p}$, $3p< p_0$
\begin{align}\label{est:Phi1}
\E\|\Phi_{{M,\eps}}\|_{\bC^{-\gamma,\eps}(\rho^{4+\kappa})}^p
\lesssim\E\|Y_{M,\eps}\|_{\bC^{\beta,\eps}(\rho^{4+\kappa})}^p+\E\|Z_{{M,\eps}}\|_{\bC^{-\gamma,\eps}(\rho^\kappa)}^p
\lesssim
1+\lambda^{4p}.
\end{align}
We use   Lemma \ref{lem:pow},   \eqref{bd:zmm2} and \eqref{bd:Z} to have for $0<\beta<\frac{1}{8p}$, $6p< p_0$
\begin{align*}
\E\|\Wick{\Phi_{{M,\eps}}^2}\|_{\bC^{-\gamma,\eps}(\rho^{9})}^p
&\lesssim \E\|Y_{M,\eps}\|_{\bC^{\beta,\eps}(\rho^{4+\kappa})}^{2p}+\E\|Y_{M,\eps}\|_{\bC^{\beta,\eps}(\rho^{4+\kappa})}^{p}\|Z_{{M,\eps}}\|_{\bC^{-\beta/2,\eps}(\rho^\kappa)}^p+\E\|Z_{{M,\eps}}\|_{\bC^{-\gamma,\eps}(\rho^\kappa)}^p
\\&\lesssim
1+\lambda^{8p},
\end{align*}
where we use H\"older's inequality in the last step.
Similarly for the last term we use Lemma \ref{lem:pow},   H\"older's inequality and \eqref{bd:zmm2} to have for $\gamma>0$ and $0<\beta<\frac{1}{12p}$, $9p< p_0$
\begin{align*}
\E\|\Wick{\Phi_{{M,\eps}}^3}\|_{\bC^{-\gamma,\eps}(\rho^{13})}^p
&\lesssim\E\|Y_{M,\eps}\|_{\bC^{\beta,\eps}(\rho^{4+\kappa})}^{3p}+\E\|Y_{M,\eps}\|_{\bC^{\beta,\eps}(\rho^{4+\kappa})}^{2p}\|Z_{{M,\eps}}\|_{\bC^{-\beta/2,\eps}(\rho^\kappa)}^p\\&+\E\|Y_{M,\eps}\|_{\bC^{\beta,\eps}(\rho^{4+\kappa})}^{p}\|\Wick{Z_{{M,\eps}}^2}\|_{\bC^{-\beta/2,\eps}(\rho^\kappa)}^p+\E\|\Wick{Z_{{M,\eps}}^3}\|_{\bC^{-\gamma,\eps}(\rho^\kappa)}^p
\lesssim
1+\lambda^{12p}.
\end{align*}
The corollary then follows from the above bounds.
\end{proof}

\br\label{re} The powers for the weight and $\lambda$ in Corollary \ref{co:zmm} are not optimal. If we use $L^{p,\eps}$-norm in \eqref{eq:zmm1p} to bound $\bC^{-\gamma,\eps}$ we have lower powers with the cost that $\gamma$ depends on $p_0$.
\er

We consider the tight limit of
$$(\cE^\eps\Phi_{{M,\eps}}(t),\cE^\eps Z_{M,\eps}(t),\cE^\eps Y_{M,\eps}(t))_{M,\eps}$$ and we denote by $(\Phi, Z,Y)$ the canonical representative of the random variables under consideration (i.e. the canonical process on the canonical probability space with the limiting measure).

\bt\label{lem:cm} For the polynomial weight and $p_0$  as in \eqref{eq:rho}, it holds that
the family
$$(\cE^\eps\Phi_{{M,\eps}}(t),\cE^\eps Z_{M,\eps}(t), \cE^\eps Y_{M,\eps}(t))_{M,\eps}$$ is tight in $(\bC^{-\gamma}(\rho^{4+2\kappa}), \bC^{-\gamma}(\rho^{2\kappa}), \bC^{\gamma}(\rho^{4+\kappa}))$ for $0<\gamma,\kappa<\frac18$. Moreover, the first marginal $\nu$ of every tight limit $\mu$ is an invariant measure of \eqref{eq:Phi1} and satisfies for every  $p\geq1$, $\sigma,\gamma>0$ with $3p<p_0$
$$\E^\nu\|\Phi\|_{\bC^{-\gamma}(\rho^{4+\kappa})}^{p}\lesssim 1+\lambda^{4p},\quad \E^{\mu}\|Y\|^2_{H^{1-\sigma}(\rho^2)}\lesssim \lambda^{2},\quad \E^{\mu}\|Y\|^p_{L^p(\rho)}\lesssim \lambda+\lambda^{2+\frac{p}2}.$$
\et
\begin{proof}
	By Lemma \ref{euniform} and \eqref{eq:zmm1} we know that for $p\geq1$, $Y_{M,\eps}=\Phi_{{M,\eps}}-Z_{M,\eps}$ satisfies
	\begin{align*}
	\E\|\cE^\eps Y_{M,\eps}\|_{L^2(\rho^2)}^{2p}+\E\|\cE^\eps Y_{M,\eps}\|_{H^{1-\sigma}(\rho^2)}^2\|\cE^\eps Y_{M,\eps}\|_{L^2(\rho^2)}^{2(p-1)}\lesssim \lambda^{p}+\lambda^{3p}.\end{align*}
By Corollary \ref{corl:1} and \eqref{est:Phi1} we have for $\gamma>0$, $p\geq 2$, $0<\beta<\frac{1}{4p}$, $3p< p_0$
	\begin{align*}
\E\|\cE^\eps Y_{M,\eps}\|_{\bC^{\beta}(\rho^{4+\kappa})}^p+	\E\|\cE^\eps \Phi_{M,\eps}\|_{\bC^{-\gamma}(\rho^{4+\kappa})}^p
\lesssim 1+\lambda^{4p}.
\end{align*}
		Then the tightness of $(\cE^\eps\Phi_{{M,\eps}}(t),\cE^\eps Z_{M,\eps}(t), \cE^\eps Y_{{M,\eps}}(t))_{M,\eps}$ in $(\bC^{-\gamma-\kappa}(\rho^{4+2\kappa}),\bC^{-\gamma}(\rho^{\kappa}),\bC^{\beta-\kappa}(\rho^{4+2\kappa}))$  follows from compact embedding Lemma \ref{lem:emb} and the moment bound follows from lower semicontinuity of the norm and \eqref{est:Phi1}, \eqref{eq:zmm1p}  and \eqref{eq:zmm1} with $p=1$.
We may  extract a converging subsequence, which is still denoted as $(\cE^\eps\Phi_{{M,\eps}}(t),\cE^\eps Z_{M,\eps}(t), \cE^\eps Y_{{M,\eps}}(t))_{M,\eps}$, modify the stochastic
basis, and find random variables $(\Phi,Z,Y)\in (\bC^{-\gamma}(\rho^{4+\kappa}),\bC^{-\gamma}(\rho^{\kappa}),\bC^{\beta}(\rho^{4+\kappa}))$ such that for $p\geq2$,   $0<\beta<\frac{1}{4p}$, $3p<p_0$
	\begin{align*}
	\E\|\cE^\eps \Phi_{{M,\eps}}- \Phi\|_{\bC^{-\gamma-\kappa}(\rho^{4+2\kappa})}^p\to 0,\quad  \E\|\cE^\eps Z_{{M,\eps}}- Z\|_{\bC^{-\gamma}(\rho^{\kappa})}^p\to 0, \quad M\to\infty, \eps\to 0,
	\end{align*}
	and
	\begin{align*}
	\E\|\cE^\eps Y_{{M,\eps}}- Y\|_{\bC^{\beta-\kappa}(\rho^{4+2\kappa})}^p\to 0, \quad M\to\infty, \eps\to 0.
	\end{align*}
Hence, by Lemma \ref{lem:Epro}, Lemma \ref{lem:pow} and \eqref{bd:zmm2}, \eqref{bd:Z}, \eqref{bd:Z1} we have for $\gamma>0$, $\frac{1}{8p}>\beta>0$, $3p<p_0$
\begin{align*}
&\E\|\cE^\eps (Y_{M,\eps}Z_{M,\eps})- YZ\|_{\bC^{-\gamma}(\rho^{4+3\kappa})}^p
\\
&\lesssim \E\|\cE^\eps (Y_{M,\eps}Z_{M,\eps})- \cE^\eps Y_{M,\eps}\cE^\eps Z_{M,\eps}\|_{\bC^{-\gamma}(\rho^{4+3\kappa})}^p+\E\| \cE^\eps Y_{M,\eps}\cE^\eps Z_{M,\eps}-YZ\|_{\bC^{-\gamma}(\rho^{4+3\kappa})}^p
\\
&\lesssim
o(\eps)+\E\|\cE^\eps Y_{M,\eps}-Y\|_{\bC^{\beta}(\rho^{4+2\kappa})}^p\|\cE^\eps Z_{M,\eps}\|_{\bC^{-\beta/2}(\rho^{\kappa})}^p+\E\|Y\|_{\bC^{\beta}(\rho^{4+2\kappa})}^p\|\cE^\eps Z_{M,\eps}-Z\|_{\bC^{-\beta/2}(\rho^{\kappa})}^p
\\
& \to 0, \quad \textrm{as } M\to\infty, \eps\to0.
\end{align*}
Similarly we have that  as $M\to\infty, \eps\to0$, for $9p<p_0$, $\gamma>0$
\begin{align}\label{eq:PhiY1}
\E\|\cE^\eps (Y_{M,\eps}^2)- Y^2\|_{\bC^{-\gamma}(\rho^{8+4\kappa})}^p+\E\|\cE^\eps (Y_{M,\eps}^3)- Y^3\|_{\bC^{-\gamma}(\rho^{12+6\kappa})}^p\to 0,
\end{align}
\begin{equs}[eq:PhiY]
\E\|\cE^\eps (Y_{M,\eps}^2Z_{M,\eps})- Y^2Z\|_{\bC^{-\gamma}(\rho^{8+5\kappa})}^p
+\E\|\cE^\eps (Y_{M,\eps}\Wick{Z_{M,\eps}^2})- Y\Wick{Z^2}\|_{\bC^{-\gamma}(\rho^{4+3\kappa})}^p\to 0,
\end{equs}
which implies that
\begin{align}\label{eq:Phic}
\E\|\cE^\eps \Wick{\Phi_{{M,\eps}}^2}- \Wick{\Phi^2}\|_{\bC^{-\gamma}(\rho^{8+4\kappa})}^p+\E\|\cE^\eps \Wick{\Phi_{{M,\eps}}^3}- \Wick{\Phi^3}\|_{\bC^{-\gamma}(\rho^{12+6\kappa})}^p\to 0, \quad M\to\infty, \eps\to 0,
\end{align}
where
$$ \Wick{\Phi^2}\eqdef Y^2+2YZ+\Wick{Z^2},$$
$$\Wick{\Phi^3}\eqdef Y^3+3ZY^2+ 3\Wick{Z^2}Y+\Wick{Z^3}.$$
Since $Y_{M,\eps}$ satisfies equation \eqref{eq:Y}, by Corollary \ref{co:zmm} it is easy to obtain for $0\leq s\leq t\leq T$
$$\E\|Y_{M,\eps}(t)-Y_{M,\eps}(s)\|_{\bC^{-2-\gamma,\eps}(\rho^{13})}^p\lesssim (1+\lambda^{12p})(t-s)^p,$$
which combined with \eqref{bd:Z}, \eqref{bd:Z1} implies that
the tightness of the stochastic process $(\cE^\eps \Phi_{M,\eps},\cE^\eps Y_{M,\eps},$ $\cE^\eps Z_{M,\eps})$ in $C_T\bC^{-2-2\gamma}(\rho^{13+\kappa})\times C_T\bC^{-2-2\gamma}(\rho^{13+\kappa}) \times C_T\bC^{-\gamma}(\rho^\kappa)$.
Since $\Phi_{{M,\eps}}$ satisfies \eqref{eq:Phi}, we know $\cE^\eps \Phi_{{M,\eps}}$ satisfies the following equation
$$\cE^\eps \Phi_{{M,\eps}}(t)=\cE^\eps Y_{{M,\eps}}(0)+\int_0^t (\Delta_\eps \cE^\eps Y_{{M,\eps}}+\cE^\eps \Wick{\Phi_{{M,\eps}}^3})\dif s+\cE^\eps Z_{M,\eps}(t).$$
For smooth test functions $\varphi$ it is each to check $\Delta_\eps \varphi\to \Delta \varphi$ for $\eps\to0$.
By the convergence in \eqref{eq:Phic}, the nonlinear term $\cE^\eps \Wick{\Phi^3_{M,\eps}}$ also converges in $L^p(\Omega,L_T^p\bC^{-\gamma}(\rho^{12+6\kappa}))$, which implies that
 the limiting process $\Phi$ satisfies equation \eqref{eq:Phi1}.
	As $\Phi_{{M,\eps}}$ is stationary solution to \eqref{eq:Phi},  $\nu$ is an invariant measure to \eqref{eq:Phi1}.
\end{proof}

\br By similar argument as in \cite{GH18a} we could prove that every $\nu$ is translation invariant and reflection positive and every $\nu$ satisfies an integration by parts formula, which has the same form as in \cite[Chapter 12]{MR887102}. As $\Phi$ is not a usual function, we will not use the integration by parts formula for $\nu$ directly but use the associated one for $\nu_{M,\eps}$ instead in the subsequent sections.
\er

\section{Integration by parts and graphs}\label{sec:3}

The purpose of this section is to study the perturbative expansion
for the $k$-point correlation $S_{\lambda,M,\eps}^k$
and apply the estimates from the previous section
to obtain a bound on the remainder of this expansion.
To this end we actually consider an equivalent expansion using integration by parts (``IBP'' for short in the sequel), i.e. Dyson-Schwinger equations. We also use the same notation as in Section  \ref{sec:A2}.

Let   $C_{\eps}(x) := \frac12(m - \Delta_\eps )^{-1}(x)$ be the Green's function of discrete Laplacian $\Delta_{\eps}$ on $\Lambda_{\eps}$ and $C_{M,\eps}$ is the periodic Green's function. Choosing $a_{M,\eps}=C_{M,\eps}(0)$ as the Wick constant, we recall
$$
\Wick{\Phi_{M,\eps}^3}=\Phi_{M,\eps}^3-3a_{M,\eps}\Phi_{\eps},\quad
\Wick{\Phi_{M,\eps}^2}=\Phi_{M,\eps}^2-a_{M,\eps}.
$$
Recall the following IBP formula with respect to $\nu_{M,\eps}$ 
 (see e.g. \cite[Sec.~6]{GH18a} or \cite[Appendix C]{SSZZ2d}) for $F (\Phi_{M,\eps})=f(\Phi_{M,\eps}(z_1),\dots,\Phi_{{M,\eps}}(z_n)), n\in \mN, z_i\in\Lambda_{M,\eps}$ with smooth $f:\mR^n\to\mR$ having polynomial growth first order derivative
\rmb{\begin{equation}\label{eq:ibp}
\aligned
 \E \Big(\frac{\delta F (\Phi_{M,\eps})}{\delta \Phi_{M,\eps}(z)}\Big)
=&\int\Big(\frac{\delta F (\Phi)}{\delta \Phi(z)}\Big)\nu_{M,\eps}(\dif \Phi)=\int F(\Phi)\Big(\frac{\delta V_{M,\eps} (\Phi)}{\delta \Phi(z)}\Big)\nu_{M,\eps}(\dif \Phi)
\\
=&2 \E \Big( F(\Phi_{M,\eps})(m-\Delta_\eps)\Phi_{M,\eps}(z) \Big)
+    \lambda
\E \Big( F(\Phi_{M,\eps}) \Wick{\Phi_{M,\eps}(z)^3} \Big) ,
\endaligned
\end{equation}}
where $V_{M,\eps}(\Phi)=\sum_{\Lambda_{M,\eps}}\Big[\frac\lambda4\Phi^4+(-\frac32\lambda a_{M,\eps}+m)\Phi^2+|\nabla_\eps\Phi|^2\Big]$ and
$$\frac{\delta F (\Phi_{M,\eps})}{\delta \Phi_{M,\eps}(z)}=\lim_{\eta\to0}\frac1\eta(F(\Phi_{M,\eps}+\eta \frac{e_z}{\eps^2})-F(\Phi_{M,\eps})),$$
for $e_z:\Lambda_{M,\eps}\to [0,1]$, $e_z(z)=1$, $e_z(y)=0$ for $y\neq z$. We 
 write \eqref{eq:ibp} in terms of Green's function $C_{M,\eps}$:
\begin{equation}\label{eq:IBP-C}
\aligned
\int_{\Lambda_{M,\eps}} C_{M,\eps}(x-z ) \E \Big(\frac{\delta F (\Phi_{M,\eps})}{\delta \Phi_{M,\eps}(z)}\Big) \dif z
&= \E \Big( \Phi_{M,\eps}(x) F(\Phi_{M,\eps})\Big)
\\& +    \lambda\int_{\Lambda_{M,\eps}} C_{M,\eps}(x-z )
\E \Big( F(\Phi_{M,\eps}) \Wick{\Phi_{M,\eps}(z)^3} \Big) \dif z
\endaligned
\end{equation}
for any $x\in {\Lambda_{M,\eps}}$   and we use $\int_{\Lambda_{M,\eps}}f(z)\dif z$ to denote $\eps^2\sum_{z\in \Lambda_{M,\eps}}f(z)$.

\subsection{$k$-point correlation}\label{sec:3.1}

Consider the $k$-point correlation given by
$$
S_{\lambda,M,\eps}^k (x_1,\dots,x_k) =\E \Big[ \prod_{i=1}^k\Phi_{M,\eps}(x_i)\Big]\in \bC^{-\gamma,\eps}(\rho^\ell)\;,
$$
for $\gamma>0$, some $\ell>0$ and the polynomial weight $\rho$.

In the following, let $\Phi$  denote the stationary solution to \eqref{eq:Phi1} with marginal distribution given by $\nu$ obtained in Theorem \ref{lem:cm}.
We  still use $(M,\eps)$ to denote the subsequence such that $\nu_{M,\eps}\circ (\cE^{\eps})^{-1}$ converge to $\nu$ weakly with $\cE^{\eps}$ being the extension operator given in \eqref{def:E}.
Define
\begin{align}\label{def:S}
\<S_\lambda^{\nu,k},\varphi\>\eqdef\lim_{\eps\to 0, M\to\infty}\E\int_{\mR^{2k}}\Big(\prod_{i=1}^k\cE^\eps\Phi_{{M,\eps}}(x_i)\Big)\varphi(x_1,\dots,x_k)\prod_{i=1}^k\dif x_i=\lim_{\eps\to 0, M\to\infty}\<{\cE}_k^\eps S^k_{\lambda,M,\eps},\varphi\>,
\end{align}
for $\varphi\in \mathcal{S}(\mR^{2k})$, where we use Fubini Theorem in the second equality. Recall ${\cE}_k^\eps$ is introduced in \eqref{def:Ek} in Appendix to extend functions on $\Lambda_{\eps}^k$ and by \eqref{einner} this coincides with the one given in \eqref{def:S1}. By Theorem \ref{lem:cm} and Lemmas \ref{lem:A1}, \ref{euniform} we know
\begin{align}\label{co:S}
\lim_{M\to\infty,\eps\to0}{\cE}_k^\eps S^k_{\lambda,M,\eps}=S_\lambda^{\nu,k} \quad \textrm{ in }\quad \bC^{-\gamma}(\rho^\ell),\end{align}
for $\gamma>0$, some $n>0$ and the polynomial weight $\rho$.

In the following we expand $S_{\lambda,M,\eps}^k$
by applying \eqref{eq:IBP-C} with
 suitable choices of test functions $F$.
 We introduce the following short-hand notation
\begin{align}\label{def:I}
\cI_\eps f(x)\eqdef \int_{\Lambda_{M,\eps}} C_{M,\eps}(x-y)f(y)\dif y.\end{align}

Before the proof, let's start by an example on how the iteration of this formula  gives  the desired  expansion  for $S_{\lambda,M,\eps}(x_0-y_0) = S^k_{\lambda,M,\eps}(x_0-y_0)$ with $k=2$, $x_0, y_0\in \Lambda_{M,\eps}$.
We first describe the initial steps of this iteration:

(1)  Taking $F(\Phi_{M,\eps}) = \Phi_{M,\eps}(y_0)$ in \eqref{eq:IBP-C} with $x=x_0$, we have
\begin{equ}[eq:DS1]
	C_{M,\eps}(x_0-y_0) = S_{\lambda,M,\eps}(x_0-y_0)
	+    \lambda 
	\E (\Phi_{M,\eps}(y_0) \cI_\eps(\Wick{\Phi_{M,\eps}^3})(x_0)).
\end{equ}
This gives us the leading order expansion of the form
 $S_{\lambda,M,\eps} = C_{M,\eps}+ O(\lambda)$
where $ O(\lambda)$ refers to the last term and we will prove in Proposition \ref{pro} below that $\bC^{-\gamma,\eps}(\rho^\ell)$-norm of this term is bounded by $\lambda(1+\lambda^{16})$ for $\gamma>0$, some $\ell>0$ and the polynomial weight $\rho$ in \eqref{eq:rho}.


(2) We can further expand the last term of \eqref{eq:DS1}: taking $F(\Phi_{M,\eps}) = \lambda 
\cI_\eps(\Wick{\Phi_{M,\eps}^3})(x_0)$ in \eqref{eq:IBP-C}
\begin{equs}
	3\lambda \int_{\Lambda_{M,\eps}} \!
	C_{M,\eps}(y_0 -z) C_{M,\eps}(x_0-z) \E &(\Wick{\Phi_{M,\eps}(z)^2})\,\dif z
	 =  \lambda  \E (\Phi_{M,\eps}(y_0) \cI_\eps(\Wick{\Phi_{M,\eps}^3})(x_0))
	\\
	&\quad +\lambda^2  \E \Big(\cI_\eps( \Wick{\Phi_{M,\eps}^3})(x_0)\cI_\eps (\Wick{\Phi_{M,\eps}^3})(y_0)\Big).	\label{eq:DS2}
\end{equs}
Substituting this into \eqref{eq:DS1} we see that in order to get the expansion up to the next order,
we need to apply IBP again to the LHS in \eqref{eq:DS2}.
Choose $F(\Phi_{M,\eps}) = \Phi_{M,\eps}(x)$.
The LHS of \eqref{eq:IBP-C} only gives a Wick constant which can be absorbed into the RHS, so we get
\begin{align}\label{eq:Ephi2}
0=\E \Big( \Wick{\Phi_{M,\eps}^2(x)}\Big)
+ \lambda 
\E \Big(  \Phi_{M,\eps}(x) \cI_\eps(\Wick{ \Phi_{M,\eps}^3})(x) \Big) ,
\end{align}
which implies that
\begin{equs}
	&3\lambda \int_{\Lambda_{M,\eps}}  C_{M,\eps}(y_0-z)C_{M,\eps}(x_0-z) \E (\Wick{\Phi_{M,\eps}(z)^2})\,\dif z
	\\=& -3\lambda^2 \int_{\Lambda_{M,\eps}}  C_{M,\eps}(y_0-z)C_{M,\eps}(x_0-z)  \E (\Phi_{M,\eps}(z) \cI_\eps(\Wick{\Phi_{M,\eps}^3})(z))\,\dif z.
\end{equs}
Therefore we have
\begin{equs}
	S_{\lambda,M,\eps}(x_0 & -y_0) = C_{{M,\eps}}(x_0-y_0)
	+\lambda^2  \E \Big( \cI_\eps(\Wick{\Phi_{{M,\eps}}^3})(x_0)\cI_\eps( \Wick{\Phi_{{M,\eps}}^3})(y_0)\Big)\,
	\\
	&+3\lambda^2 \int_{\Lambda_{M,\eps}}  \!\!\!\! C_{M,\eps}(y_0-z)C_{M,\eps}(x_0-z)  \E \Big(\Phi_{M,\eps}(z) \cI_\eps(\Wick{\Phi_{M,\eps}^3})(z)\Big)\,\dif z.
	\label{e:S2}
\end{equs}
This gives us the next order expansion  $S_{\lambda,M,\eps} = C_{M,\eps}+ O(\lambda^2)$,
i.e. the order $\lambda$ term is zero (see Proposition \ref{pro} below for the control of the term containing $\Phi_{{M,\eps}}$).

\begin{remark}
We will iteratively apply IBP as above to obtain higher order expansions.
Remark that alternatively,
if we  replace all the $\Phi_{M,\eps}$ in \eqref{e:S2}  by the Gaussian field $Z_{M,\eps}$, using
$$\E (Z_{M,\eps}(z)\cI_\eps(\Wick{Z_{M,\eps}^3}(z)))=0,$$ and
pretending that the error of this replacement is order $O(\lambda^3)$,
then we  can write $S_{\lambda,M,\eps}(x_0-y_0)$ as
\begin{equs}
	{} & C_{M,\eps}(x_0-y_0)
	+\lambda^2 
	\E \Big( \cI_\eps(\Wick{Z_{M,\eps}^3})(x_0) \cI_\eps(\Wick{Z_{M,\eps}^3})(y_0)\Big)\, +O(\lambda^3)\\
	&=C_{M,\eps}(x_0-y_0) +6\lambda^2 \int_{\Lambda_{M,\eps}\times \Lambda_{M,\eps}} C_{M,\eps}(y_0-w) C_{M,\eps}(x_0-z) C_{M,\eps}(z-w)^3\,\dif z\dif w +O(\lambda^3).		\label{e:S3}
\end{equs}
This is the expansion at one more order.
\end{remark}

\begin{remark}
The procedure illustrated above for $2$-point correlation
also applies to $k$-point correlations for general $k$. For example,
taking $F(\Phi_{M,\eps})=\prod_{i=2}^k\Phi_{M,\eps}(x_i)$ we have
\begin{equs}[eq:Phik]
{}&\sum_{i=2}^kC_{M,\eps}(x_1-x_i)S_{\lambda,M,\eps}^{k-2}(x_2,\dots,x_{i-1},x_{i+1},\dots ,x_k)
\\& =S^k_{\lambda,M,\eps} (x_1,\dots, x_k)+\lambda 
\E\Big(\prod_{i=2}^k\Phi_{M,\eps}(x_i)\cI_\eps(\Wick{\Phi_{M,\eps}^3})(x_1)\Big)\dif z.
\end{equs}

%
\end{remark}

\underline{\it Graphic notation.}
To  iterate the above procedure in a more systematic way, it will be convenient to introduce some graphic notation. We denote  $C_{M,\eps}$ by a line,
and $\Phi_{M,\eps}$ by a tiny wavy line.
Then one can write the computation \eqref{eq:DS1}--\eqref{e:S3} above graphically as
\begin{equs}[e:exH]
S^2_{\lambda,M,\eps}
&= \qquad
\begin{tikzpicture}[baseline=-5]
	\node[dot] (x0) at (0,0) {};
	\node[dot] (y0) at (1,0) {};
	\draw[Phi] (x0) -- ++(0.1,-0.25);
	\draw[Phi] (y0) -- ++(-0.1,-0.25);
\end{tikzpicture}
\\
&=  \qquad
\begin{tikzpicture}[baseline=0]
	\node[dot] (x0) at (0,0) {};
	\node[dot] (y0) at (1.5,0) {};
	\draw[C] (x0) to (y0);
\end{tikzpicture}
\qquad -\lambda  \quad
	\begin{tikzpicture}[baseline=-15]
	\node[dot] (x0) at (0,0) {};
	\node[dot] (z) at (0.7,-0.5) {};
	\node[dot] (y0) at (1.7,0) {};
	\draw[C,bend right=30] (x0) to (z);
	\draw[Phi] (z) -- ++(0.18,-0.18);\draw[Phi] (z) -- ++(-0.18,-0.18);\draw[Phi] (z) -- ++(0,-0.25);
	\draw[Phi] (y0) -- ++(-0.1,-0.25);
	\end{tikzpicture}
\\
&=\qquad
\begin{tikzpicture}[baseline=0]
	\node[dot] (x0) at (0,0) {};
	\node[dot] (y0) at (1.5,0) {};
	\draw[C] (x0) to (y0);
\end{tikzpicture}
\qquad -3\lambda \quad
	\begin{tikzpicture}[baseline=-20]
	\node[dot] (x0) at (0,0) {};
	\node[dot] (z) at (1,-0.8) {};
	\node[dot] (y0) at (2,0) {};
	\draw[C,bend right=30] (x0) to (z);
	\draw[C,bend left=30] (y0) to (z);
	\draw[Phi] (z) -- ++(0.18,-0.18);\draw[Phi] (z) -- ++(-0.18,-0.18);
	\end{tikzpicture}
\qquad + \lambda^2\quad
	\begin{tikzpicture}[baseline=-20]
	\node[dot] (x0) at (0,0) {};
	\node[dot] (z) at (0.5,-0.8) {};
	\node[dot] (y0) at (2,0) {};
	\node[dot] (w) at (1.5,-0.8) {};
	\draw[C,bend right=30] (x0) to (z);
	\draw[C,bend left=30] (y0) to (w);
	\draw[Phi] (z) -- ++(0.18,-0.18);\draw[Phi] (z) -- ++(-0.18,-0.18);\draw[Phi] (z) -- ++(0,-0.25);
	\draw[Phi] (w) -- ++(0.18,-0.18);\draw[Phi] (w) -- ++(-0.18,-0.18);\draw[Phi] (w) -- ++(0,-0.25);
	\end{tikzpicture}
\\
&=\qquad
\begin{tikzpicture}[baseline=0]
	\node[dot] (x0) at (0,0) {};
	\node[dot] (y0) at (1.5,0) {};
	\draw[C] (x0) to (y0);
\end{tikzpicture}
\qquad + \lambda^2\quad
	\begin{tikzpicture}[baseline=-20]
	\node[dot] (x0) at (0,0) {};
	\node[dot] (z) at (0.5,-0.8) {};
	\node[dot] (y0) at (2,0) {};
	\node[dot] (w) at (1.5,-0.8) {};
	\draw[C,bend right=30] (x0) to (z);
	\draw[C,bend left=30] (y0) to (w);
	\draw[Phi] (z) -- ++(0.18,-0.18);\draw[Phi] (z) -- ++(-0.18,-0.18);\draw[Phi] (z) -- ++(0,-0.25);
	\draw[Phi] (w) -- ++(0.18,-0.18);\draw[Phi] (w) -- ++(-0.18,-0.18);\draw[Phi] (w) -- ++(0,-0.25);
	\end{tikzpicture}
	\qquad
	+3\lambda^2\quad
	\begin{tikzpicture}[baseline=-20]
	\node[dot] (x0) at (0,0) {};
	\node[dot] (z) at (1,0) {};
	\node[dot] (y0) at (2,0) {};
	\node[dot] (z1) at (1,-1) {};
	\draw[C,bend right=30] (x0) to (z);
	\draw[C,bend left=30] (y0) to (z);
	\draw[C,bend left=30] (z) to (z1);
	\draw[Phi] (z1) -- ++(0.18,-0.18);\draw[Phi] (z1) -- ++(-0.18,-0.18);\draw[Phi] (z1) -- ++(0,-0.25);
	\draw[Phi] (z) -- ++(-0.1,-0.25);
	\end{tikzpicture}
\\
&=\qquad
\begin{tikzpicture}[baseline=0]
	\node[dot] (x0) at (0,0) {};
	\node[dot] (y0) at (1.5,0) {};
	\draw[C] (x0) to (y0);
\end{tikzpicture}
\qquad +6 \lambda^2\quad
	\begin{tikzpicture}[baseline=-15]
	\node[dot] (x0) at (0,0) {};
	\node[dot] (z) at (0.5,-0.3) {};
	\node[dot] (y0) at (2,0) {};
	\node[dot] (w) at (1.5,-0.3) {};
	\draw[C,bend right=30] (x0) to (z);
	\draw[C,bend left=30] (y0) to (w);
	\draw[C,bend left=60] (z) to (w);
	\draw[C,bend right=60] (z) to (w);
	\draw[C] (z) to (w);
	\end{tikzpicture}
\qquad  +O(\lambda^3)
\end{equs}

\begin{remark}\label{rmk:IBP-graph}
A very helpful intuition is that
at each step of IBP,  we simply pick up a point which has a  wavy line (which corresponds to $x$ in \eqref{eq:IBP-C}),
and then connect it to the other existing points with a  wavy line (LHS of \eqref{eq:IBP-C}),
and also connect it to a new point with $3$ new wavy lines (namely
create a new factor $\cI_\eps(\Wick{\Phi_{M,\eps}^3})$ corresponding to the last term of \eqref{eq:IBP-C}).

For instance, \eqref{eq:Phik} amounts to picking up the point $x_1$, or graphically:
\begin{equ}[e:graph-4pt]
S^4_{\lambda,M,\eps}
= \;
\begin{tikzpicture}[baseline=-15]
	\node[dot] (x1) at (0,0) {};\node at (-0.3,0) {$x_1$};
	\node[dot] (x2) at (1,0) {};
	\node[dot] (x3) at (0,-1) {};
	\node[dot] (x4) at (1,-1) {};
	\draw[Phi] (x1) -- ++(0.1,-0.25);
	\draw[Phi] (x2) -- ++(-0.1,-0.25);
	\draw[Phi] (x3) -- ++(0.1,0.25);
	\draw[Phi] (x4) -- ++(-0.1,0.25);
\end{tikzpicture}
\quad =\;
\begin{tikzpicture}[baseline=-15]
	\node[dot] (x1) at (0,0) {};\node at (-0.3,0) {$x_1$};
	\node[dot] (x2) at (1,0) {};
	\node[dot] (x3) at (0,-1) {};
	\node[dot] (x4) at (1,-1) {};
	\draw[C] (x1) -- (x2);
	\draw[Phi] (x3) -- ++(0.1,0.25);
	\draw[Phi] (x4) -- ++(-0.1,0.25);
\end{tikzpicture}
\quad +\;
\begin{tikzpicture}[baseline=-15]
	\node[dot] (x1) at (0,0) {};\node at (-0.3,0) {$x_1$};
	\node[dot] (x2) at (1,0) {};
	\node[dot] (x3) at (0,-1) {};
	\node[dot] (x4) at (1,-1) {};
	\draw[C] (x1) -- (x3);
	\draw[Phi] (x2) -- ++(-0.1,-0.25);
	\draw[Phi] (x4) -- ++(-0.1,0.25);
\end{tikzpicture}
\quad +\;
\begin{tikzpicture}[baseline=-15]
	\node[dot] (x1) at (0,0) {};\node at (-0.3,0) {$x_1$};
	\node[dot] (x2) at (1,0) {};
	\node[dot] (x3) at (0,-1) {};
	\node[dot] (x4) at (1,-1) {};
	\draw[C] (x1) -- (x4);
	\draw[Phi] (x2) -- ++(-0.1,-0.25);
	\draw[Phi] (x3) -- ++(0.1,0.25);
\end{tikzpicture}
\quad -\lambda
\begin{tikzpicture}[baseline=-15]
	\node[dot] (x1) at (0,0) {};\node at (-0.3,0) {$x_1$};
	\node[dot] (x2) at (1,0) {};
	\node[dot] (x3) at (0,-1) {};
	\node[dot] (x4) at (1,-1) {};
	\node[dot] (z) at (.5,-.3) {};
	\draw[C] (x1) -- (z);
		\draw[Phi] (z) -- ++(0.18,-0.18);\draw[Phi] (z) -- ++(-0.18,-0.18);\draw[Phi] (z) -- ++(0,-0.25);
	\draw[Phi] (x2) -- ++(-0.1,-0.25);
	\draw[Phi] (x3) -- ++(0.1,0.25);
	\draw[Phi] (x4) -- ++(-0.1,0.25);
\end{tikzpicture}
\end{equ}

\end{remark}

We also find that for $\E[\Wick{\Phi_{M,\eps}^2(x)}]$, \eqref{eq:Ephi2} holds and for $\E[\Wick{\Phi_{M,\eps}^3(x)}]$ we use \eqref{eq:IBP-C} with $F(\Phi_{{M,\eps}})=\Wick{\Phi_{{M,\eps}}^2(x)}$ to have
$$0=\E[\Wick{\Phi_{M,\eps}^3(x)}]+\lambda \E[\Wick{\Phi^2_{M,\eps}(x)}\cI_\eps(\Wick{\Phi^3_{M,\eps}})(x)].$$
Then there is no need to connect $x$ to itself.

In general,
given a graph $G$, we write $G=(V_G,E_G)$ or simply $G=(V,E)$ where $V$ is the set of vertices
and $E$ is the set of edges.
We denote by $|V|$, $|E|$ the cardinalities of these sets,
namely the number of vertices and edges. Here for any two distinct vertices $u,v\in V$,
we allow  multiple edges between $u$ and $v$ (namely we allow `multigraphs' in the language of graph theory). However, we will assume
throughout the paper that our graphs do not have self-loops, i.e. there is not any edge of the form $\{u,u\}$ for $u\in V$.

\begin{definition}\label{def:nPhik}
	\rmb{For each $\ell\ge 0$,
	we define $\mathcal H_\ell^k$ to be the set of all the graphs $G=(V,E)$ such that
	$|V|=\ell+k$ which has  $k$ ``special points'' $\{u^*_m,m=1,\dots,k\}$ in $V$
	with $\deg(u^*_m)\in\{0,1\}$,
	and such that $\deg(v)\in \{1,2,3,4\}$ for every $v\in V \backslash \{u^*_m,m=1,\dots,k\}$.}
	We then define
	$
	\mathcal H := \cup_{\ell\ge 0} \mathcal H_\ell^k$.

	%
	We  then define $\mathcal G^k_\ell$ to be the set of all the graphs $G=(V,E)\in \mathcal H^k_\ell$ such that
	$\deg(v)=4$  for every $v\in V \backslash \{u^*_m,m=1,\dots,k\}$ and  $\deg(u_m^*)=1$ for $m=1,\dots,k$.
	We then write
	$ \mathcal G := \cup_{\ell\ge 0} \mathcal G^k_\ell$. Clearly,
	for $G\in \mathcal H_\ell^k$ if we write
	$$n_{\Phi}(G) := 4\ell+k - \sum_{v\in V} \deg(v),$$
	then
	$ \mathcal G = \{G\in  \mathcal H: n_{\Phi}(G)=0\} \subset \mathcal H$. 
	
	For any such graph we will write
	$V^\partial_G = \{u^*_m,m=1,\dots,k\}$ and
	$V^0_G= V_G \backslash V^\partial_G$. 
\end{definition}

Remark that
the graphs in \eqref{e:exH} are all elements of $\cH$,
where those without any tiny wavy line  are  elements of $\cG$.

We define a mapping from $\mathcal H_\ell^k$ to the set of all functions in $\{x_{u^*_m}\}_{m=1}^k$,
which maps
$G=(V,E)\in \mathcal H_\ell^k$ to
\begin{align}\label{def:IGk}
\mathbf{I}_G (x_1,\dots,x_{k})=\int& \Big(\!\!\!\prod_{\{u,v\}\in E_G}\!\!\!\! C_{M,\eps}(x_u,x_v)\Big) \,\no
\\&\E \Big(\prod_{u_m^*\in V^\partial_G}  \Phi_{M,\eps}(x_m)^{1-\deg(u_m^*)}
\prod_{z\in V^0_G} \Wick{\Phi_{M,\eps}(x_z)^{4-\deg(z)}}
\Big)
\prod_{z\in V^0_G} \dif x_z,
\end{align}
where $x_m=x_{u_m^*}$.
In particular when
$ G \in \mathcal G_\ell^k$,
$\mathbf{I}_G $ only depends on $C_{M,\eps}$, and not on $\Phi_{M,\eps}$ (which is then the usual convention for ``Feynman diagrams'' in physics).

We will sometimes say that $G$ is the graph associated with the function $\mathbf{I}_G$.
More generally when a function $F$ is a linear combination of
functions of $\mathbf{I}_G$ ($G\in \cH'$) for a finite collection of graphs $\cH' \subset \cH$
we say that  the graphs in $\cH'$ are the graphs  associated with the  function $F$. Whenever it is clear from the context we will sometimes use the same notation (such as $x$) to denote both a vertex of a graph and the point in $\R^2$ that is parametrized by $x$.

\begin{lemma}\label{lem:FRk}
For any $k \ge 1$ and $N \ge -1$,
	we have the following representation for the $k$-point correlation
	\begin{equ}[e:S-exp-DSk]
		S_{\lambda,M,\eps}^k  = \sum_{n=0}^N  \frac{\lambda^n}{n!} F_{n,M,\eps}^k + \lambda^{N+1} R_{N+1,M,\eps}^k
	\end{equ}
	where the graphs associated with $F_{n,M,\eps}^k$ belong to $\cG_n^k$,
	and   the graphs associated with $ R_{N+1,M,\eps}^k$ belong to $\cH_{N+1}^k$. 
	The functions
	$F_{n,M,\eps}^k $ and  $R_{N+1,M,\eps}^k$ are independent of $\lambda$.
\end{lemma}

We remark that in the lemma, when $N=-1$,
in which case the `empty' sum in  \eqref{e:S-exp-DSk}
is understood as $0$ by standard convention,
 \eqref{e:S-exp-DSk}  trivially holds
which states that  $S_{\lambda,M,\eps}^k = R_{0,M,\eps}^k$
where $R_{0,M,\eps}^k$ can be indeed associated with
a graph in $\cH_0^k$, that is, the graph with only $k$ vertices and no edge.
We also note that as an example the third line of
\eqref{e:exH} shows that  \eqref{e:S-exp-DSk}  holds for $k=2$ and $N=1$,
where $F_{0,M,\eps} = C_{M,\eps}$ and $F_{1,M,\eps} =0$.

\begin{proof}[Proof of Lemma \ref{lem:FRk}]
	In the proof we omit $M,\eps$ for notation's simplicity.	
By the $\Phi\to -\Phi$ symmetry,  $S_{\lambda,M,\eps}^k=0$ if $k$ is odd, so nothing needs to be proven.
	For $k\in 2\mathbb{N}$, we will prove   by induction in $N$ that the lemma holds with $R_{N+1}^k$ having   the following form
	\begin{equ}[e:RNk1]
		R_{N+1}^k = \sum_{\substack{G\in \mathcal H_{N+1}^k \\ n_\Phi(G) \in [0,m]\cap 2\mathbb{Z}}}
		\!\!\!\! r_G \mathbf{I}_G
	\end{equ}
	for  some $m \in 2\mathbb{Z}$ which may depend on $N$,
	and some coefficients $r_G\in \R$.

As remarked above, the lemma holds with $N=-1$, with $R_0^k = S^k_\lambda = \mathbf{I}_G$,
where $G=(V,E)\in \cH_0^k$ with $ V$ consisting of the $k$ special points only and $E$ being empty. 

Assume that for a fixed integer $N\ge 0$
we have already shown that
	\begin{equ}[e:RNk11]
		S_{\lambda}^k  = \sum_{n=0}^{N-1}  \frac{\lambda^n}{n!} F_{n}^k + \lambda^{N} R_{N}^k,
		\qquad
			R_{N}^k = \sum_{\substack{G\in \mathcal H_{N}^k \\ n_\Phi(G) \in [0,m]\cap 2\mathbb{Z}}}
		\!\!\!\! r_G \mathbf{I}_G
	\end{equ}
we then prove that the same holds with $N$ replaced by $N+1$, with updated values of $m$ and $r_G$.

To this end  we use IBP to decrease the value of $ n_\Phi(G)$ for the graphs $G\in \mathcal H_{N}^k$ above,
	so that we are only left with graphs in $\mathcal G_N^k$ (i.e. $n_\Phi=0$),
at the cost of producing other graphs in $\cH_{N+1}^k$
which however are multiplied by the parameter
 $\lambda$ and which will be defined as $R_{N+1}^k$.
	
	More precisely we  claim that for each $G\in \mathcal H_N^k$ such that
	$n_\Phi (G) =m$, the term $\mathbf I_G$ in \eqref{e:RNk11}
	can be written as
	\begin{equ}[e:I_Gk]
		\mathbf I_G
		= \sum_{G'\in \mathcal H_N^k, \, n_\Phi(G') =m-2}
		\!\!\!\! a_{G'} \mathbf{I}_{G'}
		\quad
		+\lambda \sum_{G''\in \mathcal H_{N+1}^k }
		b_{G''} \mathbf{I}_{G''}
	\end{equ}
	for some $a_{G'},b_{G''} \in \R$.
	
	Assuming \eqref{e:I_Gk}, we plug all these $\mathbf I_G$ with $n_\Phi (G) =m$
	into \eqref{e:RNk11}, 
	which allows us to write
	\begin{equ}[e:RN2k]
		R_N^k = \sum_{\substack{G\in \mathcal H_N^k \\ n_\Phi(G) \in [0,m-2]\cap 2\mathbb{Z}}}
		\!\!\!\!\bar{r}_G \mathbf{I}_G
		\quad
		+ \lambda \sum_{G\in \mathcal H_{N+1}^k }
		\bar{r}_G \mathbf{I}_G
	\end{equ}
	for  some coefficients $\bar{r}_G\in \R$.
	We then iterate this argument and  plug these into \eqref{e:RN2k}
	to obtain \eqref{e:RN2k} with even lower value of $m$
	(and new coefficients $\bar{r}_G\in \R$).
	After $m/2$ iterations,
	we get
	\begin{equ}[eq:RNk]
		R_N^k = \sum_{G\in \mathcal G_N^k}
		r'_G \mathbf{I}_G
		\quad
		+ \lambda \sum_{G\in \mathcal H_{N+1}^k }
		r'_G \mathbf{I}_G
	\end{equ}
	for some $r'_G$.
	We then see that  the lemma holds for $N$ with
	$F_1,\cdots ,F_{N-1}$ remaining the same,
	$$
	F_N = N! \sum_{G\in \mathcal G_N^k}
	r'_G \mathbf{I}_G
	\qquad \mbox{and}\qquad
	R_{N+1}^k =\sum_{G\in \mathcal H_{N+1}^k }
	r'_G \mathbf{I}_G.
	$$
	Since all the graphs in the last sum
	are finite graphs, $R_{N+1}^k$
	has the form
	\eqref{e:RNk1}  
	for some $m$.

	It remains to prove \eqref{e:I_Gk}, for which
it will be helpful to recall the intuition explained in Remark~\ref{rmk:IBP-graph}.
	
	Fixing $G$, we can assume that $n_\Phi(G) \ge 2$ (otherwise
	if $n_\Phi(G)  =0$ 
	 there is nothing to prove).
	 We then know that there is a vertex $x\in V^0_G$ with $
	\deg(x)<4$ or $x\in  V^\partial_G$ with $\deg(x)=0$. Fix any such vertex $x$.
	Applying \eqref{eq:IBP-C} with $x$ therein being this vertex $x$,
	we have
	\begin{align}\label{eq:IGk}
	\mathbf I_G
	=\Big(
	\sum_{\substack{z\in V^\partial_G \\ \deg z=0, z\neq x} }(1-\deg(z)) \, \mathbf I_{G_{xz}^-}
	\Big) +
	\Big(
	\sum_{\substack{z\in V^0_G \\ \deg z<4, z\neq x} }(4-\deg(z)) \, \mathbf I_{G_{xz}^-}
	\Big)
	- \lambda \mathbf I_{G_{xz}^+}
	\end{align}
	where the graph $G_{xz}^+$ is defined by ``adding a new vertex $z$'':
	$$
	V_{G_{xz}^+} = V_G \cup \{z\},
	\qquad
	E_{G_{xz}^+} = E_G \cup \{\{x,z\}\},
	\qquad
	(z\notin V_G)
	$$
	and the graph $G_{xz}^-$ is defined by  ``connecting $x$ and some $z\in V_G$
	by a new edge''
	$$
	V_{G_{xz}^-} = V_G,
	\qquad
	E_{G_{xz}^-} = E_G \cup \{\{x,z\}\},
	\qquad
	(z\in V_G).
	$$
	This is precisely in the form \eqref{e:I_Gk},
	since $n_\Phi(G_{xz}^-) = n_\Phi(G)-2$.
\end{proof}

We remark that from the above proof,
it is clear that all
the graphs associated with $F_{n,M,\eps}^k$ and $ R_{N+1,M,\eps}^k$  are obtained by the procedure
described below \eqref{eq:IGk},
starting from the `initial graph' with $|V|=k$ and $|E|=0$.
For convenience of estimating the graphs later,
we color the edges of these graphs as follows.
In the above proof, when we use \eqref{eq:IGk} each time,
\begin{itemize}
\item	
	we use red color for the new line $\{x,z\}$ appearing in $G_{xz}^+$;
\item
we use green color for the new line $\{x,z\}$ appearing in $G_{xz}^-$.
\end{itemize}

\begin{remark}
Just as an example for \eqref{eq:IGk} in the above proof:
if $G$ is the last graph in \eqref{e:graph-4pt}, then $x$ can be any point except for the upper-left one.
Taking for instance $x$ to be the upper-right point,
then $G_{xz}^-$ will be the following graphs
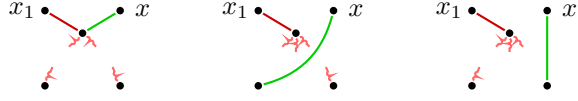
\begin{equ}
\begin{tikzpicture}[baseline=-15]
	\node[dot] (x1) at (0,0) {};\node at (-0.3,0) {$x_1$};
	\node[dot] (x2) at (1,0) {};\node at (1.3,0) {$x$};
	\node[dot] (x3) at (0,-1) {};
	\node[dot] (x4) at (1,-1) {};
	\node[dot] (z) at (.5,-.3) {};
	\draw[Cr] (x1) -- (z);\draw[Cg] (x2) -- (z);
		\draw[Phi] (z) -- ++(0.18,-0.18);\draw[Phi] (z) -- ++(-0.18,-0.18);
	\draw[Phi] (x3) -- ++(0.1,0.25);
	\draw[Phi] (x4) -- ++(-0.1,0.25);
\end{tikzpicture}
\qquad
\begin{tikzpicture}[baseline=-15]
	\node[dot] (x1) at (0,0) {};\node at (-0.3,0) {$x_1$};
	\node[dot] (x2) at (1,0) {};\node at (1.3,0) {$x$};
	\node[dot] (x3) at (0,-1) {};
	\node[dot] (x4) at (1,-1) {};
	\node[dot] (z) at (.5,-.3) {};
	\draw[Cr] (x1) -- (z);\draw[Cg,bend left=30] (x2) to (x3);
		\draw[Phi] (z) -- ++(0.18,-0.18);\draw[Phi] (z) -- ++(-0.18,-0.18);\draw[Phi] (z) -- ++(0,-0.25);
	\draw[Phi] (x4) -- ++(-0.1,0.25);
\end{tikzpicture}
\qquad
\begin{tikzpicture}[baseline=-15]
	\node[dot] (x1) at (0,0) {};\node at (-0.3,0) {$x_1$};
	\node[dot] (x2) at (1,0) {};\node at (1.3,0) {$x$};
	\node[dot] (x3) at (0,-1) {};
	\node[dot] (x4) at (1,-1) {};
	\node[dot] (z) at (.5,-.3) {};
	\draw[Cr] (x1) -- (z);\draw[Cg] (x2) -- (x4);
	\draw[Phi] (z) -- ++(0.18,-0.18);\draw[Phi] (z) -- ++(-0.18,-0.18);\draw[Phi] (z) -- ++(0,-0.25);
	\draw[Phi] (x3) -- ++(0.1,0.25);
\end{tikzpicture}
\end{equ}
and $G_{xz}^+$ is then
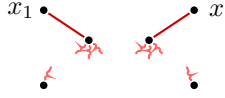
\begin{equ}
\begin{tikzpicture}
	\node[dot] (x1) at (0,0) {};\node at (-0.3,0) {$x_1$};
	\node[dot] (x2) at (2,0) {};\node at (2.3,0) {$x$};
	\node[dot] (x3) at (0,-1) {};
	\node[dot] (x4) at (2,-1) {};
	\node[dot] (z) at (0.6,-.4) {};
	\node[dot] (w) at (1.4,-.4) {};
	\draw[Cr] (x1) -- (z);
	\draw[Cr] (x2) -- (w);
	\draw[Phi] (z) -- ++(0.18,-0.18);\draw[Phi] (z) -- ++(-0.18,-0.18);\draw[Phi] (z) -- ++(0,-0.25);
	\draw[Phi] (w) -- ++(0.18,-0.18);\draw[Phi] (w) -- ++(-0.18,-0.18);\draw[Phi] (w) -- ++(0,-0.25);
	\draw[Phi] (x3) -- ++(0.1,0.25);
	\draw[Phi] (x4) -- ++(-0.1,0.25);
\end{tikzpicture}
\end{equ}
Note that we also color the edge connected to $x_1$ by red
because of the way it arises in \eqref{e:graph-4pt}.
\end{remark}
	
The reason for coloring the edges is that
we will reduce  graphs in $\cH_N^k$ to trees by ``cutting out'' the green lines.

\begin{lemma}\label{lem:tree}
For each graph $G$ associated with $F_{n,M,\eps}^k$
or $R_{N+1,M,\eps}^k$ in
Lemma~\ref{lem:FRk}, denote by $\widehat G$
the subgraph
formed by all its red lines.

Then, $\widehat G$ has exactly $k$ connected components;
each connected component is a rooted tree, and has exactly one vertex
in $V_G^\partial$ which is regarded as the root of the tree.
 \footnote{Recall that in graph theory, a tree is a graph without any cycle,
and a rooted tree is a tree in which one vertex is designated to be the root.
 In particular a graph consisting of only a single vertex and no edge is trivially a rooted tree.}
\end{lemma}

\begin{proof}
This is obvious by construction. Indeed, for the initial case
where $|V_G|=k$ and $|E_G|=0$, we have that $\widehat G$ is just $k$ trivial trees.  Assuming that we already know that
a graph $G$ satisfies the property in the lemma,
then,
when a new vertex $z$
is added  which creates a new red edge $\{x,z\}$
as in \eqref{eq:IGk},
it simply adds one more edge to one of the $k$ red trees.
%
(See the figure below for an illustration).
\end{proof}

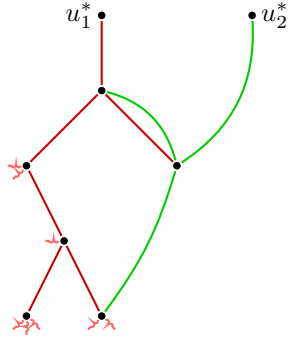
\begin{figure}[h]
\begin{tikzpicture}
	\node[dot] (x1) at (0,0) {};	\node at (-0.3,0) {$u_1^*$};
	\node[dot] (x2) at (2,0) {};\node at (2.3,0) {$u_2^*$};
	\node[dot] (y1) at (0,-1) {};
	\node[dot] (y2) at (-1,-2) {};
	\node[dot] (y3) at (1,-2) {};
	\node[dot] (z1) at (-0.5,-3) {};
	\node[dot] (z2) at (-1,-4) {};
	\node[dot] (z3) at (0,-4) {};
	\draw[Cr] (x1) to (y1);
	\draw[Cr] (y1) to (y2);
	\draw[Cr] (y1) to (y3);
	\draw[Cr] (y2) to (z1);
	\draw[Cr] (z1) to (z2);
	\draw[Cr] (z1) to (z3);
	\draw[Cg,bend left=30] (x2) to (y3);
	\draw[Cg,bend left =30] (y1) to (y3);
	\draw[Cg,bend left =10] (y3) to (z3);
	\draw[Phi] (z2) -- ++(0.18,-0.18);\draw[Phi] (z2) -- ++(-0.18,-0.18);\draw[Phi] (z2) -- ++(0,-0.25);
	\draw[Phi] (z3) -- ++(0.18,-0.18);\draw[Phi] (z3) -- ++(-0.18,-0.18);
	\draw[Phi] (z1) -- ++(-0.25,0);
	\draw[Phi] (y2) -- ++(-0.25,0);\draw[Phi] (y2) -- ++(-0.18,-0.18);
\end{tikzpicture}
\caption{An illustration of trees formed by red lines
in the case of $2$-point correlation. Here
$u_1^*,u_2^* \in V^\partial_G$ and all other vertices are in $V^0_G$.
One of the two trees, i.e. the tree which only contains one single vertex $v$, is trivial.}
\label{Fig1}
\end{figure}



\bl\label{convergencek} For the polynomial weight $\rho$ in \eqref{eq:rho},
it holds that for $\gamma>0$ and some $\ell>0$ and $n\leq N+1$
$$
\lim_{M\to\infty,\eps\to0}\cE^\eps_k F_{n,M,\eps}^k=F_{n}^k \quad\textrm{ in }\quad \bC^{-\gamma}(\rho^\ell),
$$
and
$$\lim_{M\to\infty,\eps\to0}\cE^\eps_kR_{n,M,\eps}^k=R_{n}^{\nu, k} \quad\textrm{ in } \quad\bC^{-\gamma}(\rho^\ell),$$
with $\cE^\eps_k$ being extension opeartor defined in \eqref{def:Ek}.
Here   $F_n^k$ can be written as integrals of the Green function  $C$ of $\frac12(m-\Delta)$ and $R_n^{\nu, k}$ only depends on $C$, $\Phi$. The associated graph of $F_n^k$  is the same as $F_{n,M,\eps}^k$  with $C_{M,\eps}$  in \eqref{def:IGk} replaced by the Green function  $C$ and the sum on $\Lambda_{M,\eps}$ replaced by the integral on $\mR^2$.
\el


The proof of Lemma \ref{convergencek}  is postponed to Section \ref{sec:con}. In the proof we could also give the explicit formula of $R_n^{\nu,k}$. Formally the associated graph of $R_n^{\nu,k}$ is the same as $R_{n,M,\eps}$ with $ C_{M,\eps} $ and $\Phi_{M,\eps}$ in \eqref{def:IGk}  replaced by the Green function  $C$ of $\frac12(m-\Delta)$ and $\Phi$, respectively  and the sum on $\Lambda_{M,\eps}$ replaced by the integral on $\mR^2$.

The following gives the main result of this section.

\bt\label{th:Sk} It holds that
\begin{equ}[e:S-exp-DS-lik]
	S_\lambda^{\nu,k}  = \sum_{n=0}^N  \frac{\lambda^n}{n!} F_{n}^k + \lambda^{N+1} R_{N+1}^{\nu,k},
\end{equ}
with $F_n^k, R_{N+1}^{\nu,k}$ given in Lemma \ref{convergencek}. This equality holds in $\bC^{-\gamma}(\rho^\ell)$ for $\gamma>0$, some $\ell>0$ and polynomial weight $\rho$ in \eqref{eq:rho}.
\et
 \begin{proof}
 By definition of $S^\nu_\lambda$ in \eqref{def:S}, \eqref{co:S}, \eqref{e:S-exp-DSk}  we have
	\begin{align*}
		S_\lambda^{\nu,k}
		&=\lim_{\eps\to 0, M\to\infty}\cE^\eps_kS_{\lambda,M,\eps}^k
		\\
		&=\lim_{\eps\to 0, M\to\infty}\sum_{n=0}^N\frac{\lambda^n}{n!}\cE^\eps_k F_{n,M,\eps}^k
		+\lambda^{N+1}\lim_{\eps\to 0, M\to\infty}\cE^\eps_k R_{N+1,M,\eps}^k
		\\
		&=\sum_{n=0}^N\frac{\lambda^n}{n!}F_{n}^k+\lambda^{N+1}R_{N+1}^{\nu,k},
	\end{align*}
		where we use Lemmas \ref{convergencek} to obtain the last equality. 
	\end{proof}

Now to prove \eqref{e:main}, we only need to show
\begin{proposition}\label{pro} For the polynomial weight $\rho$ in \eqref{eq:rho} and $n\leq N+1$,
	it holds that for $\gamma>0$ and some $\ell>0$
	 $$\|R_n^{\nu,k}\|_{\bC^{-\gamma}(\rho^\ell)}\lesssim1+\lambda^{4k+12n}.$$
\end{proposition}
\begin{proof}
The result  follows from  Lemma \ref{lem:estimate1} below and Lemma \ref{convergencek}.
\end{proof}

\section{Proof of results in Section \ref{sec:3}}

\subsection{Uniform bound on the remainder $R_N$}\label{sec:est}

In this section we give uniform estimates for $\mathbf{I}_G$ defined in \eqref{def:IGk} with $G\in \cH$.
For later use we construct the following random variables such that $\{Z_{M,\eps}^{(i)}\}_{i\in \mN}$ are i.i.d. random variables and each $Z_{M,\eps}^{(i)}$ is an independent copy of $Z_{M,\eps}$.
 We introduce $D_i$ as follows: let $D_0\eqdef \{1\}$ and
\begin{align}\label{eqdefD}&D_1\eqdef\{\Phi_{M,\eps}, Z_{M,\eps}^{(i)},i\in \mN\},\no\\ &D_2\eqdef\{\Wick{\Phi_{M,\eps} {Z}_{M,\eps}^{(i)}},\Wick{\Phi_{M,\eps}^2},\Wick{{Z}_{M,\eps}^{(i)}{Z}_{M,\eps}^{(j)}},  i, j\in \mN, i\neq j \},\\  &D_3\eqdef\{\Wick{\Phi^2_{M,\eps}Z_{M,\eps}^{(i)}},\Wick{\Phi_{M,\eps}^3},\Wick{\Phi_{M,\eps} {Z}_{M,\eps}^{(i)}{Z}_{M,\eps}^{(j)}},\Wick{{Z}_{M,\eps}^{(i)}{Z}_{M,\eps}^{(j)}{Z}_{M,\eps}^{(k)}}, i, j,k\in \mN, i\neq j\neq k \},\no\end{align}
where
$\Wick{\Phi_{M,\eps}^i}$ are defined in \eqref{def:Wick} and for $i,j,k\in \mN, i\neq j\neq k$
\begin{equs}
\Wick{\Phi_{M,\eps} {Z}_{M,\eps}^{(i)}}&\eqdef \Phi_{M,\eps} {Z}_{M,\eps}^{(i)},\qquad \qquad\quad
\Wick{{Z}_{M,\eps}^{(i)}{Z}_{M,\eps}^{(j)}}\eqdef {Z}_{M,\eps}^{(i)}{Z}_{M,\eps}^{(j)},
\\
\Wick{\Phi^2_{M,\eps}{Z}_{M,\eps}^{(i)}}&\eqdef\Wick{\Phi^2_{M,\eps}}{Z}_{M,\eps}^{(i)} ,\qquad\Wick{\Phi_{M,\eps} {Z}_{M,\eps}^{(i)}{Z}_{M,\eps}^{(j)}}\eqdef\Phi_{M,\eps} {Z}_{M,\eps}^{(i)}{Z}_{M,\eps}^{(j)},
\\
 &\Wick{{Z}_{M,\eps}^{(i)}{Z}_{M,\eps}^{(j)}{Z}_{M,\eps}^{(k)}}\eqdef{Z}_{M,\eps}^{(i)}{Z}_{M,\eps}^{(j)}{Z}_{M,\eps}^{(k)} .
 \end{equs}
By probabilistic calculation we have that for $\kappa,\gamma>0$ and $p>1$
\begin{align}\label{eq:re1}
\E\|{Z}_{M,\eps}^{(i)}\|_{\bC^{-\gamma,\eps}(\rho^\kappa)}^p+\E\|\Wick{{Z}_{M,\eps}^{(i)}{Z}_{M,\eps}^{(j)}}\|_{\bC^{-\gamma,\eps}(\rho^\kappa)}^p+\E\|\Wick{{Z}_{M,\eps}^{(i)}{Z}_{M,\eps}^{(j)}{Z}_{M,\eps}^{(k)}}\|_{\bC^{-\gamma,\eps}(\rho^\kappa)}^p\lesssim1,
\end{align}
where  $i\neq j\neq k$ and the proportional constant is independent of $\eps, M$ and there exist random variables
$Z^{(i)}, \Wick{Z^{(i)}Z^{(j)}}, \Wick{Z^{(i)}Z^{(j)}Z^{(k)}}\in L^p(\Omega,C_T\bC^{-\gamma}(\rho^\kappa))$ such that
\begin{align*}
&\E\|\cE^\eps{Z}_{M,\eps}^{(i)}-Z^{(i)}\|_{C_T\bC^{-\gamma,\eps}(\rho^\kappa)}^p+\E\|\cE^\eps\Wick{{Z}_{M,\eps}^{(i)}{Z}_{M,\eps}^{(j)}}-\Wick{Z^{(i)}Z^{(j)}}\|_{C_T\bC^{-\gamma,\eps}(\rho^\kappa)}^p\\&+\E\|\cE^\eps\Wick{{Z}_{M,\eps}^{(i)}{Z}_{M,\eps}^{(j)}{Z}_{M,\eps}^{(k)}}-\Wick{{Z}^{(i)}{Z}^{(j)}{Z}^{(k)}}\|_{C_T\bC^{-\gamma,\eps}(\rho^\kappa)}^p\to0, \textrm{ as } \eps\to0, M\to\infty.
\end{align*}
(c.f. \cite{MW17, GH18a, ZZ18}).
By direct calculation we have
\begin{align}\label{eq:CZ}
 C_{M,\eps}(x-y)=\E[Z_{M,\eps}(x)Z_{M,\eps}(y)].
\end{align}

\bc\label{co:2} For the polynomial weight $\rho$ and $p_0$  in \eqref{eq:rho},  it holds  for any $f_i\in D_i$, $1<p<p_0/9$, $\gamma>0$ that
$$\E\|f_1\|_{\bC^{-\gamma,\eps}(\rho^5)}^p+\E\|f_2\|_{\bC^{-\gamma,\eps}(\rho^9)}^p+\E\|f_3\|_{\bC^{-\gamma,\eps}(\rho^{13})}^p\lesssim 1+\lambda^{12p},$$
with the proportional constant independent of $M,\eps$.
\ec
\begin{proof}
	Similar as before we omit the supscript $M$ in the proof for notation's simplicity.
	 By  Corollary \ref{co:zmm} the desired estimate holds for $f_i=\Wick{\Phi^i}$. For $f_1=Z_\eps^{(i)}, f_2=Z_\eps^{(i)}Z_\eps^{(j)}$ and $f_3=Z_\eps^{(i)}Z_\eps^{(j)}Z_\eps^{(k)}$, $i\neq j\neq k$, the result follows from \eqref{eq:re1}. For other cases we write
\begin{equs}
	  \Wick{\Phi_{\eps} {Z}^{(i)}_{\eps}} & =Y_\eps Z^{(i)}_{\eps}+\Wick{Z_\eps{Z}^{(i)}_{\eps}},\\
	   \Wick{\Phi_{\eps} {Z}^{(i)}_{\eps}{Z}^{(j)}_{\eps}} & =Y_\eps \Wick{{Z}^{(i)}_{\eps}{Z}^{(j)}_{\eps}}+\Wick{Z_\eps {Z}^{(i)}_{\eps}{Z}^{(j)}_{\eps}},\\
	   \Wick{\Phi_{\eps}^2 {Z}_{\eps}^{(i)}} & =Y_\eps^2 {Z}_{\eps}^{(i)}+2Y_\eps\Wick{Z_\eps {Z}_{\eps}^{(i)}}+\Wick{Z_\eps^2 {Z}_{\eps}^{(i)}}.
\end{equs}
where $\Wick{Z_\eps^2 {Z}_{\eps}^{(i)}}=\Wick{Z_\eps^2} {Z}_{\eps}^{(i)}$ and by probabilisitc calculation for $\gamma, \kappa>0, p>1$
 $$\E\|\Wick{Z_\eps^2 {Z}_{\eps}^{(i)}}\|_{\bC^{-\gamma,\eps}(\rho^\kappa)}^p\lesssim1.$$
	  By \eqref{eq:re1}, the same calculation in Corollary \ref{co:zmm} and Lemma \ref{lem:pow}, the desired estimates also hold for the terms $\Wick{\Phi_{\eps} {Z}^{(i)}_{\eps}}$, $\Wick{\Phi_{\eps} {Z}^{(i)}_{\eps}{Z}^{(j)}_{\eps}}$ and $\Wick{\Phi_{\eps}^2 {Z}_{\eps}^{(i)}}$, which implies the result.
\end{proof}

To state the following lemma for $G\in \cH$ with $\mathbf{I}_G$ in \eqref{def:IGk},  we
extend it as periodic functions on $\Lambda_\eps$.

\begin{lemma}\label{lem:estimate1}
	For each $G\in \mathcal H^k_n$, $n\leq N+1$,
that appears in Lemma~\ref{lem:FRk},	
	letting $\mathbf{I}_G$ be the function  introduced in \eqref{def:IGk} and $\rho$ be  polynomial weight  as in \eqref{eq:rho} 
	it holds that for $\gamma>0$ and some $\ell>0$
	\begin{align}\label{est:IG}
	\|\cE^\eps_k\mathbf{I}_G\|_{\bC^{-\gamma}(\rho^\ell)}\lesssim 1+\lambda^{{4k}+12n},\end{align}
	where the proportional constant is independent of $\eps, M$ and $\cE^\eps_k$ is the extension operator introduced in \eqref{def:Ek} in Appendix \ref{App}.
\end{lemma}

\begin{proof}
In the proof we use $\rho^\ell$  to denote weight and $\ell$ may change from line to line. We also abuse the notation of $\rho$ for the weight in different spatial dimensions.
By Lemma~\ref{lem:tree} we have $\widehat G$ which is a union of $k$ rooted trees, formed by the red edges in $G$.

\textbf{Step 1. Reduction to trees.}

It will turn out that trees are easier to estimate than general graphs. To reduce our analysis for  $\mathbf{I}_G $ to the case of trees,
we proceed as follows.

Recall the definition \eqref{def:IGk} for $\mathbf{I}_G $.
For each of the green lines $\{u,v\} \in E_G$,
we use \eqref{eq:CZ} to replace each 
$C_{M,\eps}(x_u,x_v)$ in  \eqref{def:IGk}
 by 
 $\E[Z_{M,\eps}^{(i)}(x_u)Z^{(i)}_{M,\eps}(x_v)]$
 with $\{Z_{M,\eps}^{(i)}\}_{i}$ given at the beginning of this section. Each $C_{M,\eps}(x_u,x_v)$ corresponds to different $Z_{M,\eps}^{(i)}$.
This can be done for all the green lines since our graph is finite.  The resulting function is the same as
$\mathbf{I}_G $.

Note that for periodic functions $f$ on $\Lambda_{\eps}$,   $\cI_\eps(f)$ introduced in \eqref{def:I} could be written as $$\cI_\eps(f)=\int_{\Lambda_\eps}C_\eps(x-y)f(y)\dif y.$$

Since
$\widehat G$ is a disjoint union of $k$ trees (by Lemma~\ref{lem:tree}), i.e. $\widehat G = \sqcup_{i=1}^k T_i$, where each $T_i$ is a tree,
we know that
$\mathbf{I}_G $ is the expectation
of a product of $k$ functions,
and each of these  $k$ functions have the following form
\begin{equ}[e:F_T]
F_{T}(x_{u^*})=\int \prod_{\{u,v\}\in E_{T}}C_{\eps}(x_u,x_v)
	\Big( f_1(x_{u^*})^{1-\deg_T(u^*)}
	\prod_{v\in V_{T}\backslash \{u^*\}}f_v(x_v) g_v^T(x_v)
	\Big)
	\prod_{v\in V_{T}\backslash \{u^*\}} \dif x_v,
\end{equ}
where $f_1\in D_1$, $f_v\in \cup_{i=0}^3D_{i}$, $g_v^T=1$ and $\deg_T(v)$ denotes the degree of $v$ in the graph $T$.
Here and below we just write $T$ for $T_i$ to simplify the notation,
and we have introduced the function $g_v^T$ for the purpose of induction later.

For instance,
the graph in Figure \ref{Fig1} then reduces to the one shown in Figure \ref{Fig2} where each tiny green wavy line denotes a factor of  $Z^{(i)}_{M,\eps}$.
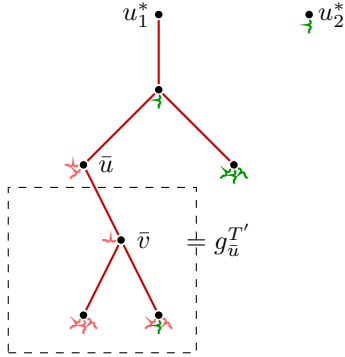
\begin{figure}[h]
\begin{tikzpicture}
	\node[dot] (x1) at (0,0) {};	\node at (-0.3,0) {$u_1^*$};
	\node[dot] (x2) at (2,0) {};\node at (2.3,0) {$u_2^*$};
	\node[dot] (y1) at (0,-1) {};
	\node[dot] (y2) at (-1,-2) {};
	\node[dot] (y3) at (1,-2) {};
	\node[dot] (z1) at (-0.5,-3) {};
	\node[dot] (z2) at (-1,-4) {};
	\node[dot] (z3) at (0,-4) {};
	\draw[Cr] (x1) to (y1);
	\draw[Cr] (y1) to (y2);
	\draw[Cr] (y1) to (y3);
	\draw[Cr] (y2) to (z1);
	\draw[Cr] (z1) to (z2);
	\draw[Cr] (z1) to (z3);

	\draw[Z] (x2) -- ++(0,-0.25);
	\draw[Z] (y1) -- ++(0,-0.25);
	\draw[Z] (y3) -- ++(0.18,-0.18);\draw[Z] (y3) -- ++(-0.18,-0.18);\draw[Z] (y3) -- ++(0,-0.25);
	
	\draw[Phi] (z2) -- ++(0.18,-0.18);\draw[Phi] (z2) -- ++(-0.18,-0.18);\draw[Phi] (z2) -- ++(0,-0.25);
	\draw[Phi] (z3) -- ++(0.18,-0.18);\draw[Phi] (z3) -- ++(-0.18,-0.18);\draw[Z] (z3) -- ++(0,-0.25);
	\draw[Phi] (z1) -- ++(-0.25,0);
	\draw[Phi] (y2) -- ++(-0.25,0);\draw[Phi] (y2) -- ++(-0.18,-0.18);

 \draw[dashed] (-2,-2.3) rectangle (.5,-4.5);
 	\node  at (-.7,-2) {$\bar u$};
		\node  at (-0.2,-3) {$\bar v$};
\node  at (0.8,-3) {$=g_{\bar u}^{T'}$};	
\end{tikzpicture}
\caption{An illustration for reducing graphs to trees, and
inductive integrations}
\label{Fig2}
\end{figure}

	
We first consider the degenerate case separately where $\deg(u^*)=0$, namely $F_T$ is simply $f_1 (x_{u^*})$ with $f_1\in D_1$.
We use \eqref{est:Phi1} and \eqref{eq:re1} and easily find for $\gamma>0$ and $3p<p_0$ with  $p_0$ in \eqref{eq:rho} 
	\begin{align}\label{est1}
	\E\|Z_{M,\eps}\rho\|_{\bC^{-\gamma,\eps}}^p+\E\|\Phi_{\eps}\rho^5\|_{\bC^{-\gamma,\eps}}^p\lesssim 1+\lambda^{4p},\end{align}

In the following we only need to consider the case that  $\deg(u^*)=1$ i.e.  nontrivial trees. We claim that for $\gamma>0$  and $\ell$ large enough,  \begin{align}\label{FT}\|F_T\|_{\bC^{2-\gamma,\eps}(\rho^\ell)}\lesssim \prod_{v\in V_T} \|f_v\|_{\bC^{-\gamma,\eps}(\rho^{13})},\end{align}
with $f_v\in \cup_{i=0}^3D_i$ defined in \eqref{eqdefD}. By the reduction 
we have for $\gamma>0$
\begin{align}\label{eq:Ho1}
&\|\mathbf{I}_G\|_{\bC^{-\gamma,\eps}_s(\Lambda_{\eps}^k,\rho^\ell)}
\lesssim \E\prod_{i=1}^k \|F_{T_i}\|_{\bC^{-\gamma,\eps}(\rho^\ell)}\no\\\lesssim& \prod_{\deg(u^*_i)=1} \Big(\E \|F_{T_i}\|_{\bC^{2-\gamma,\eps}(\rho^\ell)}^k\Big)^{\frac1k}\prod_{\deg(u^*_i)=0} \Big(\E \|F_{T_i}\|_{\bC^{-\gamma,\eps}(\rho^5)}^k\Big)^{\frac1k},
\end{align}
where $\bC^{-\gamma,\eps}_s(\Lambda_\eps^k,\rho^\ell)$ is the H\"older space for each variable on $\Lambda_\eps^k$ introduced in \eqref{def:cs} in Appendix \ref{App}.
By the claim \eqref{FT} and  Corollary \ref{co:2}, we deduce that for $\deg(u_i^*)=1$, $\gamma>0$
\begin{align}\label{eq:EF1}\Big(\E \|F_{T_i}\|_{\bC^{2-\gamma,\eps}(\rho^\ell)}^k\Big)^{1/k}\lesssim1+\lambda^{12n_i},\end{align}
 where $n_i$ is given by the number of vertices in $V_0$ for the tree $T_i$ and $\sum_{i=1}^kn_i\leq n$  and we use $9kn_i\leq 9k (N+1)<p_0$ in \eqref{eq:rho}.
  For the case that $\deg(u_i^*)=0$, by \eqref{est1},
 \begin{align}\label{eq:EF0}\Big(\E \|F_{T_i}\|_{\bC^{-\gamma,\eps}(\rho^5)}^k\Big)^{1/k}\lesssim 1+\lambda^{4}.\end{align}
As a result, \eqref{est:IG}  follows from Lemma \ref{euniform}, Lemma \ref{lem:A1} and arbitrary $\gamma$. 
It remains to prove \eqref{FT}.

	\textbf{Step 2. Estimate of each tree.}
	
	
	
	
	
	Fixing a rooted tree $T$ as above, we will integrate the variables in $F_T$ in \eqref{e:F_T}
	from the leaves \footnote{Recall that in graph theory a leaf of a rooted tree is a vertex of degree $1$ and is not the root. Any nontrivial rooted tree must have at least one leaf.}
	 of the tree $T$ and estimate the effect
	of the integrations. More precisely, we claim that
	for every subtree $\bar T$ of $T$ which contains the root $u^*$,
	\eqref{e:F_T} still holds with $T$ and $g^T$ on the RHS replaced by $\bar T$ and $g^{\bar T}$, 
where the functions $g$  (which depend on $\bar T$) are such that
 $\|g\|_{\bC^{s,\eps}(\rho^\ell)}$ is bounded by  $\Pi_{j\in I_g}\|f_j\|_{\bC^{-\gamma,\eps}(\rho^{13})}, f_j\in \cup_{i=0}^3D_i$ for some index set $I_g\subset \mN$ and $s>\gamma>0$, $n>0$.

We prove by induction downwards on  the value of $|V_{\bar T}|$.
We already have that \eqref{e:F_T} holds for $\bar T= T$ where $g=1$ trivially
satisfies the claimed bound.

Assume that \eqref{e:F_T} holds with \rmb{${F_{T}(x_{u^*})}$ on the LHS, and  the RHS  written as the integral in terms of a subtree $\bar T$.}
Let
 $\bar v$ be a leaf of $\bar T$.
Then there is a unique $\bar u\in V_{\bar T}$ such that $\bar e \eqdef \{\bar u,\bar v\}\in E_{\bar T}$. We then have
\begin{equ}\label{eq:barT}
{F_{T}(x_{u^*})}
=\int  \!\!\!\!\!\! \prod_{\{u,v\}\in E_{\bar T}\backslash \{\bar e\}}
\!\!\!\!\!\! C_{\eps}(x_u,x_v)\!\!\!\!\!\!
	\prod_{v\in V_{\bar T}\backslash \{u^*,\bar v\}}\!\!\!\!\!\!  (f_v(x_v) g_v^{\bar T}(x_v))
\Big(
\int C_{\eps}(x_{\bar u},x_{\bar v})
(f_{\bar v}(x_{\bar v}) g_{\bar v}^{\bar T}(x_{\bar v}))
 \dif x_{\bar v}
\Big)	
	\!\!\!\!\!\!\prod_{v\in V_{\bar T}\backslash \{u^*,\bar v\}}\!\!\!\!\!\! \dif x_v,
\end{equ}
and we now estimate the integration over $\bar v$ in the parenthesis,
which will imply that  \eqref{e:F_T} holds for the subtree $T'=(V_{\bar T} \backslash\{\bar v\}, E_{\bar T}\backslash \{\bar e\})$ with $g_{\bar u}^{T'}(x_{\bar u})$ redefined as
\begin{align}\label{eq:baru}
g_{\bar u}^{T'}(x_{\bar u})=
\Big(
\int C_{\eps}(x_{\bar u},x_{\bar v})
(f_{\bar v}(x_{\bar v}) g_{\bar v}^{\bar T}(x_{\bar v}))
 \dif x_{\bar v}
\Big)g_{\bar u}^{\bar T}(x_{\bar u})	
\end{align}
and $f_v, g_v$ for $v \neq \bar u$ left unchanged. (See Figure \ref{Fig2} for an illustration).

The integration over $\bar v$ in the parenthesis
 is in one of the following forms:
	\begin{align}\label{zmm1}
\cI_\eps(g)(x_{\bar u}),\quad	\cI_\eps(f_3)(x_{\bar u}),\quad \cI_\eps(f_2g)(x_{\bar u}),\quad \cI_\eps(f_1g)(x_{\bar u}),
	\end{align}
	with $f_i\in D_i$.
	By  \eqref{eq:sch} and Lemma \ref{lem:pow}, we have that for	$1>s>\gamma>0$,
\begin{equs}[e:zmm1-bd]
	\|\cI_\eps(g)\|_{\bC^{2-\gamma,\eps}(\rho^\ell)}
	&\lesssim \|g\|_{\bC^{s,\eps}(\rho^\ell)}\\
\|\cI_\eps(f_3)\|_{\bC^{2-\gamma,\eps}(\rho^{13})}
&\lesssim \|f_3\|_{\bC^{-\gamma,\eps}(\rho^{13})}
\\
	\|\cI_\eps(gf_2)\|_{\bC^{2-\gamma,\eps}(\rho^{\ell+13})}
	&\lesssim \|g\|_{\bC^{s,\eps}(\rho^{n})}\|f_2\|_{\bC^{-\gamma,\eps}(\rho^{13})}
\\
\|\cI_\eps(gf_1)\|_{\bC^{2-\gamma,\eps}(\rho^{\ell+13})}
&\lesssim \|g\|_{\bC^{s,\eps}(\rho^{n})}\|f_1\|_{\bC^{-\gamma,\eps}(\rho^{13})}.
\end{equs}
	By the inductive assumption on $g$,
	we have that the left-hand sides
	are   bounded by $$\Pi_{j\in I_g}\|f_j\|_{\bC^{-\gamma,\eps}(\rho^\ell)}\|f_{\bar v}\|_{\bC^{-\gamma,\eps}(\rho^{13})},$$
	with $f_{\bar v}\in \cup_{i=0}^3D_i$.	
This implies that \eqref{e:F_T} holds for the smaller subtree $ T'$, namely the number of vertices decreases by one.
	
By inductively decreasing
the value of $|V_{\bar T}|$, we know that \eqref{e:F_T} holds for
the subtree $\bar T$ which only has $2$ vertices including $u^*$.
Now $F_{ T}$ is simply {\it equal to} one of the cases
in \eqref{zmm1}, so we can just bound $F_{ T}$ by
\eqref{e:zmm1-bd}. On the other hand, integration over one vertices counts for at most one $\|f_i\|_{\bC^{-\gamma,\eps}(\rho^{13})}$, which gives  \eqref{FT}.
\end{proof}

\vspace{2ex}

\subsection{Convergence of $F_{n,M,\eps}^k$ and $R_{n,M,\eps}^k$}\label{sec:con}

In this section we give the proof of Lemmas \ref{convergencek}, i.e. the discrete integral converges to the corresponding continuous one.
By the uniform bound in Section \ref{sec:est}  we already know that each $\|\cE^\eps_kF_{n,M,\eps}^k\|_{\bC^{-\gamma}(\rho^\ell)}$ and $\|\cE^\eps_kR_{n,M,\eps}^k\|_{\bC^{-\gamma}(\rho^\ell)}$ 
are uniformly bounded, which implies there exists a convergent subsequence. Now, we want to give the explicit formula of the limit.


\begin{proof}[Proof of Lemmas \ref{convergencek}]
	Since we consider the limits of $\cE^\eps_kF_{n,M,\eps}^k$ and $\cE^\eps_kR_{n,M,\eps}^k$, which only depend on the law of $\Phi_{{M,\eps}}$ and $Z_{M,\eps}$,  we can assume the setting as in the proof of Theorem \ref{lem:cm}, i.e. we fix a stochasitic basis $(\Omega,\cF, \mathbf{P})$ and we have random variables $(\Phi_{{M,\eps}},Z_{M,\eps}, Y_{M,\eps})$ and $(\Phi,Z,Y)$ such that the convergence in \eqref{eq:PhiY1}-\eqref{eq:Phic} holds. We also have all the random variables $$\{Z^{(i)}_{M,\eps}, \Wick{Z^{(i)}_{M,\eps}Z^{(j)}_{M,\eps}},\Wick{Z^{(i)}_{M,\eps}Z^{(j)}_{M,\eps}Z^{(k)}_{M,\eps}},i\neq j\neq k\}$$ and their continuous limit $\{Z^{(i)}, \Wick{Z^{(i)}Z^{(j)}},\Wick{Z^{(i)}Z^{(j)}Z^{(k)}},i\neq j\neq k\}$ on the same stochastic basis,
	which can be done by Skorohod Theorem. As a result, every element in $\cup_{i=1}^3D_i$ and their continuous limit can be realized on this stochastic basis. We also view $\Phi_{{M,\eps}}, Y_{M,\eps}$ and $Z^{(i)}_{M,\eps}$ as periodic functions on $\Lambda_{\eps}$.
	
Combining the proof in Theorem \ref{lem:cm} and Corollary \ref{co:2} we know that for every element $f_i\in D_i$, $\cE^\eps f_i$ converges to the corresponding continuous one in $L^p(\Omega;\bC^{-\gamma}(\rho^{13}))$ for $\gamma>0, p>1$ with $9p<p_0$ for $p_0$ in \eqref{eq:rho}. For example, we consider
$\Wick{\Phi^2_{M,\eps}Z_{M,\eps}^{(i)}}$ and
 by similar calcuation as  \eqref{eq:PhiY} we obtain
	 $$\E\|\mathcal{E}^\eps (Y^2_{M,\eps}Z_{M,\eps}^{(i)})-Y^2Z^{(i)}\|_{\bC^{-\gamma}(\rho^{13})}^p \lesssim o(\eps),$$ for $\gamma>0$, where the proportional constant is independent of $M$.
	  For other terms in $\Wick{\Phi^2_{M,\eps}Z_{M,\eps}^{(i)}}$ we have similar convergence. In the following $n$ may change for different convergences.

 For each graph we consider each $\mathbf{I}_G$ and as in the proof of Lemma \ref{lem:estimate1} we know
	$$\mathbf{I}_G(x_1,\dots,x_k)=\E \Big(\Pi_{i=1}^k F_{T_i}(x_i) \Big),$$
	with $F_{T_i}$ defined in \eqref{e:F_T} and $T_i$ being trees and subgraphs for $G$. By  the uniform bounds  obtained in \eqref{eq:EF0} and \eqref{eq:EF1}, we first consider the convergence of $\cE^\eps F_{T_i}$ in $L^k(\Omega, \bC^{-\gamma}(\rho^\ell))$ for each tree $T_i$ and $\gamma>0$. For the degenerate case where $\deg(u^*)=0$, i.e. $F_{T_i}$ is just $f_1(x_i)$ with $f_1\in D_1$, we have proved the convergence.
	
	In the following we consider the nontrivial case, i.e. $\deg(u^*)=1$. The idea is to use Lemma \ref{Cepro}, Lemma \ref{lem:Epro} and uniform estimate in Corollary \ref{co:2} to exchange the extension operator $\cE^\eps$ to every leaf in the tree up to an $o(\eps)$ constant. We consider the subtree $\bar{T}$ of $T_i$ and  we want to prove for each subtree $\bar{T}$ as in Step 2 in the proof of Lemma \ref{lem:estimate1} and $0<\gamma<s<1$
	\begin{align}\label{eq:co}
	&\E\Big\|\mathcal{E}^\eps F_{T_i}(x_{u^*})- \int \prod_{\{u,v\}\in E_{\bar T}} C(x_{u}-x_v)
	\prod_{v\in V_{\bar{T}}\backslash \{u^*\}}(\cE^\eps{f}_v)(x_v)(\cE^\eps g_v^{\bar T})(x_v)\prod_{v\in V_{\bar{T}}\backslash \{u^*\}}\dif x_v\Big\|_{\bC^{2-s}(\rho^\ell)}^k\lesssim o(\eps),
	\end{align}
	We do induction on $|V_{\bar T}|$ and
	we start from the subtree $\bar{T}$  containing the root $u^*$ with $|V_{\bar T}|=2$ and  by Step 2 in the proof of Lemma \ref{lem:estimate1}, we have
	$$F_{T_i}(x_{u^*})=\int C_\eps(x_{u^*},x_v)f_v(x_v)g_v^{\bar T}(x_v)\dif x_v,$$
	with $f_v\in \cup_{i=0}^3D_i$ and $g_v^{\bar T}$ as in the proof of Lemma \ref{lem:estimate1}. By Lemma \ref{Cepro} and uniform bounds of $\cE^\eps F_{T_i}$  we know that for $s>\gamma>0$  and some $n>0$
	\begin{align}\label{eC}&\E\Big\|\mathcal{E}^\eps F_{T_i}(x_{u^*})- \int C(x_{u^*}-x_v)\mathcal{E}^\eps (f_v(x_v)g_v^{\bar T}(x_v))\dif x_v\Big\|_{\bC^{2-s}(\rho^\ell)}^k\lesssim o(\eps).
	\end{align}
	Using Lemma \ref{lem:Epro} we find
	\begin{align*}&\E\Big\|\mathcal{E}^\eps F_{T_i}(x_{u^*})- \int C(x_{u^*}-x_v)(\mathcal{E}^\eps f_v)(x_v)(\cE^\eps g_v^{\bar T})(x_v)\dif x_v\Big\|_{\bC^{2-s}(\rho^\ell)}^k\lesssim o(\eps).
	\end{align*}
	Hence \eqref{eq:co} holds with $\bar{T}$ satisfying $|V_{\bar{T}}|=2$. Suppose that \eqref{eq:co} holds with $\bar T$ satisfying $|V_{\bar T}|\leq n$. Now we consider the subgraph $\bar T$ with $|V_{\bar T}|=n+1$. By the Step 2 in the proof of Lemma \ref{lem:estimate1} we could find a $\bar v$ as a leaf of $\bar{T}$ and a unique $\bar u\in V_{\bar T}$ such that \eqref{eq:barT} holds. We set $g_{\bar u}^{T'}(x_{\bar u})$ as in \eqref{eq:baru} and use induction for $T'=\bar T\backslash \{\bar v\}$ to have
	\begin{align*}
	&\E\Big\|\mathcal{E}^\eps F_{T_i}(x_{u^*})- \int \prod_{\{u,v\}\in E_{\bar T}\backslash \{\bar e\}} C(x_{u}-x_v)
	\prod_{v\in V_{\bar{T}}\backslash \{u^*,\bar v\}}(\cE^\eps{f}_v)(x_v)(\cE^\eps g_v^{T'})(x_v)\prod_{v\in V_{\bar{T}}\backslash \{u^*,\bar v\}}\dif x_v\Big\|_{\bC^{2-s}(\rho^\ell)}^k\\&\lesssim o(\eps),
	\end{align*} For $\cE^\eps g_{\bar u}(x_{\bar u})$ we use Lemma \ref{Cepro} and Lemma \ref{lem:Epro} to have
	\begin{align*}&\E\Big\|\mathcal{E}^\eps g_{\bar u}^{T'}(x_{\bar u})- \cE^\eps g_{\bar u}^{\bar T}(x_{\bar u}) \int C(x_{\bar u}-x_{\bar v})(\mathcal{E}^\eps f_{\bar v})(x_{\bar v})(\cE^\eps g_{\bar v}^{\bar T})(x_{\bar v})\dif x_{\bar v}\Big\|_{\bC^{2-s}(\rho^\ell)}^k\lesssim o(\eps).
	\end{align*}
Hence \eqref{eq:co} holds for $\bar T$ satisfying $|V_{\bar T}|= n+1$ and for $\bar T=T_i$,
 which combined with  the convergence of $\cE^\eps f_v$ implies that
\begin{align}\label{eq:con}
&\E\Big\|\mathcal{E}^\eps F_{T_i}(x_{u^*})- \int \prod_{\{u,v\}\in E_{T_i}} C(x_{u}-x_v)
 \prod_{v\in V_{T_i}\backslash \{u^*\}}\bar{f}_v(x_v)\prod_{v\in V_{T_i}\backslash \{u^*\}}\dif x_v\Big\|_{\bC^{2-s}(\rho^\ell)}^k\lesssim o(\eps),
\end{align}
	 with $\bar{f}_v=\lim_{\eps\to0}\cE^\eps f_v$. Here for the second term we abuse the notation $\bar{f}_v(x_v)$ which means a distribution in $\bC^{-\gamma}(\rho^\ell)$ and the integral of $C(x_u-x_v)$ w.r.t. $x_v$ means the operator $\frac12(m-\Delta)^{-1}$ acts on the corresponding distributions $\bar{f}_v$. Now we obtain the convergence of each $F_{T_i}$ and give the explicit formula for the limit $F_n^k$ and $R_n^{\nu,k}$.
	
	  We use $G$ to denote the graph associaed with $F_n^k$. To write $F^k_n$ as the integrals of $C$, i.e.
	  $$\int  \Big(\prod_{\{u,v\}\in E_G}C(x_u-x_v)\Big)\prod_{v\in V^0_G} \dif x_v ,$$
	   which we denote by $\tilde{F}^k_n$, we only need to replace all the correlation of $Z^{(i)}$ in $F_n^k$ by $C$. However, since $Z^{(i)}$ is not a usual function, we need more argument to explain this replacement and it suffices to prove  $\<F^k_n,\varphi\>=\<\tilde{F}^k_n,\varphi\>$ with $0\leq\varphi\in \mathcal{S}(\mathbb{R}^{2k})$. As $Z^{(i)}$ is not a function, we first  replace each $Z^{(i)}$ in $F^k_n$ by $Z^{(i)}*\varphi_\eps$ with $\{\varphi_\eps\}_\eps$ being  standard mollifiers and we denote the resulting function by $F^{k,\eps}_n$. This means we replace each $\bar f_v$ in the second term of \eqref{eq:con} by  $f_v^\eps$, which converges to $\bar f_v$ in  $\bC^{-\gamma}(\rho^\ell)$. To compare $F_n^k$ with $F_n^{k,\eps}$ we are in a similar situation as  the convergence of $F_{n,M,\eps}^k$ to $F_n^k$ with the sum on $\Lambda_\eps$ replaced by the integral on $\mR^2$. Then the same argument as reduction to trees and indunction on subtrees implies that
	  $$|\langle F^{k,\eps}_n-F^k_n,\varphi\rangle|\rightarrow0,\quad \textrm{ as }\eps \to0.$$
	  For $F^{k,\eps}_n$ we could replace each $\E[(Z^{(i)}*\varphi_\eps)(x) (Z^{(i)}*\varphi_\eps)(y)]$ by $\tilde{C}_\eps(x-y)\eqdef(\varphi_\eps*C*\varphi_\eps)(x-y)$ by Fubini's Theorem. Now we check that $F^{k,\eps}_n$ converge to $\tilde{F}^k_n$. For each graph $G$ associated with $F_n^k$ we abuse the notation $m$ to denote the Lebesgue measure $\prod_{v\in V_G}\dif x_v$.
	  We then choose a finite reference  measure $\mu$ which is absolutely continuous w.r.t. $m$ and has the $C$ associated with the red lines and $\varphi$ as the density. Then it suffices to check all the green line (means $\tilde{C}_\eps$) is uniform integrable under this finite measure $\mu$ since $\tilde{C}_\eps$ converges to $C$ in measure. To obtain uniform integrability we prove that the integral of the square of all the green lines  w.r.t. $\mu$ is uniformly
bounded, which follows by the same argument as what we proved in Lemma \ref{lem:estimate1} since adding extra $Z^{(i)}$ to $f_v$ at each point will not influence our proof. Thus
$$|\langle F^{k,\eps}_n-\tilde F^k_n,\varphi\rangle|\rightarrow0,\quad \textrm{ as }\eps \to0,$$
which implies $F^k_n=\tilde F^k_n$. Now the result follows.
	\end{proof}
	
We remark that one could also use \cite[Theorem~A.3]{KPZJeremy} to prove the convergence of $F_n^{k,\eps}$ to $\tilde F_n^k$, though it would require some adaptation since
the setting is slightly different e.g.  $C$ would need to be decomposed as a compact supported function plus a smooth remainder.

\section{Short distance behavior of correlations}\label{sec:short}
In this section we use integration by parts formula and similar estimates to derive the short distance behavior of the $2$-point correlation function and $4$-point correlation function.

 To compare $S(x_0-y_0)-C(x_0-y_0)$ we use integration by parts for the last term of \eqref{e:S2}. 
Choosing
$F(\Phi_{M,\eps})= \lambda \cI_\eps(\Wick{\Phi_{M,\eps}^3})(x) $ in \eqref{eq:IBP-C} we obtain
\begin{equ}
	3\lambda\int_{\Lambda_{M,\eps}} C_{M,\eps}(x-z)^2  \E \Big( \Wick{\Phi_{M,\eps}(z)^2}\Big) \dif z
	= \lambda 
	\E \Big( \Phi_{M,\eps}(x) \cI_\eps(\Wick{    \Phi_{M,\eps}^3})(x) \Big) 
	+
	\lambda^2\E\Big((\cI_\eps(\Wick{    \Phi_{M,\eps}^3})(x))^2\Big).
\end{equ}
Substitute this into \eqref{e:S2} we obtain
\begin{equs}[eq:Phi3]
	S^2_{\lambda,M,\eps} & (x_0-y_0)
	 =  C_{M,\eps}(x_0-y_0)
	+\lambda^2 
	\E \Big( \cI_\eps(\Wick{\Phi_{M,\eps}^3})(x_0)\cI_\eps (\Wick{\Phi_{M,\eps}^3})(y_0)\Big)
	\\
	&+9\lambda^2 \int_{\Lambda_{\eps}^2} C_{M,\eps}(y_0-z)C_{M,\eps}(x_0-z)  C_{M,\eps}(z-z_1)^2
	\E \Big(   \Wick{ \Phi_{M,\eps}(z_1)^2} \Big)\,\dif z_1\dif z
	\\&-3\lambda^3\int_{\Lambda_{\eps}}  C_{M,\eps}(y_0-z)C_{M,\eps}(x_0-z) 
	\E \Big( (\cI_\eps( \Wick{\Phi_{M,\eps}^3})(z))^2 \Big) \dif z.
\end{equs}
In the following we give the proof of Theorem \ref{thm:2point}. In the proof we view all the functions on $\Lambda_{M,\eps}$ as periodic functions on $\Lambda_{\eps}$ and we will use Lemma \ref{lem:A2} to bound each term.

\begin{proof}[Proof of Theorem \ref{thm:2point}] In the proof we choose $p_0>18$ for the weight $\rho$ in \eqref{eq:rho} and  use $\rho^\ell$ to denote the weight and $\ell$ may change from line to line. We also abuse the notation of $\rho$ for different spatial dimensions.  We write the last three terms in \eqref{eq:Phi3} as $\sum_{i=1}^3I_i$, which are graphically
\begin{equ}[eq:Phi4g]
\lambda^2
	\begin{tikzpicture}[baseline=-20]
	\node[dot] (x0) at (0,0) {};  \node at (-0.3,0) {$x_0$};
	\node[dot] (z) at (0.5,-1) {};
	\node[dot] (y0) at (2,0) {}; \node at (2.3,0) {$y_0$};
	\node[dot] (w) at (1.5,-1) {};
	\draw[C,bend right=30] (x0) to (z);
	\draw[C,bend left=30] (y0) to (w);
	\draw[Phi] (z) -- ++(0.18,-0.18);\draw[Phi] (z) -- ++(-0.18,-0.18);\draw[Phi] (z) -- ++(0,-0.25);
	\draw[Phi] (w) -- ++(0.18,-0.18);\draw[Phi] (w) -- ++(-0.18,-0.18);\draw[Phi] (w) -- ++(0,-0.25);
	\end{tikzpicture}
	+ 9\lambda^2
	\begin{tikzpicture}[baseline=-20]
	\node[dot] (x0) at (0,0) {}; \node at (-0.3,0) {$x_0$};
	\node[dot] (z) at (1,0) {};
	\node[dot] (y0) at (2,0) {}; \node at (2.3,0) {$y_0$};
	\node[dot] (z1) at (1,-1) {};
	\draw[C,bend right=30] (x0) to (z);
	\draw[C,bend left=30] (y0) to (z);
	\draw[C,bend left=30] (z) to (z1);
	\draw[C,bend right=30] (z) to (z1);
	\draw[Phi] (z1) -- ++(0.18,-0.18);\draw[Phi] (z1) -- ++(-0.18,-0.18);
	\end{tikzpicture}
 -3 \lambda^3
	\begin{tikzpicture}[baseline=-20]
	\node[dot] (x0) at (0,0) {}; \node at (-0.3,0) {$x_0$};
	\node[dot] (z) at (1,0) {};
	\node[dot] (y0) at (2,0) {}; \node at (2.3,0) {$y_0$};
	\node[dot] (z1) at (0.5,-1) {};
	\node[dot] (z2) at (1.5,-1) {};
	\draw[C,bend right=30] (x0) to (z);
	\draw[C,bend left=30] (y0) to (z);
	\draw[C] (z) to (z1);
	\draw[C] (z) to (z2);
	\draw[Phi] (z1) -- ++(0.18,-0.18);\draw[Phi] (z1) -- ++(-0.18,-0.18);\draw[Phi] (z1) -- ++(0,-0.25);
	\draw[Phi] (z2) -- ++(0.18,-0.18);\draw[Phi] (z2) -- ++(-0.18,-0.18);\draw[Phi] (z2) -- ++(0,-0.25);
	\end{tikzpicture}
\end{equ}
	It suffices to calculate the $\bC^{2-\gamma,\eps}(\rho^\ell)$-norm for each graph and $\gamma>0$. By Corollary \ref{co:zmm} and \eqref{eq:sch}  we know for $\gamma>0$ and some $\ell>0$
	$$\|I_1\|_{\bC^{2-\gamma,\eps}(\Lambda_\eps^2,\rho^\ell)}\lesssim\lambda^2\E\|\Wick{\Phi_{M,\eps}^3}\|_{\bC^{-\gamma,\eps}(\rho^\ell)}^2\lesssim \lambda^2+\lambda^{26}.$$
Similarly, by translation invariance,
$$J(z)\eqdef
	\E \Big(  \cI_\eps(\Wick{\Phi_{M,\eps}^3})(z)^2  \Big) $$
is a constant independent of $z$, and we can write $I_3$ as
	$-3\lambda^3 (C_{M,\eps}*C_{M,\eps})(x_0-y_0)J$.
	By Corollary \ref{co:zmm} and \eqref{eq:sch}  we know for $\gamma>0$ and some $\ell>0$
	$$\|J\|_{L^{\infty,\eps}(\rho^n)}\lesssim 1+\lambda^{24}.$$
	Moreover, by \eqref{eq:sch2}, $C_{M,\eps}*C_{M,\eps}(x_0-y_0)\in \bC^{2-\gamma,\eps}(\Lambda_{\eps}^2)$, which  implies
	$$\|I_3\|_{\bC^{2-\gamma,\eps}(\Lambda_\eps^2,\rho^n)}\lesssim \lambda^3+\lambda^{27}.$$
	For $I_2$ 
	we set
	$$J_1\eqdef\int C_{M,\eps}(z-z_2)^2\E\Big( \Wick{\Phi_{M,\eps}(z_2)^2}\Big)\dif  z_2, $$
	which is also a constant by translation invariance.
	Then  we could write $I_2$ as
	$9\lambda^2 (C_{M,\eps}*C_{M,\eps})(x_0-y_0)J_1$
	and by \eqref{eq:sch2} and Corollary \ref{co:zmm} we have for $\gamma>0$
	\begin{align*}
	|J_1|\lesssim\|J_1\|_{L^{\infty,\eps}(\rho^\ell)}\lesssim 
	\E\|\Wick{\Phi_{M,\eps}^2}\|_{\bC^{-\gamma,\eps}(\rho^\ell)}
	\lesssim 1+\lambda^{12}.
	\end{align*}
	Using  $(C_{M,\eps}*C_{M,\eps})(x_0-y_0)\in \bC^{2-\gamma,\eps}(\Lambda_{\eps}^2)$ we deduce
	$$\|I_2\|_{\bC^{2-\gamma,\eps}(\Lambda_\eps^2,\rho^\ell)}\lesssim \lambda^2+\lambda^{14}.$$
	Now the first result follows from Lemma \ref{lem:A1} and the fact that  we can view $S^2_{\lambda,M,\eps}-C_{M,\eps}$ as a function of one variable by translation invariance.
	
	 For the second result, note that by Lemma \ref{euniform} $\|\mathcal{E}^\eps(S^2_{\lambda,M,\eps}-C_{M,\eps})\|_{\bC^{2-\gamma}(\mR^4,\rho^\ell)}$  is uniformly bounded w.r.t. $\eps, M$ with $\cE^\eps$ defined for functions on $\Lambda_{\eps}^2$ and there exists a subsequence converges, which is still denoted by $\eps, M$. We would like to prove the limit of $\mathcal{E}^\eps(S^2_{\lambda,M,\eps}-C_{M,\eps})$ given by $S^{\nu,2}_\lambda-C$.
	 By Lemma \ref{convergencek} we know the limit of
	 $\cE^\eps_2 (S^2_{\lambda,M,\eps}-C_{M,\eps})$ given by $S^{\nu,2}_\lambda-C$. Moreover, by \eqref{einner} we know  for $\varphi\in \mathcal{S}(\mathbb{R}^4)$ with compact support Fourier transform
	 \begin{align*}
	 &\lim_{\eps\rightarrow0,M\rightarrow\infty}\langle \mathcal{E}^\eps(S^2_{\lambda,M,\eps}-C_{M,\eps}), \varphi\rangle=\lim_{\eps\rightarrow0,M\rightarrow\infty}\langle S^2_{\lambda,M,\eps}-C_{M,\eps}, \varphi\rangle_\eps\\=&\lim_{\eps\rightarrow0,M\rightarrow\infty}\langle \cE^\eps_2 (S^2_{\lambda,M,\eps}-C_{M,\eps}), \varphi\rangle=\< S^{\nu,2}_\lambda-C,\varphi\>,
	 \end{align*}
	 which  by lower semicontinuity of the norm implies
	 $$\|S^{\nu,2}_\lambda-C\|_{\bC^{2-\gamma}(\mR^4,\rho^\ell)}\lesssim \lambda^2+\lambda^{27}.$$
	 Since $$\cE^\eps_2(S^2_{\lambda,M,\eps}-C_{M,\eps})(x,y)=\cE^\eps_2(S^2_{\lambda,M,\eps}-C_{M,\eps})(x-y,0),$$
	  the limit $S^{\nu,2}_\lambda-C$ also satisfies the same property and can be viewed as a function on $\mR^2$. Hence,
	   $$\|S^{\nu,2}_\lambda-C\|_{\bC^{2-\gamma}(\mR^2,\rho^\ell)}\lesssim \lambda^2+\lambda^{27}.$$
\end{proof}



In the following we consider  the connected $4$-point function and
 we give the proof of Theorem \ref{th:main3}.

\begin{proof}[Proof of Theorem \ref{th:main3}] In the proof we omit the subscript $M,\eps$ and $\Lambda_{M,\eps}$ for notation's simplicity. We also write $S(x,y)=S^2_{\lambda,M,\eps}(x,y)$. We choose $p_0>36$ for the weight $\rho$ in \eqref{eq:rho} and will use $\rho^\ell$ to denote the weight  and $\ell$ may change from line to line. We introduce the following notation
	$$\cI f(x)\eqdef \int C(x-y)f(y)\dif y,\quad \cJ f(x,x_1)\eqdef \int C(x-z)C(x_1-z)f(z)\dif z,$$
	and
	$$\cJ_1 f(x,x_1,x_2)\eqdef  \int C(x-z)C(x_1-z)C(x_2-z)f(z)\dif z.$$

	The idea of the proof is to use integration by parts formula for $S^4_\lambda(x_1,x_2,x_3,x_4)$ and we will find some  $C(x_i-x_j)$, $i,j\in \{1,2,3,4\}$ appear. However, $C$ is not in the H\"older space $\bC^{2-\gamma,\eps}(\rho^\ell)$, $\gamma>0$. By \eqref{eq:DS1} and Theorem \ref{thm:2point} we find
	\begin{align}\label{eq:Sd}
	S(x_i-x_j)=C(x_i-x_j)-R(x_i-x_j),
	\end{align}
	with $R\in \bC^{2-\gamma,\eps}(\Lambda_\eps^2,\rho^\ell)$. Here by \eqref{eq:DS1}
\begin{align}\label{eqR}R(x_i-x_j)=\lambda \E\Big(\Phi(x_j)\cI(\Wick{\Phi^3})(x_i)\Big).
\end{align}
 Hence, we use $C(x_i-x_j)$ from $S(x_i-x_j)$ to cancel with the corresponding $C(x_i-x_j)$ in the decomposition of $S^4_\lambda$.

To be more precise, for $S^4_{\lambda}=\E\Big[\prod_{i=1}^4\Phi(x_i)\Big]$ we use integration by parts formula \eqref{eq:Phik} with $k=4$ to have
\begin{align*}
S^4_\lambda(x_1,x_2,x_3,x_4)
& =C(x_1-x_2)S(x_3-x_4)+C(x_1-x_3)S(x_2-x_4)+C(x_1-x_4)S(x_2-x_3)
\\&\qquad-\lambda \E\Big(\Phi(x_2)\Phi(x_3)\Phi(x_4)\cI(\Wick{\Phi^3})(x_1)\Big).
\end{align*}
By \eqref{eq:Sd} this implies that
$U^4$ can be written as
\begin{align}
&R(x_1-x_2)C(x_3-x_4)+R(x_1-x_3)C(x_2-x_4)+R(x_1-x_4)C(x_2-x_3)\label{eq:a1}\\&-\lambda \E\Big(\Phi(x_2)\Phi(x_3)\Phi(x_4)\cI(\Wick{\Phi^3})(x_1)\Big)-\bar{U}\label{eq:ii1}
\end{align}
and by Lemma \ref{lem:A1}
 $$\bar{U}(x_1,x_2,x_3,x_4)=R(x_1-x_2)R(x_3-x_4)+R(x_1-x_3)R(x_2-x_4)+R(x_1-x_4)R(x_2-x_3)\in \bC^{2-\gamma,\eps}(\Lambda_\eps^4,\rho^\ell).$$
By Theorem \ref{thm:2point} we have
$$\|\bar{U}\|_{\bC^{2-\gamma,\eps}(\Lambda_\eps^4,\rho^\ell)}\lesssim \lambda^4+\lambda^{54}.$$
Since by \eqref{eq:sch} $\cI$ has smoothing effect, we use IBP for the first term in \eqref{eq:ii1} with the isolated  point $x_4$  to have two terms, which by \eqref{eqR} cancel with the first two terms in \eqref{eq:a1}. As a result, we obtain that
$U^4$ can be written as
\begin{align}\label{ei:1}
&R(x_1-x_4)C(x_2-x_3)+\lambda^2 \E\Big(\Phi(x_2)\Phi(x_3)\cI(\Wick{\Phi^3})(x_1)\cI(\Wick{\Phi^3})(x_4)\Big)\\
&-3\lambda \E\Big(\Phi(x_2)\Phi(x_3)\cJ(\Wick{\Phi^2})(x_1,x_4)\Big)-\bar{U}.\label{ei:2}
\end{align}
We further apply IBP for each term in \eqref{ei:1}-\eqref{ei:2}.
For $R(x_1-x_4)C(x_2-x_3)$ in \eqref{ei:1} we use integration by parts \eqref{eq:DS2} for $R$ with $x_4$ to have
\begin{align}\label{e:i3}R(x_1-x_4)C(x_2-x_3)=&3C(x_2-x_3)\lambda
 \E\Big(\cJ(\Wick{\Phi^2})(x_1,x_4)\Big)\\&-\lambda^2 C(x_2-x_3) \E\Big(\cI(\Wick{\Phi^3})(x_1)\cI(\Wick{\Phi^3})(x_4)\Big)\no,\end{align}
where the second line cancels with one term obtained from the second term   in \eqref{ei:1} by  IBP. In fact using IBP for $x_2$ the second term in \eqref{ei:1} could be written as
\begin{align}
&\lambda^2 C(x_2-x_3) \E\Big(\cI(\Wick{\Phi^3})(x_1)\cI(\Wick{\Phi^3})(x_4)\Big)\no\\&+3
\lambda^2  \E\Big(\Phi(x_3)\cJ(\Wick{\Phi^2})(x_2,x_4)\cI(\Wick{\Phi^3})(x_1)\Big)\label{eq:i1}
\\&+3
\lambda^2 \E\Big(\Phi(x_3)\cJ(\Wick{\Phi^2})(x_1,x_2)\cI(\Wick{\Phi^3})(x_4)\Big)\label{eq:i2}\\&-\lambda^3\E\Big(\Phi(x_3)\cI(\Wick{\Phi^3})(x_2)\cI(\Wick{\Phi^3})(x_1)\cI(\Wick{\Phi^3})(x_4)\Big)\label{eq:zm}.
\end{align}
As mentioned above the first term cancels with the second term for $R(x_1-x_4)C(x_2-x_3)$. By using IBP for the isolated point $x_3$ again, the  term in \eqref{eq:zm} could be written as
\begin{align}
&-3\lambda^3\E\Big(\cJ(\Wick{\Phi^2})(x_2,x_3)\cI(\Wick{\Phi^3})(x_1)\cI(\Wick{\Phi^3})(x_4)\Big)\label{eq:ii2}\\
&-3\lambda^3 \E\Big(\cI(\Wick{\Phi^3})(x_2)\cJ(\Wick{\Phi^2})(x_1,x_3)\cI(\Wick{\Phi^3})(x_4)\Big)\no
\\&-3\lambda^3\E\Big(\cI(\Wick{\Phi^3})(x_2)\cI(\Wick{\Phi^3})(x_1)\cJ(\Wick{\Phi^2})(x_3,x_4)\Big),\no
\end{align}
which share the following same graph
$$
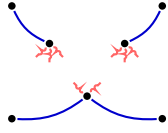
\begin{tikzpicture}
	\node[dot] (x1) at (0,0) {};   	\node[dot] (x2) at (2,0) {};
	\node[dot] (x3) at (0,-1.5) {};   	\node[dot] (x4) at (2,-1.5) {}; 	
	\node[dot] (z1) at (0.5,-0.5) {};	\node[dot] (z2) at (1.5,-0.5) {};
	\node[dot] (z) at (1,-1.2) {};
	\draw[C,bend right=20] (x1) to (z1);
	\draw[C,bend left=20] (x2) to (z2);
	\draw[Phi] (z1) -- ++(0.18,-0.18);\draw[Phi] (z1) -- ++(-0.18,-0.18);\draw[Phi] (z1) -- ++(0,-0.25);
	\draw[Phi] (z2) -- ++(0.18,-0.18);\draw[Phi] (z2) -- ++(-0.18,-0.18);\draw[Phi] (z2) -- ++(0,-0.25);
	\draw[C,bend right=20] (x3) to (z);
	\draw[C,bend left=20] (x4) to (z);
\draw[Phi] (z) -- ++(0.18,0.18);\draw[Phi] (z) -- ++(-0.18,0.18);
	\end{tikzpicture}
$$
and one term like
$\lambda^4\E\Big(\Pi_{i=1}^4\cI(\Wick{\Phi^3})(x_i)\Big).$
	By \eqref{eq:sch} and Corollary \ref{co:zmm}
we know for $\gamma>0$  and $9p<p_0$
\begin{align}\label{eq:cI1}
\E\|\cI(\Wick{\Phi^3})\|_{\bC^{2-\gamma,\eps}(\rho^\ell)}^p\lesssim \E\|\Wick{\Phi^3}\|_{\bC^{-\gamma,\eps}(\rho^\ell)}^p\lesssim1+\lambda^{12p},
\end{align}
and
\begin{align}\label{eq:cJ1}
\E\|\cJ(\Wick{\Phi^2})\|_{\bC^{2-\gamma,\eps}(\rho^\ell)}^p\lesssim \E\|\Wick{\Phi^2}\|_{\bC^{-\gamma/2,\eps}(\rho^\ell)}^p\lesssim1+\lambda^{12p},\end{align}
which by Lemma \ref{lem:A1} 
implies that the $\bC^{2-\gamma,\eps}(\Lambda_\eps^4,\rho^\ell)$-norm for the term in \eqref{eq:zm} is uniformly bounded by $\lambda^3+\lambda^{52}$ up to a multiplicative constant. Here we abuse the notation of $\rho$ for different spatial dimensions. 


Now we consider the first term in \eqref{ei:2} and we apply IBP for the isolated point $x_2$ to write it as
\begin{align}
&-3\lambda C(x_2-x_3)  \E\Big(\cJ(\Wick{\Phi^2})(x_1,x_4)\Big)-6\lambda \E\Big(\Phi(x_3)\cJ_1(\Phi)(x_1,x_2,x_4)\Big)\no
\\&+3\lambda^2 \E\Big(\Phi(x_3)\cI(\Wick{\Phi^3})(x_2)\cJ(\Wick{\Phi^2})(x_1,x_4)\Big)\label{eq:i3}.
\end{align}
We find that the first term cancels with the first term on the RHS of \eqref{e:i3}.
 We apply IBP for the second term for the isolated point $x_3$ and obtain
\begin{align}\label{eq:i5}
&-6\lambda  \int C(x_1-z)C(x_4-z)C(x_2-z)C(x_3-z)\dif z\\&+6\lambda^2\E\Big(\cI(\Wick{\Phi^3})(x_3)\cJ_1(\Phi)(x_1,x_2,x_4)\Big).\label{eq:i4}
\end{align}
We apply IBP for the terms in \eqref{eq:i1}, \eqref{eq:i2} and \eqref{eq:i3} for $x_3$, which have the same graph, and we obtain three terms (see \eqref{eq:zm2} below) similar as \eqref{eq:i4} and the following terms
\begin{align*}
&9
\lambda^2 \E\Big(\cJ(\Wick{\Phi^2})(x_2,x_4)\cJ(\Wick{\Phi^2})(x_1,x_3)\Big)-3\lambda^3\E\Big(\cJ(\Wick{\Phi^2})(x_2,x_4)\cI(\Wick{\Phi^3})(x_1)\cI(\Wick{\Phi^3})(x_3)\Big)
\\&+9
\lambda^2\E\Big(\cJ(\Wick{\Phi^2})(x_1,x_2)\cJ(\Wick{\Phi^2})(x_3,x_4)\Big) -3\lambda^3\E\Big(\cJ(\Wick{\Phi^2})(x_1,x_2)\cI(\Wick{\Phi^3})(x_4)\cI(\Wick{\Phi^3})(x_3)\Big)
\\&+9
\lambda^2 \E\Big(\cJ(\Wick{\Phi^2})(x_1,x_4)\cJ(\Wick{\Phi^2})(x_2,x_3)\Big)-3\lambda^3\E\Big(\cJ(\Wick{\Phi^2})(x_1,x_4)\cI(\Wick{\Phi^3})(x_2)\cI(\Wick{\Phi^3})(x_3)\Big)
,\end{align*}
where they have the following two graphs
$$
\begin{tikzpicture}
	\node[dot] (x1) at (0,0) {};   	\node[dot] (x2) at (2,0) {};
	\node[dot] (x3) at (0,-1.5) {};   	\node[dot] (x4) at (2,-1.5) {}; 	
	\node[dot] (z1) at (1,-0.4) {};
	\node[dot] (z) at (1,-1.1) {};
	\draw[C,bend right=10] (x1) to (z1);
	\draw[C,bend left=10] (x2) to (z1);
	\draw[Phi] (z1) -- ++(0.18,-0.18);\draw[Phi] (z1) -- ++(-0.18,-0.18);
	\draw[C,bend left=10] (x3) to (z);
	\draw[C,bend right=10] (x4) to (z);
\draw[Phi] (z) -- ++(0.18,0.18);\draw[Phi] (z) -- ++(-0.18,0.18);
	\end{tikzpicture}
\qquad
\begin{tikzpicture}
	\node[dot] (x1) at (0,0) {};   	\node[dot] (x2) at (2,0) {};
	\node[dot] (x3) at (0,-1.5) {};   	\node[dot] (x4) at (2,-1.5) {}; 	
	\node[dot] (z1) at (.7,-0.4) {};\node[dot] (z2) at (1.3,-0.4) {};
	\node[dot] (z) at (1,-1.1) {};
	\draw[C,bend right=10] (x1) to (z1);
	\draw[C,bend left=10] (x2) to (z2);
\draw[Phi] (z1) -- ++(0.18,-0.18);\draw[Phi] (z1) -- ++(-0.18,-0.18);
\draw[Phi] (z1) -- ++(0,-0.25);
\draw[Phi] (z2) -- ++(0.18,-0.18);\draw[Phi] (z2) -- ++(-0.18,-0.18);
\draw[Phi] (z2) -- ++(0,-0.25);
	\draw[C,bend left=10] (x3) to (z);
	\draw[C,bend right=10] (x4) to (z);
\draw[Phi] (z) -- ++(0.18,0.18);\draw[Phi] (z) -- ++(-0.18,0.18);
	\end{tikzpicture}
$$
By \eqref{eq:cI1} and \eqref{eq:cJ1} $\bC^{2-\gamma,\eps}(\Lambda_{\eps}^4,\rho^\ell)$-norm of all the above terms are uniformly bounded by $\lambda^2+\lambda^{39}$ up to a multiplicative constant.

Now the only term in $U^4$ not in $\bC^{2-\gamma,\eps}(\rho^\ell)$ is given by the following terms in \eqref{eq:i5} and \eqref{eq:i4} and similar terms as \eqref{eq:i4}:
\begin{align}
&-6\lambda  \int C(x_1-z)C(x_4-z)C(x_2-z)C(x_3-z)\dif z\label{eq:zm4}\\&+6\lambda^2\E\Big(\cI(\Wick{\Phi^3})(x_3)\cJ_1(\Phi)(x_1,x_2,x_4)\Big)+6\lambda^2\E\Big(\cI(\Wick{\Phi^3})(x_1)\cJ_1(\Phi)(x_2,x_3,x_4)\Big)\label{eq:zm2}
\\&+6\lambda^2\E\Big(\cI(\Wick{\Phi^3})(x_2)\cJ_1(\Phi)(x_1,x_3,x_4)\Big)+6\lambda^2\E\Big(\cI(\Wick{\Phi^3})(x_4)\cJ_1(\Phi)(x_1,x_2,x_3)\Big).\no
\end{align}
The  term in \eqref{eq:zm4} belongs to $\bC^{2-\gamma,\eps}(\Lambda_\eps^4,\rho^\ell)$ by \eqref{eq:sch3} and Lemma \ref{lem:A1} and could be bounded by $\lambda$ up to a multiplicative constant.
The last four terms have the same graph. We only consider the term in \eqref{eq:zm2} and the other terms could be estimated similarly. By \eqref{eq:sch3} and Corollary \ref{co:zmm} we know for $\gamma>0$  and $3p<p_0$
$$\E\|\cJ_1(\Phi)\|_{\bC^{2-\gamma,\eps}(\rho^\ell)}^p\lesssim\E\|\Phi\|_{\bC^{-\gamma/2,\eps}(\rho^\ell)}^p\lesssim1+\lambda^{4p},$$
which combined with \eqref{eq:cI1} and Lemma \ref{lem:A1} implies that the $\bC^{2-\gamma,\eps}(\Lambda_\eps^4,\rho^\ell)$-norm of the terms in
\eqref{eq:zm2}  are uniformly bounded. Now the first result follows.

 Using Lemma \ref{euniform} we obtain $\|\mathcal{E}^\eps U^4_{M,\eps}\|_{\bC^{2-\gamma}(\mR^8,\rho^\ell)}$ is uniformly bounded w.r.t. $\eps$, $M$ with $\cE^\eps$ defined for functions on $\Lambda_\eps^4$. We define the limit of a subsequence  as $U^{\nu,4}_\lambda$, which is also the limit of $\mathcal{E}_4^\eps U^4_{M,\eps}$ by similar argument as in the proof of Theorem \ref{thm:2point}.  On the other hand, by  \eqref{eq:sch3} and Lemmas \ref{lem:A1},  \ref{euniform}, $$\cE^\eps\cJ_2(x_1,x_2,x_3,x_4)\eqdef\mathcal{E}^\eps\int_{\Lambda_{M,\eps}} C_{M,\eps}(x_1-z)C_{M,\eps}(x_4-z)C_{M,\eps}(x_2-z)C_{M,\eps}(x_3-z)\dif z$$ is uniform bounded in $\bC^{2-\gamma}(\rho^\ell)$, which implies that there exists a subsequence, still denoted by $\eps, M$, converges. Then by \eqref{einner} for $\varphi(x_1,x_2,x_3,x_4)=\prod_{i=1}^4\varphi_i(x_i)$ and $\varphi_i\in \cS(\mR^2)$ with Fourier transform compact support, it holds
 $$\lim_{M\to\infty,\eps\to0}\<\cE^\eps J_2,\varphi\>=\lim_{M\to\infty,\eps\to0}\<\cE^\eps_4 J_2,\varphi\>.$$
 Then by a  similar argument  as in the proof  in Section \ref{sec:con} for $\cE^\eps_4 J_2$ we replace $C_{M,\eps}(x_i-z)$ by $\E(Z_{M,\eps}^{(i)}(x_i)Z_{M,\eps}^{(i)}(z))$ for $i=2,3,4$ and apply Lemmas \ref{lem:Epro} and \ref{Cepro} to obtain that
 $\cE^\eps_4\cJ_2$ converges to $\int C(x_1-z)C(x_4-z)C(x_2-z)C(x_3-z)\dif z$ when testing positive $\varphi$. The the lower-semicontinuity of the norm gives the result.
%
%
\end{proof}

 \appendix
\renewcommand{\appendixname}{Appendix~\Alph{section}}
\renewcommand{\theequation}{A.\arabic{equation}}

\section{Discrete Besov spaces}\label{App}
In this section we introduce  Besov spaces on the lattice $\Lambda_{\eps}=\eps \mZ^d $ where $\eps=2^{-N}, N\in \mN\cup\{0\},$ from \cite{MP19,GH18a}  and  Besov spaces on $\mR^d$ from \cite{Tri78}. For $f\in \ell^1(\Lambda_{\eps})$ and $g\in L^1(\eps^{-1}\mT^d)$ we define the Fourier and the inverse Fourier transform as
$$\cF f(\xi)=\eps^d\sum_{x\in \Lambda_{\eps}}f(x)e^{-2\pi \iota \xi\cdot x},\quad \cF^{-1} g(x)=\int_{\eps^{-1}\mT^d} g(\xi)e^{2\pi \iota \xi\cdot x}\dif \xi,$$
for $\xi\in \eps^{-1}\mT^d, x\in \Lambda_{\eps}$. We use $\eps=0$  to denote the continuous setting with $\cF_{\mR^d}$ and $\cF_{\mR^d}^{-1}$ being the usual Fourier transform and its inverse on $\mR^d$.
 Let $(\varphi_j)_{j\geq-1}$ be a dyadic partition of unity on $\mathbb{R}^d$. We define the dyadic partition of unity for $x\in \eps^{-1}\mathbb{T}^d$:
  \begin{align}\label{dyadic}\varphi_j^\eps(x)=\begin{cases}\varphi_j(x), \qquad & j<j_\eps,\\
  1-\sum_{j<j_\eps}\varphi_j(x), \qquad & j=j_\eps.
  \end{cases}
  \end{align}
  Here $j_\eps:=\inf\{ j:\textrm{supp} \varphi_j\cap \partial (\eps^{-1}\mathbb{T}^d)\neq \emptyset\}$.

  Now we define the Littlewood-Paley blocks for distributions on $\Lambda_{\eps}$ by
$$\Delta_j^\eps f:=\cF^{-1}(\varphi_j^\eps\cF f),$$
which leads to the definition of weighted Besov spaces. Let $\rho$ denote a polynomial weight of the  form $\rho(x)=(1+|x|^2)^{-\delta/2}$ for some $\delta\geq0$. For $\alpha\in \mR, p, q\in [1,\infty]$ and $\eps\in [0,1]$ we define the weighted Besov spaces on $\Lambda_{\eps}$ given by the norm
$$\|f\|_{B^{\alpha,\eps}_{p,q}(\Lambda_\eps,\rho)}=\Big(\sum_{-1\leq j\leq j_\eps}2^{\alpha jq}\|\Delta_j^\eps f\|_{L^{p,\eps}(\rho)}^q\Big)^{1/q}<\infty.$$
 If $\eps=0$, $B^{\alpha,\eps}_{p,q}(\rho)$ is the classical Besov space
$B^{\alpha}_{p,q}(\rho)$ on $\mR^d$.
Similarly, we  extend the definition of Besov space to functions on $\Lambda_{\eps}^k=\eps\mZ^{dk}$, $k\in \mN$, which is denoted by $B^{\alpha,\eps}_{p,q}(\Lambda_\eps^k,\rho)$.
Whenever there is no confusion we also write $B^{\alpha,\eps}_{p,q}(\Lambda_\eps^k,\rho)$ as $B^{\alpha,\eps}_{p,q}(\rho)$ for simplicity.
We also set $\bC^{\alpha,\eps}(\rho)\eqdef B^{\alpha,\eps}_{\infty,\infty}(\rho)$ and $H^{\alpha,\eps}(\rho)\eqdef B^{\alpha,\eps}_{2,2}(\rho)$.  We also define Besov-H\"older space w.r.t. each component in $\Lambda_{\eps}^k$ for $k\in \mathbb{N}$: for $f:\Lambda_{\eps}^k\to \mR$, $\alpha\in \mR$,
\begin{align}\label{def:cs}\|f\|_{\bC^{\alpha,\eps}_{s}(\Lambda_{\eps}^k,\rho)}\eqdef\sup_{-1\leq j_i\leq j_\eps}2^{\alpha  \sum_{i=1}^kj_i}\Big\|(\prod_{i=1}^k\Delta_{j_i,x_i}^\eps )f\Big\|_{L^{\infty,\eps}(\rho)}<\infty.\end{align}
Here we write $x=(x_1,...,x_k)\in \Lambda_{\eps}^k$ and $\Delta_{j_i,x_i}^\eps$ means  the Littlewood-Paley blocks for the $i$-th component $x_i\in \Lambda_{\eps}$. If $\eps=0$, $\bC^{\alpha}_{s}(\mR^{kd},\rho)$ is the  Besov space w.r.t. each component in $\mR^d$.
   The duality on $\Lambda_\eps$ is given by
 $$\<f,g\>_\eps\eqdef \eps^d\sum_{x\in \Lambda_\eps}f(x)g(x).$$
We also set
$$(f*_\eps g)(x)\eqdef \eps^d\sum_{y\in \Lambda_\eps}f(x-y)g(y).$$
For the polynomial weight $\rho$, by \cite[Lemma A.1]{GH18a} it holds  for $\alpha\in \mR, p,q\in [1,\infty]$ that
\begin{align}\label{eq:eqv}\|f\|_{B^{\alpha,\eps}_{p,q}(\rho)}\asymp \|f\rho\|_{B^{\alpha,\eps}_{p,q}},\end{align}
where the implicit constant is independent of $\eps$.

For a function $f$ on $\eps^{-1}\mathbb{T}^d$ we use $f_{\textrm{ext}}$ to denote its periodic extension to $\mathbb{R}^d$.

\textbf{Extension operator:}

We follow \cite{MP19} to introduce the following extension operator. Recall $\varphi_j^\eps$ given in \eqref{dyadic}.  We choose a symmetric function $\psi\in C_c(\mathbb{R}^d)$ satisfying the following property:
\begin{itemize}
\item $\sum_{k\in\mathbb{Z}^d}\psi(\cdot-k)=1,$
\item $\psi=1$ on supp$\varphi_j$ for $j<j_1,$
\item $(\textrm{supp} \psi\cap \textrm{supp} (\varphi_j^1)_{\textrm{ext}})\backslash \mathbb{T}^d\neq\emptyset \Rightarrow j=j_1.$
\end{itemize}
Here  $(\varphi_j^1)_{\textrm{ext}}$ means periodic extension to $\mathbb{R}^d$ of function $\varphi_j^1$ on $\mathbb{T}^d$. More precisely,
In \cite{MP19} such $\psi$ is called smear function. It is easy to see that $\psi(\eps\cdot)$ satisfies the same property on $\Lambda_\eps$. Set $w^\eps=\mathcal{F}^{-1}_{\mathbb{R}^d}\psi(\eps\cdot)$ and  define
\begin{align}\label{def:E}
\cE^\eps f(x)\eqdef\eps^d\sum_{y\in \Lambda_{\eps}} w^\eps(x-y)f(y), \quad f\in B^{\alpha,\eps}_{p,q}(\rho).\end{align}
We also introduce the following extension operators for functions  $f:\Lambda_{\eps}^k\to \mR$:
\begin{align}\label{def:Ek}
\cE_k^\eps f(x_1,\dots,x_k)\eqdef\eps^{dk}\sum_{\substack{y_i\in \Lambda_{\eps},\\i=1,\dots,k}}\Big(\prod_{i=1}^k w^\eps(x_i-y_i)\Big)f(y_1,\dots,y_k).\end{align}
 It is easy to see that
\begin{align}\label{einner}\langle \cE^\eps f,g\rangle=\langle f,w^\eps*g\rangle_\eps,\quad f,g\in L^{2,\eps}(\Lambda_{\eps}).\end{align}

We have the following elementary result for the Besov-H\"older space.

\bl\label{lem:A1}  For $\eps\in[0,1)$, if $\alpha<0$, then $\bC^{\alpha,\eps}_{s}(\Lambda_{\eps}^k,\rho)\subset \bC^{k\alpha,\eps}(\Lambda_{\eps}^k,\rho)$;
if $\alpha>0$, then $$\|f\|_{\bC^{\alpha,\eps}(\Lambda_{\eps}^k,\rho)}\asymp \sum_{i=1}^k\|f\|_{\bC_{x_i}^{\alpha,\eps}(\Lambda_{\eps}^k,\rho)}.$$
Here $\|f\|_{\bC_{x_i}^{\alpha,\eps}(\Lambda_{\eps}^k,\rho)}:=\sup_{x_j\in \Lambda_{\eps}, j\neq i}\|f(x_1,...,x_{i-1},\cdot,x_{i+1},...,x_k)\|_{\bC^{\alpha,\eps}(\Lambda_{\eps},\rho(x_1,...,x_{i-1},\cdot,x_{i+1},...,x_k))}.$

\el
\begin{proof} 
Let   $\Delta_j^\eps$ and $\Delta_{j_i,x_i}^\eps$ be the Littlewood--Paley blocks in $\Lambda_\eps^k$ and $\Lambda_\eps$, respectively.
By Littelewood--Paley decomposition we have
\begin{align*}
&2^{\alpha kj}\|\Delta_{j}^\eps f\|_{L^{\infty,\eps}(\rho)}
\lesssim\sum_{\substack{-1\leq j_i\leq j_\eps,\\ i=1,\dots , k}}2^{\alpha  kj}\Big\|\Delta_{j}^\eps(\prod_{i=1}^k\Delta_{j_i,x_i}^\eps )f\Big\|_{L^{\infty,\eps}(\rho)}
\\
&\lesssim 2^{\alpha  kj}\sum_{j_i\lesssim j}\Big\|(\prod_{i=1}^k\Delta_{j_i,x_i}^\eps )f\Big\|_{L^{\infty,\eps}(\rho)}
\lesssim 2^{\alpha  kj}\sum_{j_i\lesssim j}2^{-\alpha  \sum_{i=1}^kj_i}\|f\|_{\bC_s^{\alpha,\eps}(\Lambda_\eps^k,\rho)}\lesssim\|f\|_{\bC_s^{\alpha,\eps}(\Lambda_\eps^k,\rho)},
\end{align*}
where in the second inequality we use $\varphi_j(\xi_1,...,\xi_k)\Pi_{i=1}^k\varphi_{j_i}(\xi_i)\neq0$ only if $j_i\lesssim j$. Here and in the following, with a small abuse the notation we write $\varphi_j$ in different dimensions.
Thus the first result follows from the definition.
In the following we prove the second result. By definition we have
 \begin{align*}&\|f\|_{\bC_{x_i}^{\alpha,\eps}(\Lambda_{\eps}^k,\rho)}=\sup_{x_j\in \Lambda_{\eps}, j\neq i}\sup_{l}2^{l\alpha}\|\Delta^\eps_{l,x_i}f(x_1,...,x_{i-1},\cdot,x_i,...,x_k)\|_{L^\infty(\Lambda_{\eps},\rho(x_1,...,x_{i-1},\cdot,x_i,...,x_k))}
 \\&=\sup_{l}2^{l\alpha}\|\Delta^\eps_{l,x_i}f\|_{L^\infty(\Lambda^k_{\eps},\rho)}
\lesssim\sup_{l}2^{l\alpha}\sum_{l\lesssim j}\|\Delta^\eps_{l,x_i}\Delta^\eps_j f\|_{L^\infty(\Lambda^k_{\eps},\rho)}
  \\&\lesssim\sup_{l}2^{l\alpha}\sum_{l\lesssim j}\|\Delta^\eps_j f\|_{L^\infty(\Lambda^k_{\eps},\rho)}\lesssim\sup_{l}2^{l\alpha}\sum_{l\lesssim j}2^{-j\alpha}\| f\|_{\bC^{\alpha,\eps}(\Lambda_{\eps}^k,\rho)}\lesssim\| f\|_{\bC^{\alpha,\eps}(\Lambda_{\eps}^k,\rho)},
 \end{align*}
 where in the first inequality we used $\varphi_j(\xi_1,...,\xi_k)\varphi_{l}(\xi_i)\neq0$ only if $l\lesssim j$ and in the last inequality we use $\alpha>0$. For the converse statement we have
  \begin{align*}&\| f\|_{\bC^{\alpha,\eps}(\Lambda_{\eps}^k,\rho)}=\sup_{j}2^{j\alpha}\|\Delta^\eps_{j}f\|_{L^\infty(\Lambda^k_{\eps},\rho)}
  \lesssim\sup_{j}2^{j\alpha}\sum_{\substack{-1\leq j_i\leq j_\eps,\\ i=1,\dots , k}}\Big\|\Delta_{j}^\eps(\prod_{i=1}^k\Delta_{j_i,x_i}^\eps )f\Big\|_{L^{\infty,\eps}(\Lambda_\eps^k,\rho)}
  \\\lesssim &\sum_i\sup_{j}\sum_{ j_i\backsim j}2^{\alpha  j}\Big\|\Delta_{j}^\eps(\prod_{i=1}^k\Delta_{j_i,x_i}^\eps )f\Big\|_{L^{\infty,\eps}(\Lambda_\eps^k,\rho)}
  \lesssim \sum_i\sup_{j}\sum_{ j_i\backsim j}2^{\alpha  j}\Big\|\Delta_{j_i,x_i}^\eps f\Big\|_{L^{\infty,\eps}(\Lambda_\eps^k,\rho)}\lesssim\sum_{i=1}^k\|f\|_{\bC_{x_i}^{\alpha,\eps}(\Lambda_{\eps}^k,\rho)}
 \end{align*}
where  we used the fact that  $\varphi_j(\xi_1,...,\xi_k)\Pi_{i=1}^k\varphi_{j_i}(\xi_i)\neq0$ only if there exists one $j_i\backsim j$ in the third inequality. Thus the second result follows.
\end{proof}

By similar argument as in the proof of \cite[Lemma A.16]{GH18a} and \cite[Proposition 3.6]{MP19}, we obtain the following estimate. We recall that we view all the functions on $\Lambda_{M,\eps}$ as periodic functions on $\Lambda_{\eps}$.

\bl\label{lem:A2} For $\alpha\in \mR, p,q\in [1,\infty]$, $d=2$
\begin{align}\label{eq:sch}\|(m-\Delta_\eps)^{-1}f\|_{B^{\alpha+2,\eps}_{p,q}(\rho)}\lesssim \|f\|_{B^{\alpha,\eps}_{p,q}(\rho)}.\end{align}
Moreover, set
$$\cJ f(x_1,x_2)=\int_{\Lambda_{M,\eps}} C_{M,\eps}(x_1-z)C_{M,\eps}(x_2-z)f(z)\dif z,$$
and
$$\cJ_1 f(x_1,x_2,x_3)=\int_{\Lambda_{M,\eps}}  C_{M,\eps}(x_1-z)C_{M,\eps}(x_2-z)C_{M,\eps}(x_3-z)f(z)\dif z.$$
If $\alpha<0$, $\gamma>0$, then for $\rho(x_1,x_2)=\rho_1(x_1)\rho_1(x_2)$ with $\rho_1\lesssim1$
\begin{align}\label{eq:sch2}
\| \cJ f\|_{\bC^{\alpha+2-\gamma,\eps}(\rho)}\lesssim \|f\|_{\bC^{\alpha,\eps}(\rho_1)},
\end{align}
and for $\rho(x_1,x_2,x_3)=\rho_1(x_1)\rho_1(x_2)\rho_1(x_3)$  with $\rho_1\lesssim1$
\begin{align}\label{eq:sch3}
\| \cJ_1 f\|_{\bC^{\alpha+2-\gamma,\eps}(\rho)}\lesssim \|f\|_{\bC^{\alpha,\eps}(\rho_1)}.
\end{align}
Here all the proportional constants are independent of $\eps$ and $M$.
\el
\begin{proof}
	 We only give the proof of  \eqref{eq:sch2};  \eqref{eq:sch} and \eqref{eq:sch3} follow in the same way. Since $C_{M,\eps}=\sum_{l\in M\mathbb{Z}^2}C_\eps(\cdot+l),$ it is easy to see that
\begin{equs}
\int_{\Lambda_{M,\eps}} C_{M,\eps}(x_1-z)C_{M,\eps}(x_2-z)f(z)\dif z&= \int_{\Lambda_{\eps}}  C_\eps(x_1-z) C_{M,\eps}(x_2-z)f(z)\dif z\\
&=\int_{\Lambda_{\eps}}  C_{M,\eps}(x_1-z) C_{\eps}(x_2-z)f(z)\dif z,
\end{equs}
with $f$ being the periodic extension from $\Lambda_{M,\eps}$ to $\Lambda_\eps$.
By the definition of Besov spaces it is sufficient to prove for $-1\leq j\leq j_\eps$
\begin{align}\label{pf:1}
\|\Delta_j^\eps \cJ f\|_{L^\infty(\rho)} \lesssim 2^{(-2+\gamma-\alpha)j}\| f\|_{\bC^{\alpha,\eps}(\rho_1)}.
\end{align}

We have for $\xi,\eta\in \eps^{-1}\mathbb{T}^2$
$$\mathcal{F}(\cJ f)(\xi,\eta)=\mathcal{F}C_\eps(\xi)\mathcal{F}C_{M,\eps}(\eta)\mathcal{F}f(\xi+\eta)=\mathcal{F}C_{M,\eps}(\xi)\mathcal{F}C_{\eps}(\eta)\mathcal{F}f(\xi+\eta),$$
which combined with $\sum_{-1\leq l\leq j_\eps}\varphi_l=1$ implies that
\begin{align*}
\Delta^\eps_j\cJ f
&=\sum_{-1\leq k,i,j\leq j_\eps}\mathcal{F}^{-1}\Big(\varphi^\eps_j(\xi,\eta)\varphi^\eps_k(\xi)\varphi^\eps_i(\eta)\mathcal{F}C_\eps(\xi)\mathcal{F}C_{M,\eps}(\eta)\mathcal{F}f(\xi+\eta)\Big)
\\
&=\sum_{k\sim j,i\lesssim j}\mathcal{F}^{-1}\Big(\varphi^\eps_j(\xi,\eta)\varphi^\eps_k(\xi)\varphi^\eps_i(\eta)\mathcal{F}C_\eps(\xi)\mathcal{F}C_{M,\eps}(\eta)\mathcal{F}f(\xi+\eta)\Big)
\\&+\sum_{i\sim j,k\lesssim j}\mathcal{F}^{-1}\Big(\varphi^\eps_j(\xi,\eta)\varphi^\eps_k(\xi)\varphi^\eps_i(\eta)\mathcal{F}C_{M,\eps}(\xi)\mathcal{F}C_{\eps}(\eta)\mathcal{F}f(\xi+\eta)\Big)
\\&=J_1+J_2.
\end{align*}
Here we used the fact that on the support of $\varphi_j^\eps$, $k\sim j$ or $i\sim j$. 
In the following we only consider $J_1$ and the estimate for $J_2$ follows similarly. We have
\begin{align*}
J_1
&=\sum_{-1\leq l\leq j_\eps}\sum_{k\sim j,i\lesssim j}\mathcal{F}^{-1}\Big(\varphi^\eps_j(\xi,\eta)\varphi^\eps_k(\xi)\varphi^\eps_i(\eta)\mathcal{F}C_\eps(\xi)\mathcal{F}C_{M,\eps}(\eta)\varphi^\eps_l(\xi+\eta)\mathcal{F}f(\xi+\eta)\Big)
\\
&=\sum_{ l\lesssim j}\sum_{k\sim j,i\lesssim j}\mathcal{F}^{-1}\Big(\varphi^\eps_j(\xi,\eta)\varphi^\eps_k(\xi)\varphi^\eps_i(\eta)\mathcal{F}C_\eps(\xi)\mathcal{F}C_{M,\eps}(\eta)\varphi^\eps_l(\xi+\eta)\mathcal{F}f(\xi+\eta)\Big)
\\
&=\sum_{ l\lesssim j}\sum_{k\sim j,i\lesssim j}\mathcal{F}^{-1}\Big(\varphi^\eps_j(\xi,\eta)\varphi^\eps_k(\xi)\varphi^\eps_i(\eta)\mathcal{F}C_\eps(\xi)\mathcal{F}C_{M,\eps}(\eta)\Big)*_\eps \mathcal{F}^{-1}\Big(\varphi^\eps_l(\xi+\eta)\mathcal{F}f(\xi+\eta)\Big)
\\
&\eqdef\sum_{ l\lesssim j} I_1*_\eps I_2.
\end{align*}
Here we use the notation $\varphi^\eps_j$ for different spatial dimensions which is a small abuse of notation, and in the second equality we note that on the support of $\varphi^\eps_j$, $|\xi+\eta|\lesssim 2^j$.
Moreover,  by a similar argument as \cite[Proposition 3.6]{MP19} we can use the smear function $\psi$ introduced in the definition of extension operator $\cE^\eps$ and rewrite $I_1$ as the Fourier transform on $\mathbb{R}^4$. More precisely,
\begin{align}I_1=\sum_{k\sim j,i\lesssim j}\mathcal{F}_{\mathbb{R}^4}^{-1}\Big(\psi(\eps \xi,\eps \eta)(\varphi^\eps_j)_{\textrm{ext}}(\xi,\eta)(\varphi^\eps_k)_{\textrm{ext}}(\xi)(\varphi^\eps_i)_{\textrm{ext}}(\eta)\mathcal{F}C_\eps(\xi)\mathcal{F}C_{M,\eps}(\eta)\Big).\end{align}
  Here $(\varphi^\eps_j)_{\textrm{ext}}$ means periodic extension of $\varphi^\eps_j$ from $\eps^{-1}\mathbb{T}^2$ to $\mathbb{R}^2$  and we extend $\mathcal{F}C_{\eps}, \mathcal{F}C_{M,\eps}$ to $\mathbb{R}^2$ by the following formula:
 $$\mathcal{F}C_{\eps}(\xi_1,\xi_2)=\frac{1}{2[m+4(\sin^2(\eps\pi \xi_1)+\sin^2(\eps\pi \xi_2))/\eps^2]}=:\frac{1}{2(m+l^\eps(\xi))},$$
and by Possion summation formula \begin{align}\label{eqcm}\mathcal{F}C_{M,\eps}(\xi)=\frac{1}{M^2}\mathcal{F}C_{\eps}(\xi)\sum_{n\in \frac{1}{M}\mathbb{Z}^2}\delta_{0}(\xi-n),\end{align}
where $\delta_{0}$ is the Dirac measure at $0$.
It is easy to see that $$I_2(x_1,x_2)=\Delta_l^\eps f(x_1)\eps^{-2}\mathbf{1}_{x_1=x_2},, \quad x_1,x_2\in \Lambda_{\eps}.$$
Thus  we use $\rho(x_1)\lesssim\rho(y_1)(1+|x_1-y_1|^\beta)$  for some $\beta>0$ to obtain
\begin{align*}
\|J_1\|_{L^{\infty,\eps}(\rho)}
\lesssim \sup_{y_1,y_2\in\Lambda_\eps} \sum_{ l\lesssim j}\eps^2\sum_{x_1\in \Lambda_\eps}|I_1(y_1-x_1,y_2-x_1)|(1+|y_1-x_1|^\beta)\|\Delta_l^\eps f\|_{L^{\infty,\eps}(\rho_1)}.
\end{align*}
On the other hand since $\alpha<0$
\begin{align*}
\sum_{ l\lesssim j}\|\Delta_l^\eps f\|_{L^{\infty,\eps}(\rho_1)}\lesssim \sum_{ l\lesssim j}2^{-l\alpha}\|f\|_{\bC^{\alpha,\eps}(\rho_1)}\lesssim 2^{-j\alpha}\|f\|_{\bC^{\alpha,\eps}(\rho_1)},
\end{align*}
and
$$I_1(x)=2^{2j}\bar{V}_j(2^jx_1,x_2),$$
with
$$\bar{V}_j=\sum_{k\sim j,i\lesssim j}\mathcal{F}_{\mathbb{R}^4}^{-1}\Big(\psi(\eps 2^j\xi,\eps\eta)(\varphi^\eps_j)_{\textrm{ext}}(2^j\xi,\eta)(\varphi^\eps_k)_{\textrm{ext}}(2^j\xi)(\varphi^\eps_i)_{\textrm{ext}}(\eta)\mathcal{F}C_\eps(2^j\xi)\mathcal{F}C_{\eps,M}(\eta)\Big).
$$
Now it suffices to prove for $N\in \mathbb{N}$
\begin{align}\label{eq:I1}(1+|x_1|^2)^N\bar{V}_j(x)\lesssim 2^{(-2+\gamma)j}.\end{align}
Note that by \eqref{eqcm}
\begin{align*}
	\bar{V}_j(x_1,x_2)=\frac{1}{M^2}\sum_{k\sim j,m\lesssim j}\sum_{n\in \frac{1}{M}\mathbb{Z}^2}\int_{\mathbb{R}^2}&\psi(\eps 2^j\xi,\eps n)(\varphi_j^\eps)_{\mathrm{ext}}(2^j\xi,n)(\varphi^\eps_k)_{\textrm{ext}}(\xi)(\varphi^\eps_m)_{\textrm{ext}}(n)\\&\mathcal{F}C_\eps(2^j\xi)\mathcal{F}C_{\eps}(n) e^{2\pi \iota(\xi x_1+nx_2)}\dif \xi_1.
\end{align*}
By \cite[Lemma 3.5]{MP19} and using $l_\eps(2^j\xi)=2^{2j}l_{\eps 2^j} (\xi)$, $k\sim j$ and $\eps 2^j\lesssim 1$
$$\Big|\partial_{\xi}^N\Big(\psi(\eps 2^j\xi,\eps n)(\varphi^\eps_j)_{\mathrm{ext}}(2^j\xi,n)(\varphi^\eps_k)_{\textrm{ext}}(2^j\xi)\mathcal{F}C_\eps(2^j\xi)\Big)\Big|\lesssim 2^{-2j}1_{|\xi|\lesssim1},$$
which combined with
\begin{align*}
&|(1+|x_1|^2)^N\bar{V}_j(x)|\\\lesssim&\sum_{k\sim j,m\lesssim j}\frac{1}{M^2}\sum_{n\in \frac{1}{M}\mathbb{Z}^2}\Big|\int  e^{2\pi \iota x_1\cdot \xi}(1-\Delta_{\xi})^N\Big(\psi(\eps 2^j\xi,\eps n)(\varphi^\eps_j)_{\textrm{ext}}(2^j\xi,n)\\&\qquad\qquad\qquad\times(\varphi^\eps_k)_{\textrm{ext}}(2^j\xi)\mathcal{F}C_{\eps}(2^j\xi)\Big)\dif \xi\Big|
\mathcal{F}C_{\eps}(n) (\varphi^\eps_m)_{\textrm{ext}}(n),
\end{align*}
 implies that for $\gamma>0$
$$|(1+|x_1|^2)^N\bar{V}_j(x)|\lesssim 2^{-2j}\Big(\frac{1}{M^2}\sum_{n\in \frac{1}{M}\mathbb{Z}^2,|n|\lesssim 2^j}\frac{1}{m+|n|^2}\Big)\lesssim 2^{(-2+\gamma)j}.$$
Thus \eqref{eq:I1} and \eqref{pf:1} holds, which implies \eqref{eq:sch2}.
\end{proof}

 We recall the following results  for discrete Besov embedding, duality, interpolation and all the results hold for $\eps=0$, i.e. the continuous setting.
The following interpolation inequality are used frequently, which is an easy consequence of H\"older's inequality and the corresponding definition.
(see \cite[Lemma A.3]{GH18a} for the proof).
\bl\label{Le32}
Let $\rho$ be a polynomial weight and $\theta\in[0,1]$. Let $\alpha,\alpha_1,\alpha_2\in\mR$ and
$\delta,\delta_1,\delta_2\in \R$ satisfy
$$
\delta=\theta\delta_1+(1-\theta)\delta_2,\ \alpha=\theta \alpha_1+(1-\theta)\alpha_2,
$$
and $p,q,p_1,q_1,p_2,q_2\in[1,\infty]$ satisfy
$$
\tfrac{1}{p}=\tfrac{\theta}{p_1}+\tfrac{1-\theta}{p_2},\ \ \tfrac{1}{q}=\tfrac{\theta}{q_1}+\tfrac{1-\theta}{q_2}.
$$
Then it holds that
\begin{align}\label{DQ1}
\|f\|_{B^{\alpha,\eps}_{p,q}(\rho^\delta)}\leq \|f\|_{B^{\alpha_1,\eps}_{p_1,q_1}(\rho^{\delta_1})}^\theta\|f\|_{B^{\alpha_2,\eps}_{p_2,q_2}(\rho^{\delta_2})}^{1-\theta}.
\end{align}
\el

\bl\label{Le:du}
Let $\rho$ be a polynomial weight. Let $\alpha\in\mR$
and $p_1,q_1,p_2,q_2\in[1,\infty]$ satisfy
$$
1=\tfrac{1}{p_1}+\tfrac{1}{p_2},\ \ 1=\tfrac{1}{q_1}+\tfrac{1}{q_2}.
$$
Then it holds that
$$\<f,g\>_\eps\lesssim \|f\|_{B^{\alpha,\eps}_{p_1,q_1}(\rho)}\|g\|_{B^{-\alpha,\eps}_{p_2,q_2}(\rho^{-1})}.$$
\el
\begin{proof} See \cite[Lemma A.2]{GH18a}.
	\end{proof}

We recall the following Besov embedding theorems (c.f.  \cite[Lemma 2.22]{MP19}):
\vskip.10in
\bl\label{lem:emb}
(i) Let $1\leq p_1\leq p_2\leq\infty$ and $1\leq q_1\leq q_2\leq\infty$, and let $\alpha\in\mathbb{R}$, $\rho_2\lesssim \rho_1$. Then $B^{\alpha_1,\eps}_{p_1,q_1}(\rho_1)$ is continuously embedded in $B^{\alpha_2,\eps}_{p_2,q_2}(\rho_2)$ for $\alpha_2-\tfrac{d}{p_2}\leq \alpha_1-\tfrac{d}{p_1}$.  Furthermore, if $\alpha_2-\tfrac{d}{p_2}< \alpha_1-\tfrac{d}{p_1}$ and $\lim_{|x|\to\infty}\rho_2(x)/\rho_1(x)=0$, the embedding is compact.

(ii) Let  $1\leq p<\infty$, $\eps>0$. Then $B^{0,\varepsilon}_{2,1}(\rho)\subset L^{2,\eps}(\rho)\subset B^{0,\varepsilon}_{2,2}(\rho)$ and
$ B^{0,\varepsilon}_{p,1}(\rho)\subset L^{p,\eps}(\rho)\subset B^{0,\varepsilon}_{p,\infty}(\rho)$.

(iii) Let $\gamma\in (0,1),p\in[1,\infty]$. Then
$$\|f\|_{B^{1-\gamma,\eps}_{p,p}}\lesssim\|\nabla_\eps f\|_{B^{-\gamma,\eps}_{p,p}}+\|f\|_{B^{-\gamma,\eps}_{p,p}}.$$
\el
\begin{proof}
See  \cite[Lemma 2.22]{MP19} for the proof of (i). (ii) and (iii) follow from \cite[Lemma A.4, Lemma A.5]{GH18a}.
	\end{proof}

Recall the following result on the bounds for powers of functions (c.f. \cite[Lemma A.7]{GH18a}, \cite[Lemma 4.2]{MP19}).

\bl\label{lem:pow} Let $\alpha>0$. Let $\rho_1, \rho_2$ be polynomial weights. Then for every $\beta>0$ it holds that
$$\|f^2\|_{B^{\alpha,\eps}_{1,1}(\rho_1\rho_2)}\lesssim \|f\|_{L^{2,\eps}(\rho_1)}\|f\|_{H^{\alpha+2\beta,\eps}(\rho_2)}.$$
For $p\in [1,\infty]$, $\gamma<0<\alpha$ with $\alpha+\gamma>0$ and $\beta>0$ it holds that
$$\|fg\|_{B^{\gamma,\eps}_{p,\infty}(\rho_1\rho_2)}\lesssim\|f\|_{B^{\alpha,\eps}_{p,\infty}(\rho_1)}\|g\|_{\bC^{\gamma+\beta,\eps}(\rho_2)}.$$
\el


Now we prove the following Schauder estimate for discrete heat semigroup $P_t^\eps=e^{t(\Delta_\eps-m)}$:

\bl\label{lem:Sch} Let $\rho$ be a polynomial weight and $\alpha\in\mathbb{R}, p\in[1,\infty],T>0$. Then it holds that for $f\in L^p_TB^{\alpha-2,\eps}_{p,p}(\rho),g\in B^{\alpha-2/p,\eps}_{p,p}(\rho)$
$$\Big\|\int_0^\cdot P_{\cdot-s}^\eps f\dif s\Big\|_{L^p_TB^{\alpha,\eps}_{p,p}(\rho)}\lesssim \|f\|_{L^p_TB^{\alpha-2,\eps}_{p,p}(\rho)},$$
and
$$\| P_{\cdot}^\eps g\|_{L^p_TB^{\alpha,\eps}_{p,p}(\rho)}\lesssim\|g\|_{B^{\alpha-2/p,\eps}_{p,p}(\rho)},$$
where  the proportional constants are independent of $T$ and $\eps$.
\el
\begin{proof} By \cite[Lemma A.16]{GH18a} we have that for $c>0$
	\begin{align*}
	\|\Delta^\eps_j P_{t}^\eps f\|_{L^{p,\eps}(\rho)}\lesssim e^{-ct(2^{2j}+m)}\|\Delta^\eps_j f\|_{L^{p,\eps}(\rho)},
	\end{align*}
	with the proportional constant independent of $\eps$ and $t$\footnote{\cite[Lemma A.16]{GH18a} only proves the result for $p=1$ and the same result also holds for general $p\in [1,\infty]$ by exactly the same argument}, which implies that
		\begin{align*}&\Big\|\int_0^\cdot P_{\cdot -s}^\eps f\Big\|^p_{L^p_TB^{\alpha,\eps}_{p,p}(\rho)} =\sum_j2^{\alpha pj}\int_0^T \Big\|\int_0^t\Delta^\eps_jP_{t-s}^\eps f\dif s\Big\|^p_{L^{p,\eps}(\rho)}\dif t
	\\\lesssim &\sum_j2^{\alpha jp}\int_0^T\Big(\int_0^t\|\Delta^\eps_j P_{t-s}^\eps f\|_{L^{p,\eps}(\rho)}\dif s\Big)^p\dif t
	\\\lesssim &\sum_j2^{\alpha jp}\int_0^T\Big(\int_0^te^{-c(t-s)(2^{2j}+m)}\|\Delta^\eps_j f\|_{L^{p,\eps}(\rho)}\dif s\Big)^p\dif t
	\\\lesssim &\sum_j2^{\alpha jp}\int_0^T\Big(\int_0^te^{-c(t-s)(2^{2j}+m)}\|\Delta^\eps_j f\|^p_{L^{p,\eps}(\rho)}\dif s\Big)\Big(\int_0^te^{-c(t-s)(2^{2j}+m)}\dif s\Big)^{p-1}\dif t
	\\\lesssim &\sum_j2^{\alpha jp}2^{-2j(p-1)}\int_0^T\int_s^Te^{-c(t-s)(2^{2j}+m)}\dif t\|\Delta^\eps_j f\|^p_{L^{p,\eps}(\rho)}\dif s
	\\\lesssim &\sum_j2^{(\alpha-2) jp}\int_0^T\|\Delta^\eps_j f\|^p_{L^{p,\eps}(\rho)}\dif s,
	\end{align*}
	where we change the order of integral in the fourth inequality. The first result then follows form the definition of Besov space.
Similarly we have
\begin{align*}
	\| P_{\cdot }^\eps g\|^p_{L^p_TB^{\alpha,\eps}_{p,p}(\rho)}
	&\leq \sum_{j}2^{\alpha jp}\int_0^T\|\Delta^\eps_j P_{t}^\eps g\|_{L^{p,\eps}(\rho)}^p\dif t
	\\
	&\lesssim \sum_{j}2^{\alpha jp}\int_0^Te^{-ctp(2^{2j}+m)}\|\Delta^\eps_j g\|^p_{L^{p,\eps}(\rho)}\dif t
	\\
	&\lesssim \sum_j2^{jp\alpha-2j}\|\Delta^\eps_j g\|^p_{L^{p,\eps}(\rho)}.
	\end{align*}
	Thus the result follows.
\end{proof}

Now we recall the following property of $\cE^\eps$ in Besov spaces.
\bl\label{euniform} For any $\alpha\in\mathbb{R}, p,q\in[1,\infty]$ the family of extension operators
$$\cE^\eps:B^{\alpha,\eps}_{p,q}(\rho)\rightarrow B^{\alpha}_{p,q}(\rho),\qquad\cE_k^\eps:\bC_s^{\alpha,\eps}(\rho)\rightarrow \bC_s^{\alpha}(\rho),$$
defined in \eqref{def:E} and \eqref{def:Ek} are uniformly bounded in $\eps$.
\el
\begin{proof}
	See \cite[Lemma 2.24]{MP19} for the proof of the  result for $\cE^\eps$. The  result for $\cE^\eps_k$ follows from similar argument.
	\end{proof}

\bl\label{lem:Epro}
For $p\in [1,\infty]$, $\gamma<0<\alpha$ with $\alpha+\gamma>0$ and $\beta>0$ it holds that
$$\|\cE^\eps(fg)-\cE^\eps f\cE^\eps g\|_{B^{\gamma,\eps}_{p,\infty}(\rho_1\rho_2)}\lesssim o(\eps)\|f\|_{B^{\alpha,\eps}_{p,\infty}(\rho_1)}\|g\|_{\bC^{\gamma+\beta,\eps}(\rho_2)}.$$
\el
\begin{proof}
	See \cite[Lemma 4.2]{MP19}.
	\end{proof}

We prove the following result by the same argument as in \cite[Theorem 5.13]{MP19}.

\bl\label{Cepro} Assume that $$(m-\Delta_\eps)u_\eps=f_\eps$$ and $\rho$ be some polynomial weight.
 It holds that for $\alpha\in \mR$ and $\delta>0$
$$\|\cE^\eps u_\eps-(m-\Delta)^{-1}\cE^\eps f_\eps\|_{\bC^{\alpha+2-\delta}(\rho)}\lesssim \eps^\delta \|f_\eps\|_{\bC^{\alpha,\eps}(\rho)}.$$
\el
\begin{proof}It is easy to see that
	\begin{align}\label{eq:m}(m-\Delta_\eps)\cE^\eps u_\eps=\cE^\eps(m-\Delta_\eps)u_\eps=\cE^\eps f_\eps.\end{align}
	Moreover, by \cite[Lemma 3.4]{MP19} and Lemma \ref{euniform}, \eqref{eq:sch} we know for any $\delta>0$
	$$\|(m-\Delta_\eps)\cE^\eps u_\eps-(m-\Delta)\cE^\eps u_\eps\|_{\bC^{\alpha-\delta}(\rho)}\lesssim \eps^\delta \|\cE^\eps u_\eps\|_{\bC^{\alpha+2}(\rho)}\lesssim \eps^\delta \|u_\eps\|_{\bC^{\alpha+2,\eps}(\rho)}\lesssim \eps^\delta \|f_\eps\|_{\bC^{\alpha,\eps}(\rho)},$$
	which implies the result by \eqref{eq:m} and Schauder estimate for $(m-\Delta)^{-1}$.
\end{proof}

\bibliographystyle{alphaabbr}
\bibliography{Reference}

\end{document}